 \documentclass{article}
   \usepackage{amsmath,amsfonts, amssymb}

\usepackage{makeidx}         
\usepackage{graphicx}        
\usepackage[bottom]{footmisc}
\usepackage{color, amsmath, amstext, latexsym}
\usepackage{amsfonts}
\usepackage{amssymb}
\newcommand{\al}{{\alpha}}
\newcommand{\Ga}{{\Gamma}}
\newcommand{\ga}{{\gamma}}
\newcommand{\De}{{\Delta}}

\newcommand{\om}{{\omega}}
\newcommand{\Om}{{\Omega}}
\newcommand{\tht}{{\theta}}
\renewcommand{\th}{\theta}

\newcommand{\si}{{\sigma}}

\newcommand{\la}{\lambda}

\newcommand{\D}{\Delta}
\newcommand{\eps}{\varepsilon}

 \newcommand{\E}{\epsilon}

\newcommand{\vr}{{\varrho}}

\newcommand{\B}{{\mathbb B}}

\newcommand{\R}{{\mathbb R}}
\newcommand{\Q}{{\mathbb Q}}

\newcommand{\sD}{\mathcal D}

\newcommand{\sR}{\mathcal  R}

\newcommand{\Ac}{\mathfrak{A}}

\newcommand{\cA}{{\cal A}}
\newcommand{\cB}{{\cal B}}

\newcommand{\cO}{{\cal O}}
\newcommand{\cF}{{\cal F}}

\newcommand{\cE}{{\cal E}}
\newcommand{\cH}{{\cal H}}

\newcommand{\cR}{{\cal R}}
\newcommand{\cM}{{\cal M}}
\newcommand{\cN}{{\cal N}}

\newcommand{\cL}{{\cal L}}

\newcommand{\cW}{{\cal W}}

\newcommand{\g}{{\nabla}}

\newcommand{\lw}{\tilde{w}}
\def\d{\partial}

\def\t{\tau}

\newcommand{\Dn}{{\operatorname{\frac{\partial}{\partial \textit{n}}}}}
\newcommand{\Dt}{{\operatorname{\frac{\partial}{\partial \tau}}}}

\newcommand{\dist}{{\operatorname{dist}}}

\newcommand{\pd}{\partial}

\newcommand{\hf}{\frac12}

 \newcommand{\di}{{\rm div\,}}

\renewcommand{\eqref}[1]{{\rm (\ref{#1})}}

\newcommand{\rref}[1]{{\rm \ref{#1}}}

\newcommand{\wrt}{{with respect to }}

\newcommand{\Since}{{Because }}

\newtheorem{theorem}{Theorem}[section]
\newtheorem{lemma}[theorem]{Lemma}
\newtheorem{proposition}[theorem]{Proposition}
\newtheorem{corollary}[theorem]{Corollary}
\newtheorem{assumption}[theorem]{Assumption}
\newtheorem{definition}[theorem]{Definition}
\newtheorem{remark}{Remark}[section]
\numberwithin{equation}{section}
\numberwithin{theorem}{section}
\numberwithin{remark}{section}

\newenvironment{declaration}[1]{\trivlist
\item[\hskip \labelsep{\bf #1 }]\ignorespaces}{\endtrivlist}
\newenvironment{proofof}[1]{\begin{declaration}{#1}}{\hfill $\blacksquare$
\end{declaration}}
\newenvironment{proof}{\begin{proofof}{Proof.}}{\end{proofof}}

\numberwithin{equation}{section}
\numberwithin{theorem}{section}

\begin{document}

 \title{Well-posedness and long time behavior in nonlinear dissipative hyperbolic-like evolutions
with critical exponents
}
\author{Igor Chueshov\thanks{Department of Mechanics and Mathematics,
Kharkov National University, Kharkov, 61022, Ukraine; chueshov@univer.kharkov.ua }
\quad and \quad Irena Lasiecka\thanks{
 Department of Mathematics, University of Virginia,
Charlottesville, VA 22901; Systems Research Institute,
Polish Academy of Sciences, Warsaw;
il2v@virginia.edu.
Research partially supported by NSF Grant DDMS-0606682 and AFOSR Grant FA9550-09-1-0459}
}
\maketitle
\begin{abstract}
These lectures present   the analysis of stability and  control  of long time behavior of PDE  models described by nonlinear evolutions of hyperbolic type.
Specific examples  of the models under consideration include:
(i) nonlinear systems of dynamic elasticity: von Karman systems,
Berger's equations,  Kirchhoff  - Boussinesq  equations,
 nonlinear waves  (ii) nonlinear  flow - structure and fluid - structure  interactions, (iii) and nonlinear thermo-elasticity.
A characteristic feature of the models under consideration is  criticality or super-criticality  of sources  (with respect to Sobolev's embeddings) along with super-criticality  of   damping mechanisms which, in addition,  may be also  geometrically  constrained.
\par
Our aim  is to present several methods  relying on    cancelations, harmonic analysis and
geometric analysis,   which  enable to handle criticality and also  super-criticality in both  sources and the damping of the underlined  nonlinear PDE. It turns out that if carefully analyzed the  nonlinearity can be taken  "advantage of"  in order to produce implementable damping mechanism.
\par
Another goal   of  these lectures is the understanding of control mechanisms which are geometrically constrained.
The final task boils down to showing that appropriately damped system is "quasi-stable"  in the sense that any two trajectories approach each other
 exponentially fast up to a compact term which can grow in time.
 Showing this property- formulated as quasi-stability estimate -is the key and technically demanding issue that requires suitable tools.
  These include: weighted energy inequalities,  compensated compactness, Carleman's estimates and some elements of microlocal analysis.
\end{abstract}

\tableofcontents   


\section{Introduction}\label{intro}
These lectures are devoted  to the analysis of stability and  control  of long time behavior of PDE  models described by
nonlinear evolutions of hyperbolic type.
Specific examples  of the models under consideration include:
(i) nonlinear systems of dynamic elasticity: von Karman systems,
Berger's equations,  Kirchhoff  - Boussinesq  equations,
 nonlinear waves  (ii) nonlinear  flow - structure and fluid - structure  interactions, (iii) and nonlinear thermo-elasticity.
 A goal to accomplish is to reduce the  asymptotic behavior of  the dynamics  to  a tractable  finite dimensional and possibly smooth sets.
 This type of results beside having interest in its own within the realm of dynamical systems, are
 fundamental for control theory where  finite dimensional control theory can be used in order to forge a desired  outcome for  the dynamics evolving in the attractor.
\par

A characteristic feature of the models under consideration is  criticality or super-criticality  of sources  (with respect to Sobolev's embeddings) along with super-criticality  of   damping mechanisms which may be also  geometrically  constrained. This  means the actuation takes place on a "small" sub-region only.
Super-criticality of the damping   is often a consequence of the "rough" behavior of nonlinear sources in the equation.  Controlling supercritical potential energy may require a calibrated nonlinear damping that is also supercritical.  On the other hand super-linearity of the potential energy provides beneficial effect  on
the long time boundedness of semigroups.
 From this point of view,  the nonlinearity  does help controlling the system but, at the same time,  it  also does  raise a long list of mathematical issues starting with a fundamental question of uniqueness  and continuous dependence of solutions with respect to the given (finite energy)  data.
It is known that solutions to these problems  can not be handled by standard nonlinear  analysis-PDE techniques.
\par
The aim of these lectures is to present several methods
of nonlinear PDE which include  cancelations, harmonic analysis and
geometric methods which  enable to handle criticality and also  super-criticality in both  sources and the damping. It turns out that if carefully analyzed the  nonlinearity can be taken  "advantage of"  in order to produce implementable
control algorithms.

Another aspects that will be considered is the understanding of control mechanisms which are geometrically constrained.
Here one would like  to use minimal sensing and minimal actuating (geometrically) in order to achieve the prescribed goal.
This is indeed possible, however analytical  methods  used are more subtle. The final task boils down to showing that appropriately damped system is "quasi-stable"
  in the sense that any two trajectory approach each other
 exponentially fast up to modulo a compact term which can grow in time.
 Showing this property- formulated as quasi-stability estimate -is the key and technically demanding issue that requires suitable tools.
  These include: weighted energy inequalities,  compensated compactness, Carleman's estimates and some elements of microlocal analysis.
\smallskip\par
 The lecture are organized as follows.
 \begin{itemize}
 \item
 We start in section 2  with description of  main    PDE models   such as wave and plate equations  with both interior and boundary damping. Instead of striving for the most general formulation, we provide  canonical models in the simplest possible form which however retains the main features of the problems studied.
 \item
 Section 3 deals with wellposedness of weak  solutions corresponding to these models.
 \item
 Section 4 describes general and abstract tools used for proving existence of attractors and also properties of attractors such as structure, dimensionality and smoothness. Here we emphasize methods which allow to deal within non-compact environment typical when one deals with hyperbolic like dynamics and critical sources.
 Specifically a simple but very useful method  of energy relation due to J. Ball
  (see \cite{ball} and also \cite{MRW}) is presented  (allowing to deal with supercritical sources), a version of
 "compensated compactness method" introduced first by A. Khanmamedov \cite{turk-jmaa} (allowing to deal with some  critical sources)
 and a method based on "quasi-stability" estimate originated in the authors work
 \cite{cl-jdde} (and further developed in
  \cite{cl-mem}) which gives in one shot several properties of the attractor such as
 smoothness, finite-dimensionality and also limiting properties for the family of attractors.
 \item
 Section 5 demonstrates how the abstract methods can be applied to  the problems of interest.
 Clearly, the trust of the arguments is in derivation of appropriate inequalities. While we provide the basic insight into the arguments, the details of calculations are often referred  to the literature.
 This way the interested reader will be able to find a  complete justification of the claims made.
 We should also emphasize that  presentation of our results is focused on main features of the dynamics
 and not necessarily  on a full generality. However, subsections such as Generalizations or Extensions provide information on possible generalizations with a well documented literature and citations.
 \item
 Section 6 is devoted to other models such as:  structurally damped plates,  fluid-structures, fluid-flow interactions, thermoelastic interactions,
 Midlin-\-Timoshenko beams and plates, Quantum Zakharov system and Shr\"o\-dinger-\-Boussinesq equations.
 These equations  exemplify models where general abstract tools  for the treatment of long behavior, presented in section 4  apply. Due to space limitations, the analysis here is brief with details deferred to the literature.
 \item
 Each section concludes with examples of further generalizations-exten\-sions and also with a list of open problems.
 \end{itemize}
\par
{\bf Basic notations:}
Let $\Omega \subset \R^d $, $d =2,3$, be a bounded  domain with a smooth boundary
$\Ga=\pd\Om$.
We denote by $H^s(\Om)$ $L_2$-based Sobolev space of the order $s$
endowed with the norm $\|u\|_{s,\Omega} \equiv \|u\|_{H^s(\Omega)}$ and the
scalar product $(u,v)_{s,\Om}$.
As usual for the closure of $C^\infty_0(\Om)$ in $H^s(\Om)$
we use the notation $H^s_0(\Om)$.
Below we  use the notations:
\[(( u,v)) \equiv \int_{\Omega} u v d\Omega , ~~~~~~~||u||^2 \equiv ((u,u)), ~~
<< u, v>> \equiv
\int_{\Ga} uv d\Ga, \]
$$~\quad Q_T \equiv \Om \times (0, T), \quad
\Sigma_T \equiv \Ga \times (0, T).
$$
We also denote by $C(0,T, Y) $  a space  of strongly continuous functions
on the interval $[0,T]$ with values in a Banach space $ Y$. In the case when
we deal with weakly continuous functions we use the notation
 $C_w(0,T, Y)$.

\section{Description of the  PDE models }\label{sect1.2}

We  consider second order PDE models of   evolutions  with   nonlinear
(velocity type) feedback control represented by monotone, continuous functions,  and critical-supercritical  sources   defined on $\Omega \subset \R^d$, $d =2,3$.
In what follows, we  describe models of wave and plate equations with a  dissipation
occurring either in the interior of the domain or on the boundary.
 Three examples given below constitute  canonical models  which admit vast array of generalizations, some of which are  discussed in Section~\ref{sect:other}.
We begin with the  interior damping first.

 \subsection{Waves and plates with nonlinear
 interior  damping \\ and  critical-supercritical sources}
 \label{sect:1.2.1}
 {\bf (A) Wave equation with nonlinear damping-source interaction.}
 \index{wave equation!interior  damping}
  In a bounded domain $\Omega\subset\R^3$ we consider the following wave equation
with the  Dirichlet boundary conditions:
\begin{equation}\label{1.1}
w_{tt} - \Delta w +  a(x) g(w_t) =f(w) ~~\mbox{in}~~\Omega  \times (0, T);
~~~  w=0 ~~\mbox{on}~~ \Gamma \times (0, T),
\end{equation}
where $T >0$ may be finite or $\infty$. We suppose that {\it  the damping }
has the structure
 $ g(s) = g_1 s + |s|^{m-1} s$ for some  $ m \geq 1 $
and assume the following
typical behavior  for  {\it the source}   $f(w) \sim - |w|^{p-1}w$,
where either $1\le p\le 3$ or else $3< p\le \min\left\{5,\frac{6m}{m+1}\right\}$.
The nonnegative function
$a \in C^1(\Omega)$   represents
the support and intensity   of the damping.
The associated  {\it energy function} has the form:

\begin{equation}\label{w-energy}
\cE(t) \equiv \frac{1}{2}   ||w_t(t)||^2 +  \frac{1}{2}  || \nabla w(t) ||^2 - \int_{\Omega}  \hat{f}(w(t)) dx,
\end{equation}
where   $\hat{f} $ denotes the antiderivative of $f$, i.e.  $\hat{f}' =f $.
The {\it energy balance} relation for this model has the form
\begin{equation}\label{w-energy-rel}
 \cE(t) + \int_s^t ((a g(w_t), w_t) )) d\tau = \cE(s).
\end{equation}
 When $p \leq 3 $ the  source in the  wave equation is  up to critical. This is due to Sobolev's embedding $H^1(\Omega) \subset L_6(\Omega) $ so that $f(w) \in L_2(\Omega)$
 for $w$ of finite energy. In this case local (in time)  well-posedness holds  with  no restrictions for $m$. Globality of solutions is guaranteed by  suitable a priori bounds resulting from either a structure of the source $f(w)$  or  from the interplay with the damping, see Theorem \ref{th:wp-wave1} below.
 \par
 Long time behavior will be analyzed under the condition of
  weak degeneracy of the damping coefficient $a(x)$ in the so called "double critical case",
 i.e.\ $ p\leq 3$, $m \leq 5$. We use the compensated compactness  criterion in
 Theorem~\ref{th7-turk} for proving existence of an attractor   and quasi-stability inequality \eqref{8.4.2mc} required by Definition \ref{de:ms-stable}
 in order to establish both smoothness and finite dimensionality of global attractor.
  \par
 When $3< p \leq 5  $ the damping  required for  uniqueness  of solutions  needs to be superlinear.
 With sufficiently superlinear damping we  show that the corresponding  dynamical system is well posed on the finite energy phase space. However, the issue of long time behavior, in that case, is still an open problem.
 The known frameworks  for studying supercritical sources fail due to nonlinearity of the damping. More details will be given later.
The  ranges of $p \in [5,6) $ are considered in \cite{bociu}.
\par
We also refer to \cite[Chapter XIV]{CV02} where the case of a linear damping and
a supercrtical force is considered from point of view of
trajectory attractors.

\medskip\par\noindent
{\bf (B) Von  Karman plate equation with nonlinear interior  damping.}
\index{evolutionary Karman equations with rotational
forces!internal dissipation}
\index{evolutionary Karman equations without rotational
forces!internal dissipation}
Let $\Omega \subset \R^2 $ and $\alpha \in [0,1] $.
We denote $M_{\alpha}  \equiv I - \alpha \Delta$ and consider
the equation
\begin{equation}\label{1.2}
M_{\alpha}  w_{tt}  + \Delta^2 w + a(x)\big[ g(w_t) -\alpha\, {\rm div} G(\nabla w_t)\big]   = [\cF(w),w] + P(w)
\end{equation}
in $\Om\times (0,\infty)$ with
the {\it clamped} boundary conditions:
\begin{equation}\label{Cl}
    w = \Dn w =0 ~~ \mbox{on} ~~\Gamma \times (0,\infty),
\end{equation}
where
the {\it   Airy stress function $\cF(w)$} solves the elliptic problem
\begin{equation}\label{airy}
\Delta^2 \cF(w) = -[w,w] , \mbox{in}~ \Omega, ~~\mbox{with}~~\cF = \Dn \cF =0  ~~
\mbox{on}~~ \Gamma,
 \end{equation}
 and
the von {\it Karman bracket} $[u,v]$  is given by
\begin{equation}\label{bracket}
[u,v] = \partial ^{2}_{x_{1}} u\cdot \partial ^{2}_{x_{2}} v +
\partial ^{2}_{x_{2}} u\cdot \partial ^{2}_{x_{1}} v -
2\cdot \partial ^{2}_{x_{1}x_{2}} u\cdot
\partial ^{2}_{x_{1}x_{2}}v,
\end{equation}
The {\it  damping}  functions $g:\R\mapsto\R_+$ and $G:\R^2\mapsto\R^2_+$ have the following
 form
\begin{equation}\label{dmp-str}
g(s) = g_1 s + |s|^{m-1}  s ~~\mbox{and}~~ G(s,\si) = G_1 \cdot( s;\si) + (|s|^{m-1}s;
\si^{m-1} \si)
\end{equation}
where $g_1$ and $G_1$ are nonnegative constants.
\par
The {\it source} term  $ P $ is assumed  locally Lipschitz operator acting from $H^2(\Omega)$ into $L_2(\Omega) $ when $\alpha =0$
and  from $H^2(\Omega)$ into $H^{-1}(\Omega) $ when $\alpha >0$.
\par
The associated  {\it energy function} and the {\it  energy balance} relation have
 the form
\begin{equation}\label{plate-energy}
\cE(t) \equiv \frac{1}{2}  \left(  ||w_t(t)||^2 +  \alpha ||\nabla w_t(t)||^2 \right)+  \frac{1}{2}  || \Delta w(t) ||^2 +
 \frac{1}{4} || \Delta \cF(w)||^2
 \end{equation}
and
\begin{equation}\label{plate-energy-rel}
\cE(t) + \int_s^tD(w_t)d\tau= \cE(s) + \int_s^t ((P(w), w_t ))d\tau,
 \end{equation}
where the interior damping form $D(w_t)$ is given by
\begin{equation}\label{damp-plate-in}
D(w_t)= (( a g(w_t), w_t))  +  \alpha ((a G(\nabla w_t), \nabla w_t )).
 \end{equation}
\par
 The case when $\alpha > 0 $ is subcritical with respect to the Airy source.
 The well-posedness
 and existence of attractors  in this case is more  standard \cite{cl-book}. When $\alpha =0$ the Airy stress source is critical.
Wellposedness will be achieved by displaying {\it hidden regularity} (see
Lemma~\ref{l:kar-br} below)  for Airy's stress function. Existence of attractors along with their smoothness  will be shown by means  of  quasi-stability inequality
 (see \eqref{8.4.2mc} in Definition \ref{de:ms-stable})
 which   can be proved    as long as $P $ is  subcritical.

\medskip\par\noindent
{\bf (C) Kirchhoff-Boussinesq plate with   interior damping}.
\index{Kirchhoff-Boussinesq model!interior damping}
With notations as in the case of the von Karman plate we consider
\begin{equation}\label{1.3}
M_{\alpha}  w_{tt}  + \Delta^2 w +  a(x)\big[ g(w_t) - \alpha \, {\rm div}\, G(\g w_t)\big]  = {\rm div}\,\big[|\nabla w|^2 \nabla w \big] + P(w)
\end{equation}
with the clamped boundary conditions \eqref{Cl}.
The {\it  damping}  functions $g:\R\mapsto\R_+$,  $G:\R^2\mapsto\R^2_+$
and the {\it source} $ P $
are the same as above.
The associated {\it energy function} has the form
\begin{equation}\label{KB-energy}
\cE(t) \equiv \frac{1}{2}  \left(  ||w_t(t)||^2 +  \alpha ||\nabla w_t(t)||^2 \right)+  \frac{1}{2}  || \Delta w(t) ||^2 +
 \frac{1}{4} \int_{\Omega}  | \nabla w(x,t)|^4 dx
  \end{equation}
The corresponding
{\it energy balance} relation is given by \eqref{plate-energy-rel}.
\par
 The  well-posedness  for the case $\alpha > 0 $ is standard.
 This is due to the fact that ${\rm div}\,|\nabla w|^2 \nabla w \in H^{-1}(\Omega) $ for finite energy solutions $w$.
 The case  $\alpha =0 $  is subtle. Its analysis requires special consideration
  and depends on linearity of the damping. Similarly, long time behavior  dealt with by using Ball's method  presented in an abstract version  in
  Theorem \ref{th:ball}~
 also requires linear damping.
\subsection{
Nonlinear waves and plates with  geometrically  \\ constrained damping
  and critical-supercritical sources}\label{sect:1.2.2}
This class contains  models in the situation when the support of the damping $a(x)$ is strictly contained in $\Omega$. This is to say
supp$\,a(x) \subset \Omega_0 \subset \Omega $. In addition, we  considered singular case
of  interior localized damping which is the  boundary damping. These models are described next.
\medskip\par\noindent
{\bf (A) Wave equation with nonlinear  boundary damping-source interaction.}
 \index{wave equation!boundary  damping}
In a bounded domain $\Omega\subset\R^3$ we consider the following wave equation
\begin{equation}\label{1.4}
 w_{tt} - \Delta w +  a(x) g(w_t) =f(w)~~\mbox{in}~~   \Omega  \times (0, T),  \end{equation}
 with the {\it Neumann} boundary conditions:
\begin{equation}\label{w-bc-n}
\Dn w + g_0(w_t)  = h(w)~~\mbox{in}~~  \Ga  \times (0, T).
\end{equation}
 Concerning the internal damping $g$ and source $f$ terms our hypotheses
 are the same as in Section~\ref{sect:1.2.1}{\bf(A)}. The internal damping coefficient
 $a(x)$ satisfies relations
 \[
 a \in C^1(\Omega),~~ a(x) \geq 0  ~\mbox{in}~ \Omega~~   \mbox{with~~ supp}\,a \subset \Omega_0 \subset \subset  \Omega
 \]
 The {\it boundary damping} has the form $g_0(s) =  g_2 s +|s|^{q-1} s $ with
$q \geq 1$ and $g_2\ge 0$.
The  {\it boundary source} has the behavior  $ h(w) \sim  -|w|^{k-1}w $,
where  $ 1\leq k \leq \max\left\{3, \frac{4q}{q+1}\right\}$.

The associated  {\it energy function} is given by
\begin{equation}\label{w-bc-energy}
 \cE(t) \equiv \frac{1}{2}   ||w_t(t)||^2 +  \frac{1}{2}  || \nabla w(t) ||^2 - \int_{\Omega}  \hat{f}(w(t)) dx
- \int_{\Gamma}  \hat{h}(w(t)) dx
\end{equation}
where as above $\hat{f}$ and $\hat{h} $ denote the antiderivative of $f$ and $h$.
The {\it energy balance} relation in this case has the form
\begin{equation}\label{w-bc-energy-rel}
\cE(t) + \int_s^t \big[(( a g(w_t), w_t)) +  << g_0(w_t) , w_t >>\big] d\tau = \cE(s).
\end{equation}
\par
Local and global well-posedness result will be presented for the boundary- source and damping model.    One could consider larger range of boundary and interior  sources $k \in [1,4) $, $p \in [1,6)$ \cite{bociu}.
In this case, however, potential energy corresponding to the sources is not well defined. Well-posedness result requires higher integrability of initial data \cite{bociu,bociu-JDE}.
More general structure of sources and damping can be considered. However,  the required polynomial bounds are critical.
\par
In the  triple  critical case $\{ p =3, q =3, m=5\}$  with $k=2$ the existence of a global  attractor  is shown by taking advantage of compensated compactness result in Theorem~\ref{th7-turk} (see \cite{snowbird}).
 With  the  additional hypotheses imposed on the damping smoothness finite-dimensionality of attractor is
 established by proving that the system is quasi-stable (in the sense
 of Definition~\ref{de:ms-stable}).
\par
In the case when only boundary damping is active (or internal damping is localized) existence and smoothness of attractor require additional growth condition restrictions imposed on the damping.   This is to say, we need to assume
$ p \leq 3 $  and $ q \leq 1 $ \cite{clt-dcds08,clt-jdde09}.
\medskip
\par\noindent
{\bf (B) Von  Karman plate equation with nonlinear boundary  damping.}
 \index{evolutionary Karman equations with rotational
forces!boundary dissipation, clamped--hinged b.c.}
\index{evolutionary Karman equations without rotational
forces!boundary dissipation, clamped--hinged b.c.}
In a bounded domain $\Omega \subset \R^2 $  we consider the equation
\begin{equation}\label{1.5}
M_{\alpha}  w_{tt}  + \Delta^2 w + a(x)\big[ g(w_t) -
\alpha \, {\rm div}\, G(\nabla w_t)\big]   = [\cF(w),w] + P(w)
\end{equation}
with the
{\it hinged  dissipative}  boundary conditions:
\begin{equation}\label{hing-bc}
w=0,~~~
\Delta w = - g_0(\Dn w_t) , ~~\mbox{on}~~ \Gamma \times (0, T)
\end{equation}
Here we use the same notations as in Section~\ref{sect:1.2.1}{\bf (B)}.
In particular, Airy's stress function $\cF(w)$  solves \eqref{airy}
and the internal damping functions $g$ and $G$ have the form \eqref{dmp-str}.
Concerning the
{\it boundary damping} we assume that $g_0 (s) \sim g_2 s + |s|^{q-1} s$, $q \geq 1 $.
The associated energy function $\cE$ has the same form as
in Section~\ref{sect:1.2.1}{\bf (B)}, see \eqref{plate-energy}.
The corresponding
{\it  energy balance} relation reads as follows
\begin{equation}\label{hing-bc-en-rel}
\cE(t) + \int_s^t \left[ D(w_t) + <<g_0\big(\frac{\pd w_t}{\pd n}\big), \frac{\pd w_t}{\pd n} >>\right] d\tau
 = \cE(s) + \int_s^t ((P(w), w_t ))  d\tau,
\end{equation}
where the internal damping term $D(w_t)$ is given by \eqref{damp-plate-in}.
\par
Below we show
 the existence of a continuous semiflow corresponding to (\ref{1.5}) and \eqref{hing-bc}. The   existence and smoothness of  an  attractor  with boundary damping alone is shown under additional hypotheses restricting the growth of the boundary damping.
   In the  case of boundary damping one could consider  damping affecting other  boundary conditions such as free and  simply supported, see \cite{cl-book}.

\medskip\par\noindent
 {\bf (C) Kirchhoff -Boussinesq plate with boundary damping}.
 \index{Kirchhoff-Boussinesq model!boundary damping}
  With the same
 notations as above in a  domain
$\Omega \subset \R^2 $ we consider the equation
\begin{equation}\label{1.6}
M_{\alpha}  w_{tt}  + \Delta^2 w +  a(x)\big[ g(w_t) - \alpha\, {\rm div}\,  G(\g w_t)\big]  = {\rm div}\big[|\nabla w|^2 \nabla w \big] + P(w)
\end{equation}
with the
{\it hinged  dissipative}  boundary conditions \eqref{hing-bc}.
The internal damping functions $g$ and $G$ satisfy \eqref{dmp-str}.
The boundary damping is the same as the previous case,
i.e.,  $g_0 (s) \sim g_2 s + |s|^{q-1} s$, $q \geq 1$.
The {\it source} $ P $ is  same as above.
The associated  {\it energy function} has the form \eqref{KB-energy}.
The {\it energy balance} relation is given by \eqref{hing-bc-en-rel}.
\par
When $\alpha > 0 $ the wellposedness of finite energy solutions with internal damping  follows from the observation that  the model can be represented as a locally Lipschitz perturbation of monotone operator.
This  approach is  also adaptable to boundary damping, like in the case of von Karman plate.
In the case $\alpha =0$ the problem is much more delicate.  While existence of finite energy solutions can be established with
 general form of the interior-boundary damping, the uniqueness of solutions requires  the linearity of the damping.
 In addition, continuous dependence on the data depends on time reversibility of dynamics which, in turn, does not allow for boundary damping. In view of the above most of the questions asked are open in the presence of {\it boundary damping}.

\section{Well-posedness and generation of continuous flows }
In this section we  present several methods which enable to deal with
Hadamard well-posedness of  PDE  equations with supercritical  Sobolev exponents.
By supercritical, we mean sources that may not belong  to  finite kinetic energy space for trajectories  of finite energy.
 While existence of  finite energy solutions  is often  handled by various variants of Galerkin method, it is the   uniqueness and continuous dependence on the data that is problematic.  There is no unified theory for the treatment of such problems, however there  exist methods that are applicable to classes of these problems.
\par
  In this section we  concentrate on three models described in Section~\ref{sect:1.2.1}, where each of the model displays different characteristics and associated difficulties.  In order to cope with this,  different methods need to be applied. We
   provide a  show case   for the following methods which are critical for proving Hadamard well-posedness  in the examples cited:
\begin{itemize}
\item
Interaction of superlinear sources via the damping. \\Illustration: wave equation.
\item Cancelations-harmonic analysis and microlocal analysis methods.\\
Illustration: Von Karman equations.
\item Sharp  control of Sobolev's embeddings and duality scaling.\\
Illustration: Kirchoff-Boussinesq plate.
\end{itemize}
{\bf 1.} In the first example the equation can be viewed as a {\it perturbation}  of a monotone operator. However,  the superlinearity of the source   destroys {\it locally Lipschitz} character of the semilinear equation.
In order to offset the difficulty, superlinear damping $g(w_t) $ is used.   The interplay between the superlinearity of potential and kinetic energy lies in the heart of the problem.  The presence of this  damping  becomes critical in  establishing well defined dynamical system evolving on a finite energy  phase space.
\\
{\bf 2}. In the second example, classical regularity results for   von Karman
nonlinearity, when $\alpha =0$,
fail to show that the perturbation of monotone operator is locally Lipschitz. However, in this case, more subtle methods based on harmonic analysis and  compensated compactness allow to show that the Airy's stress function
has hidden regularity property. This estimate allows to prove that, contrary to the original prediction, the semilinear term is locally Lipschitz. This allows to prove, again, that the resulting dynamical system is well posed on a finite energy space.
\\
{\bf 3.}
In the third  example, the source of the difficulties is restoring force
 $${\rm div}{}\, \big[|\nabla w |^2 \nabla w \big]
 $$
 which fails to be locally Lipschitz with respect to finite energy solutions in the case $\al=0$.
By using method relying on  logarithmic control of Sobolev's embeddings and topological shift of the energy ("Sedenko's method") we are able to prove that the resulting system with linear damping is well posed on finite energy space.
Critical role in the argument is played by  {\it time reversibility of the dynamics} and {\it linearity of the damping}.
This excludes superlinear damping and boundary dissipation in the case $\al=0$.

\medskip\par\noindent
{\bf Conclusion:} These three canonical examples present three different methods on how to deal with the loss of local Lipschitz property and still be able to obtain well-posed  flows defined  on   finite energy phase space. The methods presented are
transcendental  and applicable to other dynamics displaying similar properties. Some of the examples are given in Section \ref{sect:other}.

\subsection{Wave equation  with a  nonlinear interior  damping - \\
 model in (\ref{1.1})}

\subsubsection{The statements of the results}
 \index{wave equation!interior  damping}
 With reference to the model (\ref{1.1}) with $a(x) \geq a_0 > 0$ in $\Omega $  the following assumption is assumed throughout.
\begin{assumption}\label{as-wave}
\begin{enumerate}
\item A
scalar function $g(s) $ is assumed to be of the form
$g(s) = g_1 s + g_2(s) $, where $g_1 \in \R $ and  $g_2(s) $ is  continuous and
 nondecreasing on $\R$ with $g_2(0) =0$
\item
The source $f(s) $  is \underline{either} represented by a  $C^1$  function and such that
$ |f'(s)|\leq C (1 + |s|^{p-1} )$  with $ 1 \leq p \leq  3 $, \underline{or else}
we have that $f\in C^2 (\R) $ and (i) there exist $ m_{g_2}, M_{g_2}>0$
such that the damping function $g_2$ satisfies the inequality
\begin{equation}\label{g}
 m_{g_2}  |s|^{m+1}  \leq  g_2(s) s  \leq M_{g_2}  |s|^{m+1},~~ |s| \geq 1.
 \end{equation}
 for some $m>1$, (ii)  $|f'' (s) | \leq C (1 + |s|^{p-2} ) $
 with  $p$ satisfying the following compatibility  growth  condition
\begin{equation}\label{gg}
3 \le p < 6~~~\mbox{and}~~~  p  \leq \frac{6m}{m+1}
 \end{equation}
\end{enumerate}
\end{assumption}
While in the case when the source exponent $ p \leq 3 $ we adopt classical semigroup definition of the  generalized
solution (see, e.g., \cite{cl-mem}),
for supercritical cases $ p > 3 $ we provide the following variational definition.
\begin{definition}[Weak solution]\label{def}
Let \eqref{g} and \eqref{gg} be in force.
By a \underline{weak} \underline{solution} of \eqref{1.1}, defined on some
interval $(0,T) $ with  initial data $(w_0;w_1)$  from $H_1(\Om)\times L_2(\Om)$,  we mean a function $u \in C_w(0, T; H^1(\Omega))$, such that
\begin{enumerate}
\item $\displaystyle u_t \in L_{m+1}((0,T)\times\Om)\cap  C_w(0,T; L_2(\Omega) )$.
\item For all $\phi \in C(0,T, H_0^1(\Omega)) \cap C^1(0,T; L_2(\Om)) \cap
L_{m+1}([0,T]\times\Om)$, we have that
\begin{multline}\label{nimic1}
\int_0^T \int_{\Om} (-u_t \phi_t + \nabla u \nabla \phi) \ d\Om dt
+ \int_0^T \int_{\Om} g(u_t) \phi \ d\Om dt
 \\ = -\int_{\Omega}
u_t \phi d \Omega  \Big|_0^T +
\int_0^T \int_{\Om} f(u) \phi \ d\Om dt.
\end{multline}
\item $\displaystyle \lim_{t \to 0} (u(t)-u_0, \phi)_{1,\Omega } = 0$ and
$\displaystyle \lim_{t \to 0}((u_t(t) - u_1,\phi)) = 0$ for all
$\phi \in H_0^1(\Omega)$.
\end{enumerate}
\end{definition}

\begin{theorem}\label{th:wp-wave1}
Let Assumption \ref{as-wave} be in force. We assume  that initial data $(w_0;w_1)$
possess the properties
\[
 w(0) =w_0 \in H^1_0(\Omega),
 ~~~ w_t (0) = w_1  \in  L_2(\Omega)\]
 and for $ p > 5$, $w_0 \in L_r(\Omega)$ with $r =\frac32 (p-1)$.

Then
\begin{itemize}
\item
 There exist a unique, local (in time)   solution $w(t)$ of finite energy.
This  is to say: there exists  $T > 0 $ such that
$$
w \in C([0, T]; H^1_0(\Omega)),~~ w_t \in C([0, T]; L_2(\Omega)),
$$
where  $w$  is a generalized solution when $p \leq 3$  and $w$ is a weak solution  when  condition \eqref{g} holds
with some $m\geq 1$
(which is the case when $p \ge 3$).
\item
    In this case $3<p\le 5$
    weak solutions satisfy the energy relation in \eqref{w-energy-rel}.
\item If $ p \leq 3$ and \eqref{g} holds with some $ m$,
then generalized solution satisfies also the  variational form \eqref{nimic1} with  the test functions $\phi$ from the class
\[
C(0,T; H_0^1(\Omega) ) \cap C^1(0,T; L_2(\Omega) ) \cap L_{\infty} (Q_T).
\]
Moreover, the said solutions  are  continuously dependent  on the initial data.
\item
When $ p \leq m $ the obtained  solutions are global,  i.e., $ T =\infty $.
The same holds under dissipativity condition:
\begin{equation}\label{dis-f-wave}
\liminf_{|s|\rightarrow \infty }\frac{-f(s)}{s} > - \lambda_1
\end{equation}
where $\lambda_1 $ is the first eigenvalue of $ -\Delta $ with the  zero Dirichlet data.
\item
When $ p > 3 $, $ p > m $  and $ f(s) = |s|^{p-1}s$  local  solution has finite time blow-up for negative energy initial data.
\end{itemize}
\end{theorem}

\begin{remark}{\rm
In the case when $p \in [1,3] $ one  constructs solutions  as the limits of strong semigroup solutions \cite{snowbird}.
These are referred  to as {\it generalized solutions}\cite{snowbird,cl-mem}.
  It can be shown that generalized semigroup solutions do satisfy the variational form  in question and variational weak solutions are unique  \cite{cl-mem}.  In order to justify this, it suffices to consider only
  the  nonlinear damping term $g(s)$  with  $g_1 =0$, since the latter provides a bounded linear perturbation-thus it does not affect the limit passage.
  The energy inequality provides $L_1$ bound for $g(w_n)$.
  This  allows to apply Dunford-Pettis compactness criterion (see later)  in order to transit with the weak limit in $L_1$, and thus obtaining variational form of solutions (\ref{nimic1})   with appropriate test function as specified in the third part of Theorem \ref{th:wp-wave1}.
}
\end{remark}

\subsubsection{ Sketch of the proof }
Without loss of generality we may assume $a(x) =1 $ and  $g_1 =0 $, so $ g_2(s) = g(s) $  (these parameters have no impact on the proof).
\medskip\par\noindent
{\bf Case 1:  Critical/subcritical  sources and  arbitrary monotone dam\-ping}.
In this case the proof is standard and follows from monotone operator theory.
Setting the second order equation as a first order system
\begin{equation}\label{W}
 W_t  + A(W) = F(W),~~t>0,~~~ W(0) = (w_0; w_1)
 \end{equation}
on the space $ H \equiv H_0^1(\Omega) \times L_2(\Omega) $
where $W=(w;w_t) $ and
$$
A \equiv \left( \begin{array}{cc} 0 & -I \\
- \Delta &{}~~ g(\cdot) \end{array} \right),\quad
F(W) \equiv \left( \begin{array}{cc} 0\\
 f (w)  \end{array} \right).
$$
 allows to the use of  monotone operator theory (see, e.g.,~\cite{barbu,showalter} for the basic theory).
 Indeed, $A$ is maximally monotone and $F(W)$ is locally Lipschitz on $H$.
 This last statement is due to the restriction $ p \leq 3 $ and Sobolev's embedding $H^1(\Omega) \subset L_6(\Omega)$.
 General theorem (see \cite{cel} and also \cite[Chapter 2]{cl-book}) allows to conclude local unique  existence of semigroup solutions.
 \par
 The arguments leading to variational characterization of weak solutions  are given in the Appendix \cite{snowbird}. It can be based on weak $L_1$
compactness criteria due to Danford and Pettis (see \cite[Section IV.8]{danf}) which states that
the set $M \subset L_1(Q)$ is relatively compact with respect to weak topology
if and only if this set is bounded and uniformly absolutely continuous, i.e.,
\[
\forall\, \eps>0\;~ \exists\, \delta>0\; :  ~~ {\rm mes}\,(E)\le \delta~~
\Rightarrow~~  \sup_{f\in M}\left|\int_E f d Q\right|\le \eps.
\]
We want to show that $M=\{g(w_n)\}$ is weakly $L_1(Q)$ compact for a sequence
of strong solutions which  also converges almost everywhere.
For this using dissipation inequality $\int_Qg(w_n) w_n dQ \leq C$
which (by splitting the integration into $|w_n| \leq R $ and $|w_n |\geq R $)
 implies that
\[
\int_E |g(w_n)|dQ  \leq   \frac{1}{R} \int_Q g(w_n) w_n dQ  +  g_R {\rm mes}\,(E),
~~~\forall\, R>0.
\]
\par\noindent
{\bf Case 2: supercritical source.} In this case superlinearity of the  damping  is exploited.
  The key role is played by energy inequality obtained first for smooth  solutions corresponding to
  truncated problem where $f$ is approximated by locally Lipschitz function  $f_K$ defined in \cite{bociu-JDE}.
  Locally Lipschitz theory (Case~1 above) allows to obtain the energy estimate
  \begin{equation}\label{energy}
  E(t) + \int_0^t (( g(w_t), w_t)) ds \le E(0) +  \int_0^t (( f_K(w), w_t)) ds
  \end{equation}
  with $E(t) \equiv \frac{1}{2} (||w_t(t)||^2 + ||\nabla w(t)||^2 )$.
  Exploiting the  growth conditions  imposed on $g$ and $f$ yields
  $$
   E(t) + \|w_t\|^{m+1}_{L_{m+1}(Q_t) }  \leq C \left[ E(0) +
  \|f_K(w)\|^{\frac{m+1}{m}}_{L_{\frac{m+1}{m}} (Q_t) } \right]
  $$
  and Sobolev's embeddings $ H^1(\Omega) \subset L_6(\Omega) $  along with the condition\\
  $p (m+1) \leq 6 m $ implies
 \[
  E(t) + \|w_t\|^{m+1}_{L_{m+1}(Q_t) }
 \leq  C \left[ E(0) + L \left( t+\int_0^t \|w\|^{\frac{p(m+1)}{m} }_{1, \Omega}ds
 \right)\right].
 \]
The above estimate  followed by  a  rather technical limit  argument
(with $K\to\infty$)
\cite{bociu-JDE}  allows to establish {\bf local existence} of weak  solutions.
\par
{\it Energy identity}. This is an important step of the argument.
By using  finite difference approximation  for the velocity $w_t(t)$  one shows that {\it  weak } solutions satisfy
not only energy {\it inequality} but also  {\it energy identity}. The detailed argument is given in the proof of Lemma 2.3  \cite{bociu}.

{\it Uniqueness} of the solutions is more subtle. We need to  show that any two solutions from the existence class,
i.e. satisfying
$$
 w \in C_w([0, T], H^1_0(\Omega) \cap C^1_w([0, T], L_2(\Omega) ) , ~~w_t \in L_{m+1}(Q_T)
$$
is uniquely determined by the initial data.
To this end let's consider the difference of two solutions:
$ z \equiv  w -u $, where both $w$ and $u$ are finite energy weak  solutions specified as above with the same initial data $(w_0;w_1)$.
Denoting by $$E_z(t) \equiv  || \nabla z ||^2 + ||z_t ||^2 $$ owing  to energy inequality satisfied for weak solutions  and using bounds imposed on the source  function $f(s) $  one obtains  after some calculations:
\begin{equation}\label{energyinequality}
    E_z(t) \leq \int_0^t (( f(w) - f(u) , z_t)) ds \equiv  R_f(u,w,z,t)
\end{equation}
The following lemma is critical:
\begin{lemma}\label{l:1} Let $p<5$. Then
$ \forall \E> 0$   and
$\forall  (w_0;w_1) \in H_0^1(\Omega)\times L_2(\Omega)$
such that
$ \|w_0\|_{1,\Omega} +  \|w_1\|_{0,\Omega}\leq R$,
 there exist a constant $C_{\E}(R,T)>0$  such that
\begin{align*}
|R_f(T)| \leq & {}\; \E E_z(T) + C_{\E}(R,T) \Big[\, T  E_z(T)
\\ &
+     \int_0^T
\left(1+\|u_t(t)\|_{L_{m+1}(\Om)} +
\|v_t(t)\|_{L_{m+1}(\Om)}\right) E_z(t) \ dt \Big].
\end{align*}
\end{lemma}
This lemma  is specialized to the case when $p < 5 $. For $ p\in [5,6) $ there is another term
in the inequality above  which accounts for the fact that the potential energy term may not be in $L_1$.
In order to focus presentation we omit this term and refer the reader to the source \cite{bociu}.

\begin{proof} Using integration by parts we have that
\begin{align*}
 R_f = &(( f(w) - f(u), z ))\Big|_0^t  - \int_0^t
((f'(w) w_t - f'(u) u_t, z)) d\tau
\\
 = &  ((f(w) - f(u), z ))\Big|_0^t  - \int_0^t
((f'(w) z_t + (f'(w) - f'(u)) u_t, z))d\tau
\\
 = & \left[(( f(w) - f(u), z ))-\hf ((f'(w),z^2))\right]_0^t
 \\
 & +\hf \int_0^t
((f''(w)w_t z, z))d\tau
- \int_0^t
(( (f'(w) - f'(u)) u_t, z))d\tau
\\
 \leq & C \int_{\Omega} z^2(t) [1+ |u(t)|^{p-1} + |w(t)|^{p-1}] dx
 \\
& +  C   \int_0^t \int_{\Omega} z^2(t) [1+ |u(t)|^{p-2} + |w(t)|^{p-2}]
 [|u_t| + |w_t| ]dx d\tau.
\end{align*}
It is convenient to use the following elementary  estimate  valid  with any  finite energy element $z$ such that  $z(0) =0$:
\begin{equation}\label{z}
||z(t)||^2 = \int_{\Omega} \Big| \int_0^t z_t(s) ds \Big|^2 dx \leq
t \int_0^t ||z_t(s)||^2ds  \leq 2 t \int_0^t E_z(s) ds.
\end{equation}
 {\sc Estimate for $\int_0^T \int_{\Om} z^2(t) |u_t(t)| \ d x
dt$.}  We use Holder's Inequality with $p=3$ and $q = 3/2$ and obtain:
\begin{equation}\label{termen2}
 \int_0^T \int_{\Om} z^2(t) |u_t(t)| \ dQ \leq
\int_0^T \|z(t)\|^2_{L_6(\Om)} \cdot \|u_t(t)\|_{L_{3/2}(\Om)} \ dt
\end{equation}
This leads to
\begin{equation}\label{term3}
 \int_0^T \int_{\Om} z^2(t) [|w_t(t)| + |u_t(t)| ] \ dQ \leq C(R)
\int_0^T E_z(t)  dt.
\end{equation}
{\sc Estimate for $\int_{\Om}
|u(T)|^{ p-1}  z^2(T) \ d x$}. First, we write
\begin{equation}\label{term4split}
 \int_{\Om} |u(T)|^{ p-1 } z^2(T) \ d x \leq
||z(T)||^2 + \int_{\Om \cap \{|u(T)|> 1\}} |u(T)|^{p-1}
z^2(T) \ d x.
\end{equation}
The first term is estimated in (\ref{z}). The argument for the second term is given below.
We  consider here the supercritical  case $3 < p < 5$.
The corresponding estimate for   $p \in [5,6) $ is given in \cite{bociu}.
\par
 Since $|u(T)| > 1$, there exists $\E_0 > 0$ such that
$|u(T)|^{p-1} \leq |u(T)|^{4-\E_0}$. We choose $\E < \frac{\E_0}{4}$
and apply Holder's inequality with $q = \frac{3}{1+2\E}$ and
$\overline{q} = \frac{3}{2(1-\E)}$. We then use Sobolev's embeddings, the
fact that $(4 - \E_0)\frac{3}{2(1-\E)} \leq 6$, interpolation and Young's inequality
 to obtain
\begin{equation}\label{terme4case1}
\begin{split}
& \int_{\Om \cap \{|u(T)|> 1\}} |u(T)|^{p-1} z^2(T) \ d x\\
& \leq \Big(\int_{\Om} |z(T)|^{\frac{6}{1+2\E}} \ d x
\Big)^{\frac{1+2\E}{3}} \Big(\int_{\Om}
|u(T)|^{\frac{3(4-\E_0) }{2(1-\E)}}\ d x
\Big)^{\frac{2(1-\E)}{3}}\\
& \leq \;
C\|z(T)\|^2_{H^{1-\E}}\|u(T)\|^{(4-\E_0)}_{L_{\frac{3(4-\E_0)}{2(1-\E)}}(\Om)}
 \leq \; \E E_z(T) + C_{\E}(R) ||z(T)||^2,
\end{split}
\end{equation}
Combining (\ref{term4split}) with (\ref{z}) and
(\ref{terme4case1}), we obtain the final estimate in this case:
\begin{equation}\label{term4case1}
\displaystyle \int_{\Om} |u(T)|^{p-1} z^2(T) \ d x \leq \E E_z(T)
+ C_{\E}(R) T \int_0^T E_z(t) \ dt
\end{equation}
{\sc Estimate for $\int_0^T\int_{\Om}
|u(t)|^{p-2} |u_t(t)| z^2(t) \ dQ$.}
Since for $|u(t)| \leq 1$, we obtain the
term estimated in the previous case,
it is sufficient to look at
\[
\int_0^T\int_{\tilde\Om} |u(t)|^{p-2}
|u_t(t)| z^2(t) \ dQ~~\mbox{with}~~\tilde \Om = \Om \cap \{|u(t)| >
1\}.
\]
 We start by applying Holder's inequality with $q = 3$ and $
\overline{q} = 3/2$:
\begin{equation}\label{termen6}
\begin{split}
\int_0^T & \int_{\tilde \Om} |u(t)|^{p-2} |u_t(t)|z^2(t) \leq
\int_0^T \|z(t)\|^2_{L_6(\Om)} \Big[\int_{\Om}
|u(t)|^{\frac{3(p-2)}{2}} |u_t(t)|^{\frac{3}{2}} \ d x
\Big]^{\frac{2}{3}}
\end{split}
\end{equation}
We use Holder's inequality again, with $ q = \frac{4}{p-2}$ and $
\overline{q} = \frac{4}{6-p}$, and notice that $\frac{6}{6-p} \leq
m+1 \Leftrightarrow p \leq \frac{6m}{m+1}$. Hence (\ref{termen6})
becomes:
\begin{equation}\label{terme6}
\begin{split}
\int_0^T\int_{\tilde \Om} |u(t)|^{p-2}
|u_t(t)| z^2(t)  & \leq C\int_0^T  E_z(t)
\|u(t)\|^{p-2}_{L_6(\Om)} \|u_t(t)\|_{L_{\frac{6}{6-p}}(\Om)}
\\ & \leq C \int_0^T  E_z(t) \|u(t)\|^{p-2}_{L_6(\Om)}
\|u_t(t)\|_{L_{m+1}(\Om)}.
\end{split}
\end{equation}
Since  $\|u(t)\|_{1,\Omega}
\leq C_{R,T}$  we obtain:
\begin{equation}\label{term6}
\displaystyle \int_0^T\int_{\tilde \Om} |u(t)|^{p-2}
|u_t(t)| z^2(t) \ dQ \leq C_{R,T} \int_0^T E_z(t)
\|u_t(t)\|_{L_{m+1}(\Om)} \ dt.
\end{equation}
Similarly, following the strategy used in the previous step, we obtain the estimate
\begin{equation*}
\int_0^T\int_{\tilde\Om_u } |u(t)|^{p-2} |w_t(t)| z^2(t)
\ dQ
\leq C_{R,T} \int_0^T E_z(t) \|w_t(t)\|_{L_{m+1}(\Om)}\ dt,
\end{equation*}
\begin{equation*}
 \int_0^T\int_{\tilde\Om_w} |w(t)|^{p-2} |u_t(t)|z^2(t)
\ dQ
\leq  C_{R,T} \int_0^T E_z(t) |u_t(t)|_{L_{m+1}(\Om)} \ dt
\end{equation*}
with $\tilde\Om_w=\Om \cap \{|w(t)| > 1\}$. This completes
the proof of Lemma~\ref{l:1}.
\end{proof}
{\sc Completion of the proof}:
In (\ref{energyinequality}), we use Lemma \ref{l:1} to
obtain:
\begin{equation}\label{Gronwalltypeineq}
\begin{split}
E_z(T) \leq & \; \E E_z(T) + C_{\E}(R,T) \Big[\, T  E_z(T)
\\ &
+     \int_0^T
\left(1+\|u_t(t)\|_{L_{m+1}(\Om)} +
\|v_t(t)\|_{L_{m+1}(\Om)}\right) E_z(t) \ dt \Big]
\end{split}
\end{equation}
for every $\eps>0$.
In (\ref{Gronwalltypeineq}), we choose $\E$ and $T$ such
that $ \E  + C_{\E}(R,T) T  < 1/2$
(for uniqueness it is enough to look at a small interval for $T$, since the process
can be reiterated) and apply Gronwall's inequality  to obtain that $E_z(T) = 0$
for all $T \leq T_{max}$, where $[0,T_{max}]$ is a common  existence interval
  for both solutions $w$ and $u$.

The above argument completes the proof of uniqueness of finite energy solutions.
This result along with energy {\it equality} leads to Hadamard wellposedness.
Details are in \cite{bociu,bociu-JDE}.\\
{\it Finite-time blow up } of solutions is established in \cite{bociu-1}.  This is accomplished
by showing that appropriately constructed anti-Lyapunov function \cite{serrin,vitillaro-2}  blows up in a finite time.

\subsubsection{Generalizations-Extensions}

\begin{itemize}
\item
In the subcritical and critical case additional regularity of weak solutions equipped with more regular initial data  can be established. Quantitative statements  are in \cite{snowbird}.
\item
The wave model discussed is equipped with zero Dirichlet data.
However, the  same result holds for {\it Neumann } or {\it Robin } boundary conditions.
\item
In the subcritical and critical case, $p \leq 3 $, the same results hold with  localized damping.
This is to say when supp${}\,a(x) \subset \Omega_0 \subset \Omega $.
\item
The analysis of wellposedness for  the same model in the case when solutions are confined to a potential well
are  given in  \cite{bociu-r}.
 \end{itemize}
{\bf Open question:}
Obtain the same result with partially localized damping supp${}\, a(x) \subset \Omega_0 \subset \Omega$
in the supercritical case  $p > 3 $.

\subsection{Von Karman equation with interior damping - model (\ref{1.2}) }

\subsubsection{The statement of the results}

With reference to the model (\ref{1.2}) the following assumption is assumed throughout.
\begin{assumption}\label{as-karman}
\begin{enumerate}
\item
Scalar function $g(s)$ is assumed to be continuous and monotone on $\R$
with $g(0) =0$.
\item In the case $\al>0$ we assume that the damping $G$ has the form
$G(s,\si)=(g_1(s); g_2(\si))$, where $g_i(s)$, $i=1,2$, are
continuous  and monotone on $\R$ with $g_i(0) =0$.
Moreover, they are of polynomial growth, i.e.,
 $ |g_i(s)|\leq C (1 + |s|^{q-1})$  for some $q\ge 1$.
\item
The source $P(w)$ is assumed locally Lipschitz from $H^2(\Omega)$ into
$[\cH_\al(\Omega)]'$,
where
\begin{equation}\label{h-alpha}
   \cH_\al(\Om)=\left\{
                  \begin{array}{ll}
                    L_2(\Om) & \hbox{in the case $\al=0$;} \\
                   H^1_0(\Om) & \hbox{in the case $\al>0$.}
                  \end{array}
                \right.
\end{equation}
\end{enumerate}
\end{assumption}
Below we also use the following refinement of the hypothesis
concerning the source $P$.
\begin{assumption}\label{as:karman-P}
 We assume that   $P(w)=-P_0(w)+P_1(w)$, where
 \begin{enumerate}
   \item $P_0(w)$ is a Fr\'echet  derivative of the functional $\Pi_0(w)$
and the  property holds: there exist $a,b \in \R$ and $\epsilon >0$
such that
\[
 \Pi_0(w)  +   a\|w\|_{2-\epsilon,\Omega} ^2 \geq b,~~  ((P_0(w)  , w))  +   a\|w\|_{2-\epsilon,\Omega} ^2 \geq b,~~\forall w\in H_0^2(\Om);
 \]
   \item there exists $K>0$ and $\eps>0$ such that
$\|P_1(w)\|\le K \|w\|_{2-\epsilon,\Omega}$
for all  $w\in H_0^2(\Om)$.
 \end{enumerate}
\end{assumption}
\begin{remark}\label{re:F0}{\rm
A specific choice of interest in applications is $P(w) = [F_0, w] -p$,
 where $[\cdot,\cdot]$ denotes the von Karman bracket (see \eqref{bracket}),
$ F_0 \in H^{3+\delta}(\Omega) \cap H_0^1(\Omega)$, $\delta>0$, and
$p\in L_2(\Om)$. The term $F_0$ models
in-plane forces in the plate and $p$ is a transversal force.
 The given source  complies with all the hypotheses  stated.
In the case $\al>0$ the hypotheses concerning $P(w)$, $F_0$ and $p$
can be relaxed. However we do not pursue this generalizations
and refer to \cite{cl-book}.
}
\end{remark}
In addition to the concept of {\it generalized } (semigroup) solutions \cite{cl-mem,cl-book} we can also define {\it weak} solutions.
\begin{definition}[Weak solution]
By a \underline{weak solution} of (\ref{1.2}), defined on an
interval $[0,T]$, with initial data $(u_0;u_1)$  we mean a function
\[
u \in C_w(0, T; H_0^2(\Omega)), ~~~u_t
\in C_w(0,T; \cH_\al(\Omega) )
\]
 such that
\begin{enumerate}
\item For all $\phi \in  H_0^2(\Omega)$,
\begin{align}
((u_t(t), \phi))& +\alpha ((\g u_t(t), \g \phi)) +
 \int_0^t  (( \Delta u, \Delta  \phi)) \  dt \\ &
+  \int_0^t\big[ ((g(u_t), \phi)) + \alpha ((G(\g u_t(t)), \g \phi))  \big] \  dt
 \nonumber
\\ \label{nimic2} =& ((u_1, \phi)) +\alpha ((\g u_1, \g \phi))  +
\int_0^t ((P(u) + [\cF(u),u], \phi)) \  dt.\nonumber
\end{align}
\item $\displaystyle \lim_{t \to 0} (u(t)-u_0, \phi)_{2,\Omega } = 0$.
\end{enumerate}
Here as above $C_w(0,T, Y) $ denotes a space  of weakly continuous functions
with values in a Banach space $ Y$.
\end{definition}
\index{evolutionary Karman equations with rotational
forces!internal dissipation!well-posedness}
\index{evolutionary Karman equations without rotational
forces!internal dissipation!well-posedness}
\begin{theorem}\label{th:wp-karman1}
Under the assumption \ref{as-karman} for all initial data
$$
w(0) =w_0 \in H^2_0(\Omega),~~ w_t (0) = w_1  \in  \cH_\al(\Omega)
$$
(with $\cH_\al(\Omega)$ defined by \eqref{h-alpha})
there exist a unique, local (in time) generalized (semigroup)
solution of finite energy, i.e.,
 there exists  $T > 0$ such that
$$
w \in C([0, T]; H^2_0(\Omega)), w_t \in C([0, T]; \cH_\al(\Omega)).
$$
Moreover,
\begin{itemize}
\item
If, in addition in the case $\al=0$ the damping $g(s)$ is of some polynomial growth, generalized solution becomes weak solution.
 A weak  solution is  also continuously dependent  of the initial data.
\item
The  solutions are global provided Assumption~\ref{as:karman-P} holds.
In this case a bound for the energy of solution is
independent of time horizon.
\end{itemize}
\end{theorem}

 \subsubsection{Sketch of  the proof}
We concentrate on a more challenging   case when $\alpha =0 $.
The case $\alpha > 0 $ can be found in \cite[Chapter 3]{cl-book}.
\par
For the proof of this result the support of the damping described by $a(x) $ plays no role. Thus, without loss of generality we may assume that $a(x) =1 $.
\par
 The following two lemmas describing properties of von Karman brackets  can be found in
 \cite{cl-book}. The first Lemma is   critical for the proof of well-posedness and the second Lemma implies   boundedness of solutions that is independent on time horizon.

\begin{lemma}\label{l:kar-br}
 Let $  \Delta^{-2} $ denotes the  map  defined by
 $z \equiv  \Delta^{-2} f $
  iff
  $$ \Delta^2 z =f ~\mbox{in} ~\Omega ~\mbox{and}
  ~ f = \Dn f =0 ~\mbox{on}~~ \Gamma.
  $$
 Then
 \[
 \|\Delta^{-2} [u,v] \|_{W^{2, \infty} (\Omega) } \leq C \|u\|_{2, \Omega} \|v\|_{2, \Omega},
 \]
 where the von Karman bracket $[u,v]$ is defined by \eqref{bracket}. In particular, for the Airy stress function $\cF$ we have the estimate
 $\| \cF(w) \|_{W^{2, \infty} (\Omega) } \leq C \|w\|^2_{2, \Omega}$.
 \end{lemma}
 Notice that  standard result \cite{lions} gives
 $$
 \|\Delta^{-2} [u,v] \|_{3-\epsilon,\Omega } \leq C \|u\|_{2, \Omega} \|v\|_{2, \Omega} $$
  which does not imply the result stated in the lemma.

The second lemma \cite{cl-book}  has to do with a control of low frequencies by nonlinear term.
\begin{lemma}\label{l:2}
Let $u\in H^2(\Omega)\cap H_0^1(\Omega)  $. Then for every $\epsilon > 0 $ there exists $M_{\epsilon} > 0 $  such that
$$ || u||^2 \leq \epsilon \left( ||\Delta u||^2  + ||\Delta \cF (u) ||^2\right) + M_{\epsilon} $$
 \end{lemma}
Equipped with these two lemma, the proof of global  well-posedness
follows standard by now procedure:
\medskip\par\noindent
 { \bf Step 1:}
 Establish maximal monotonicity
of the operator
$$A \equiv \left( \begin{array}{cc} 0 & - I \\
\Delta^2 &{}\; g(\cdot) \end{array} \right)  $$
on the space $ H \equiv H_0^2(\Omega) \times L_2(\Omega) $
with
$$
D(A) = \{ (u;v) \in H_0^2(\Omega) \times L_2(\Omega)\, :~~  \Delta^2 u + g(v) \in L_2(\Omega) \} $$
This follows from a standard argument in monotone operator theory \cite{barbu,showalter}.
\medskip\par\noindent
 { \bf Step 2:}  Consider nonlinear term as a perturbation of $A$:
$$F(W) \equiv \left( \begin{array}{c} 0  \\
 {}[\cF(w), w] + P(w)   \end{array} \right)  $$
\medskip\par\noindent
 { \bf Step 3:}  Show that $F(W) $ is locally Lipschitz on $H$.
 This can be done with the help of Lemma  \ref{l:kar-br}
 which implies the following estimate
 \begin{equation}\label{loc-lip}
 || [\cF (u) , u ] - [\cF(w), w] || \leq
 C  ( 1 + \|u\|^2_{2, \Omega} + \|w\|^2_{2, \Omega} ) \|u-w\|_{2, \Omega}
 \end{equation}
\smallskip\par
The above steps lead  to  well-defined  semigroup solutions, defined locally in time,  which are unique and satisfy local Hadamard well-posedness.
\par
 The final step is global well-posedness for which a priori bounds are handy.
 The needed  a priori bound results  from the property  of von
Karman bracket described in Lemma \ref{l:2}.
 Indeed, in order to claim global well-posedness, it suffices to   notice that
 the relation
 $$ (( [\cF(w),w]+ P(w), w_t )) =  - \frac{d}{dt}\left[\hf ||\Delta \cF(w)||^2
 +\Pi_0(w)\right] + ((P_1(w),w_t))
 $$
 implies the  a priori bounds for the energy function $\cE(t)$  given by \eqref{plate-energy} with $\al=0$ via energy inequality and Gronwall's lemma.
\medskip\par\noindent
{\bf Stationary solutions.}
For description of the structure of the global attractor we also need
consider stationary solutions to problem \eqref{1.2}.
We define a stationary  weak solution  as a function $w\in H_0^2(\Om)$
satisfying in variational sense the following equations
\begin{equation}\label{1.2stat}
\Delta^2 w    = [\cF(w),w] + P(w)~~\mbox{in} ~~\Om,
~~~w = \Dn w =0 ~~ \mbox{on} ~~\Gamma,
\end{equation}
where
the  Airy stress function $\cF(w)$ solves the elliptic problem \eqref{airy}.
\begin{proposition}\label{pr:stationary}
 Let Assumption~\ref{as:karman-P} be in force. In addition assume that
$P$ is weakly continuous mapping from $H^2_0(\Om)$ into   $H^{-2}(\Om)$.
Then  problem \eqref{1.2stat} has a solution. Moreover, the set $\cN_*$ of all
stationary solutions is compact in $H^2_0(\Om)$.
\end{proposition}
For the proof of this proposition we refer to \cite{CR80},
see also \cite{cl-book,ciarlet1}. We note problem  \eqref{1.2stat}
may have several solutions \cite{CR80}. However in a {\it generic}
situation the set $\cN_*$ is finite (see a discussion in \cite[Chapter 1]{cl-book}).

   \subsubsection{Generalizations-Extensions}
   \begin{enumerate}
   \item
   By assuming more regular initial data one obtains regular solutions.
   Precise quantitative statement of this is given in \cite{cl-book}. Regular solutions are global  in time.
   \item
  Related   results hold when $\alpha > 0 $. The analysis here is simpler and there is no need for sharp results in Lemma \ref{l:kar-br}, see \cite{cl-book}.
  \item
  One can consider other boundary conditions such as simply supported or free or combination thereof, \cite{cl-book}.
  \item
  Damping can be partially supported in $\Omega$.  This has no effect on the arguments.
  \item
  A related  wellposedness result holds for non-conservative models with additional energy level terms that are non-conservative (for instance $\nabla w \cdot \Psi $ for some smooth  vector $\Psi$).
  \end{enumerate}
  \subsubsection{Open question}
Full von Karman systems that consist of system of $2D$ elasticity coupled with von Karman equations.
This model accounts for  in-plane accelerations. As such, there is no decomposition using Airy's stress function.
The nonlinearities entering are super-critical (even  in the rotational case when  $\alpha > 0$) and hidden regularity of Airy's stress function plays no longer any role.
For such model  with $\alpha > 0$ one can still prove existence and uniqueness of solutions, by appealing to "Sedenko's method" (see \cite{Sed91b}).
Even more, full Hadamard well-posedness and energy identity can also  be proved in that case \cite{sicon,koch}.
However, the problem is entirely open in the non-rotational case $\alpha =0$.
In this latter case
only existence of weak solutions, obtained by Galerkin method,  is known.

  \subsection{Kirchhoff-Boussinesq model with interior damping - model in (\ref{1.3})}
  \index{Kirchhoff-Boussinesq model!interior damping}
As before   we  concentrate on the most demanding  case when $\alpha =0$.
For the case $\al>0$ we refer to \cite{cl-mem}.
  The main challenge of this model is the presence of restorative force term
  ${\rm div}\, \big[|\nabla w|^2 \nabla w\big]$
  which is not in $L_2(\Omega) $ for finite energy solutions $w$.
  This is due to the failure of Sobolev's embedding $H^1\subset L_{\infty} $ in two dimensions.
  We also assume, without loss of generality for the well-posedness,
  that $a(x) \equiv 1$.

  \subsubsection{The statement of the results}
With reference to the model (\ref{1.3}) with $\al=0$ the following assumption is assumed throughout.
\begin{assumption}\label{as-BK}
\begin{enumerate}
\item
$g(s)$ is  continuous and monotone on $\R$ with $g(0) =0$.
In addition $g(s) $ is of polynomial growth at infinity and also $a(x) \equiv 1$ (without loss of generality for the well-posedness).
\item
The source $P(w) $ is assumed locally Lipschitz from $H^2(\Omega)$ into
 $L_2(\Omega)$.
\end{enumerate}
\end{assumption}
 Below we deal with {\it weak} solutions.
\begin{definition}[Weak solution]
By a \underline{weak solution} of (\ref{1.3}),
with $\al=0$ and $a(x) \equiv 1$, defined on some
interval $(0,T) $, with initial data $(u_0;u_1)$
we mean a function $u \in C_w(0, T; H_0^2(\Omega))$ such that $u_t
\in C_w(0,T; L_2(\Omega))$  with $g(u_t)\in L_1( Q_T)$ and
\begin{enumerate}
\item For all $\phi \in  H_0^2(\Omega)$,
\begin{align}\nonumber
& \int_{\Omega}
u_t(t) \phi d \Omega+\int_0^t \int_{\Om} ( \Delta u \Delta  \phi) \ d\Om dt
+ \int_0^t \int_{\Om} g(u_t) \phi \ d\Om dt
 \\ \label{nimic-2} & =
\int_{\Omega}
u_1 \phi d \Omega +
\int_0^t \int_{\Om} \left[P(u) \phi  -|\nabla w|^2 (\nabla w, \nabla \phi )\right]  \ d\Om dt
\end{align}
\item $\displaystyle \lim_{t \to 0} (u(t)-u_0, \phi)_{2,\Omega } = 0$.
\end{enumerate}
\end{definition}

\begin{theorem}\label{t:KB} Let $\al=0$.
Under the assumption \ref{as-BK} for all initial data
$$ w(0) =w_0 \in H^2_0(\Omega),~~ w_t (0) = w_1  \in  L_2(\Omega) $$
there exist a  local (in time)   weak  solution of finite energy.
This  is to say: there exists  $T > 0 $ such that
\begin{equation}\label{KB-weak-sm}
 w \in C_w([0, T]; H^2_0(\Omega)), ~~ w_t \in C_w([0, T]; L_2(\Omega)).
\end{equation}
Moreover
\begin{itemize}
\item
Under additional assumption that $g(s) $ is linear the said solution
is {\it unique}.  In addition energy {\it identity} is satisfied for all weak solutions  which are also Hadamard well-posed (locally).
\item
The  solutions are global  under the  following  dissipativity condition:
for every $\delta > 0 $ there exists $C_{\delta} > 0 $  such that
\begin{equation}\label{dis-BK}
  \int_0^t ((P(w(s)), w_t(s))) ds   \leq
\delta \cE(t) + C_{\delta}\left[\cE(0)+ \int_0^t \cE(s)ds\right]
\end{equation}
for any function $w(t)$ possessing the properties in  \eqref{KB-weak-sm}.
\end{itemize}
\end{theorem}

\begin{remark}
{\rm
A specific choice of interest in applications is  Boussinesq source given by
$P(w) = \Delta [ w^2]$.
 This source complies with all the hypotheses  stated. Indeed,
 \begin{align*}
((\Delta w^2, w_t))  = & ((\Delta w, \frac{d}{dt}w^2)) + 2((|\nabla w |^2, w_t))
  \\
  = &\frac{d}{dt} (( \Delta w, w^2)) - ((\Delta w_t, w^2))  + 2((|\nabla w |^2, w_t)),
 \end{align*}
 which then  gives
\begin{equation}\label{KB-del-2}
    (( \Delta w^2, w_t))   = -
\frac{d}{dt} (( |\nabla w|^2, w)) + (( |\nabla w |^2, w_t )).
\end{equation}
Since
\[
\|w(t)\|^2\le \|w(0)\|^2+2\int_0^t \|w(\tau)\|\|w_t(\tau)\| d\tau\le
C\left[\cE(0)+\int_0^t \cE(\tau) d\tau\right],
\]
relation \eqref{KB-del-2} quickly implies the conclusion desired in \eqref{dis-BK}.
}
\end{remark}

\subsection{Sketch of the  proof}
As mentioned before the challenge in the proof of Theorem \ref{t:KB} is to handle the lack of local Lipschitz condition satisfied by the restorative force.
\par
We  explain the main steps. Full details are given in \cite{cl-kb2,cl-kb}
\medskip\par\noindent
 { \bf Step 1 (Existence):} It follows via Faedo-Galerkin method.
This step is standard.
\medskip\par\noindent
{\bf Step 2 (Uniqueness):} For this we use Sedenko's method which based on
writing the difference of two solutions  in a split form as projection and
coprojection on some finite dimensional space
 and estimating.
 \par
 We start with some preliminary facts.
\par
We introduce the operator $\cA$ in $L_2(\Om)$ by the formula $\cA
u=\Delta^2u$
with the domain
 $D(\cA)= H^4(\Om)\cap H^2_0(\Om)$. The  operator $\cA$ is  a strictly positive self-adjoint
operator
with the compact resolvent.
Let $\{ e_{k} \}$  be the orthonormal basis in $L ^{2} ( \Omega  )$  of
eigenvectors of
the operator $\cA$  and $\{ \lambda_k \}$ be the corresponding eigenvalues:
$$
\cA e_k = \lambda_k e_k, \quad k=1,2,....; \quad 0 \le \lambda_1 \le
\lambda_2 \le\ldots
$$
W also note that for every $s\in [0,1/2]$ we have
$D(\cA^{s})=  H_0^{4s}(\Omega),\; s\neq 1/8, 3/8$.
 Moreover, the corresponding Sobolev norms are equivalent to the graph norms
of the corresponding fractional powers of $\cA$, i.e.
\begin{equation}\label{dom-A-s-n}
c_1\|\cA^{s}u\|\le \|u\|_{4s} \le c_2\|\cA^{s}u\|,\quad  u\in D(\cA^{s}),
\end{equation}
for all admissible $s\in [0,1/2]$.
\par
The following assertion is critical for the proof.
\begin{lemma}[\cite{ABC98}]\label{LEMMA 2.2}
Let  $P_N$  be the  projector in $L^2 (\Omega )$ onto the space
span\-ned by $\{e_1, e_2,...,e_N \}$ and
  $f(x) \in D(\cA^{1/4})$. Then there exists $N_0 > 0$
such that
\begin{equation}\label{KB-sedenko}
\max_{x \in \Omega}\vert (P_N f ) (x)\vert \le
C\cdot \{\log (1+\lambda_N)\}^{1/2}\parallel f\parallel_1
\end{equation}
for all $ N\ge N_0$. The constant $C$ does not depend on $N$.
\end{lemma}
\begin{remark}\label{re:sedenko}{\rm
For the first time a relation similar to  \eqref{KB-sedenko} was used for uniqueness
in some shell models in \cite{Sed91b}.
Latter the same method was applied for coupled $2D$ Schr\"odinger
and wave equations \cite{ChuShch,CS-11}, for the inertial $2D$ Cahn-Hilliard equation
\cite{zelik-cpde09}
and also for  some models of  fluid-shell
interaction~\cite{cryzh-12f}.
One can show
(see \cite[Appendix A]{cl-book} and the discussion therein)
 that
\eqref{KB-sedenko} is equivalent   to some Br\'esis--Gallouet \cite{Brezis}  type   inequality.
}
\end{remark}
Assume that $\al=0$, $a(x)\equiv 1$ and $g(s)=ks$ in (\ref{1.3}).
Let
$w_1 (t)$ and
  $w_2 (t)$ be two weak solutions of the  original problem  (\ref{1.3}) with the same initial data and  $w(t)=w_1 (t)-w_2 (t)$.
  Then $w_N (t)=P_N w(t)$ is a solution of the linear, but
nonhomogenous  problem
\begin{equation}\label{Ch4.6.3.1}
  w_{tt} + k w_t+ \Delta ^{2}w  = (P_N M)(t), \quad x \in  \Omega , t>0,
\end{equation}
with the boundary and initial conditions
\[
  w\vert_{\partial \Omega } = \frac{\partial w}{\partial n}\Big\vert_{\partial
\Omega } =
0, \quad
w\vert_{t=0} = 0, \quad \partial _t w\vert_{t=0} = 0.
\]
Here
\[
   M(t) = {\rm div}\left[ |\nabla w_1(t)|^2 \nabla w_1(t) -
    |\nabla w_2(t)|^2 \nabla w_2(t)\right]
   + \left[   P(w_1(t))- P(w_2(t))\right].
   \]
Multiplying the equation (\ref{Ch4.6.3.1}) by
$\cA^{-1/2} w_t $ and  integrating we obtain that
$$ ||\cA^{-1/4} P_Nw_t(t) ||^2 +  ||\cA^{1/4} P_Nw(t) ||^2
\leq C \int_0^t || \cA^{-1/4}M(\tau)|| || \cA^{-1/4} w_t(\tau) || d\tau
$$
for all $t \in [0,T]$.
 From here after accounting for  (\ref{dom-A-s-n}) we
obtain
\[
\parallel w_{t}(t) \parallel ^{2}_{-1,\Omega} +
\parallel  w(t) \parallel ^{2}_{1} \le   C\cdot\int_0^t
\parallel  M(\tau) \parallel _{-1,\Omega}\cdot
\parallel w_{t}(\tau) \parallel _{-1,\Omega} d\tau.
\]
In particular  this implies that
\begin{equation}\label{Ch4.6.3.3.n}
\parallel  w(t) \parallel_{1,\Omega} \le   C\cdot\int_0^t
\parallel  M(\tau) \parallel _{-1,\Omega}
d\tau.
\end{equation}
The inequality above is   the basis for further estimates. Below we also use the estimate which follows from the definition
of weak solutions:
\begin{equation}\label{w12-b}
\sup_{t\in [0,T]}\left\{ \|w_1(t)\|_{2,\Omega}+ \|w_2(t)\|_{2,\Omega}\right\}\le R,
\end{equation}
where $R>0$ is a constant. Using Lemma~\ref{LEMMA 2.2} we can estimate
 the quantity $\Vert M(t) \Vert_{-1}$ in the following way:
\begin{equation}\label{un8}
  \Vert M(t) \Vert_{-1} \le C_1\cdot \log (1+\lambda_N) \cdot \parallel w(t)
\parallel_1+
   C_2\cdot \lambda^{-s/4}_{N+1}
\end{equation}
for some $0<s<1$ and with the constants $C_1$ and  $C_2$ depending
on $R$ from (\ref{w12-b}) (for details we refer to \cite{cl-kb}).
Therefore
it follows from (\ref{Ch4.6.3.3.n}) and (\ref{un8})  that $\psi (t) =
\| w(t)\|_1$ satisfies the inequality
  $$
\psi (t) \le   C_1
\cdot \log (1+\lambda_N) \int_0^t \psi (\tau ) d \tau +
C_2\cdot T\cdot\lambda_{N+1}^{-s/4}, \quad t \in [0,T],
$$
for some $0<s<1$.
Thus using Gronwall's lemma we conclude that
  $$
\psi (t) \le    C_2\cdot
T\cdot\lambda_{N+1}^{-s/4}\cdot (1+\lambda_N )^{C_1t}, \quad t
\in [0,T].
$$
If we let $N \rightarrow \infty$, then for $0\le t < t_0\equiv s\cdot
(4C_1)^{-1}$
we obtain
$\psi \equiv 0$. Thus $w_1(t) \equiv w_2(t)$ for $0\le t<t_0$.
Now
we can reiterate the procedure  in order to
conclude that $w_1(t) \equiv w_2(t)$ for all $0\le t\le T$,
where $T$ is the time of  existence. This  completes the proof of uniqueness.

\medskip\par\noindent
  {\bf Step 3 (Energy identity):}
Energy inequality  relies on a standard weak  lower-semicontinuity argument.
Since the system is {\it time reversible}, one obtains energy inequality for the backward problem. Combining the two inequalities: forward and backward, leads to energy {\it identity} \eqref{plate-energy-rel}
with $\cE$ given by \eqref{KB-energy}  satisfied  by {\it weak solutions}. Details are given in \cite{cl-kb2}.
\medskip\par\noindent
 { \bf Step 4:}
 Equipped with uniqueness and energy identity  standard argument furnishes Hadamard well-posedness.
\medskip\par\noindent
{\bf Step 5:}
The last step is to extend local (in time) solutions to the global ones. This is based on a priori bounds postulated by nonlinearities in (\ref{dis-BK}).

 \subsubsection{Generalizations-Extensions}
 \begin{enumerate}
 \item
 Higher regularity of solutions can be proved by assuming more regular initial data. Quantitative statements are given in \cite{cl-kb2}.
 \item
 Rotational models, when $\alpha > 0 $,  can be considered without extra difficulty. In fact, in this  subcritical case one  obtains full Hadamard  wellposedness  also in a presence of nonlinear damping subject to the same assumptions
 as in the von Karman case. Some details can be found in \cite{cl-mem}.
 \item
 We can also consider
 different boundary conditions such as hinged and free. Free boundary conditions are most challenging due to intrinsic nonlinearity on the boundary, \cite{cl-kb}.
 \item
 More general structures of restoring forces can be also considered,
 see \cite{cl-mt,cl-kb,cl-mem,cl-kb2}.
 \item
 The support of the damping may be localized to a small (or even empty) subset of $\Omega_0\subset \Omega$.
 This will not affect a  finite time behavior of solutions.
 \end{enumerate}
\medskip\par\noindent
{\bf Open Problem:}
Uniqueness of  weak solutions  with {\it nonlinear} damping.

\subsection{Wave equation with boundary source and  damping - model in (\ref{1.4})}
 \index{wave equation!boundary  damping}
  With reference to the model (\ref{1.4}) and \eqref{w-bc-n}, where, for simplicity,  we take $g \equiv0$ and
  $f\equiv 0$, the following assumption is assumed throughout.
\begin{assumption}\label{as-wave-b}
\begin{enumerate}
\item
Scalar function $ g_0(s)$  is  assumed to be continuous and monotone on $\R$ with $ g_0(0) =0$.
\item
The source $h(s) $ is represented by a  $C^2$  function  such that
$ |h''(s)|\leq C (1 + |s|^{k-2} )$, where $ 2 \leq k  < 4$,
 and  the following growth condition is imposed on the damping $g_0(s)$
with the constants $m_{g_0},M_{g_0}  > 0$:
\begin{equation}\label{g_0}
 m_{g_0}  |s|^{q+1}   \leq  g_0(s) s  \leq M_{g_0}  |s|^{q+1},~ |s| \geq 1, ~~  with  ~~q\ge \frac{k}{4-k}\ge 1.
 \end{equation}
 When the damping is sublinear i.e.,  $q \in (0, 1)$,
  then  the the source is required to satisfy
 $|h'(s)|\leq C ( 1+ |s|^{k-1} ) ,  1 \leq k < 4 $ and the condition in (\ref{g_0}) should be satisfied for this $q \in (0,1)$.
\end{enumerate}
\end{assumption}

\begin{definition}[Weak solution]\label{def-b}
By  \underline{weak solution} of problem (\ref{1.4}) and \eqref{w-bc-n} with $g \equiv0$ and
  $f \equiv0$, defined on some
interval $(0,T) $ with initial data $(u_0;u_1)$,  we mean a function $u \in C_w(0, T; H^1(\Omega))$
such that
\begin{enumerate}
\item $\displaystyle u_t
\in C_w(0,T; L_2(\Omega) )~~\mbox{and}~~ u_t \in L_{q+1}(\Sigma_T)$, where $\Sigma_T= [0,T]\times\Gamma$,
\item For all $\phi \in C(0,T, H^1(\Omega)) \cap C^1(0,T; L_2(\Om)) \cap
L_{q+1}(\Sigma_T)$,
\begin{align}\nonumber
& \int_0^T \int_{\Om} (-u_t \phi_t + \nabla u \nabla \phi) \ d\Om dt
+ \int_0^T \int_{\Ga} g_0(u_t) \phi \ d\Om dt
 \\ \label{nimic-b} & = -\int_{\Omega}
u_t \phi d \Omega  \Big|_0^T +
\int_0^T \int_{\Ga} h(u) \phi \ d\Om dt.
\end{align}
\item $\displaystyle \lim_{t \to 0} (u(t)-u_0, \phi)_{1,\Omega } = 0$ and
$\displaystyle \lim_{t \to 0}((u_t(t) - u_1,\phi)) = 0$ for all $\phi\in H^1(\Omega)$.
\end{enumerate}
\end{definition}

\begin{theorem}\label{t:b}
Let $f\equiv 0$, $g\equiv 0$ and Assumption \ref{as-wave-b} be in force. Let  initial data $(w_0;w_1)$
be such that
$$
w(0) =w_0 \in H^1(\Omega),~~ w_t (0) = w_1  \in  L_2(\Omega),
$$
and also $ w_0|_{\Ga} \in L_{2k-2}(\Ga)$ when $k>3$. Then
 there exist a unique, local (in time)  weak solution of finite energy.
This  is to say: there exists  $T > 0$ such that
$$ w \in C([0, T]; H^1(\Omega)),~~~ w_t \in C([0, T]; L_2(\Omega)). $$
Moreover,
\begin{itemize}
\item When $k \leq 3 $ the energy {\it identity} holds for weak solutions and  weak  solution is continuously dependent  on the initial data.
\item
When $ k \leq q $ the obtained  solutions are global,  i.e., $T =\infty$.
The same holds under dissipativity condition: $ -h(s) s \geq  0$.
\item
When $ 1< k\leq  3  $, $ k> q $  and $ h(s) = |s|^{k-1}s$  local  solution  blows up in a finite time  for negative energy initial data.
\end{itemize}
\end{theorem}
\begin{remark}
In the case when $h(s) =\alpha s $, no growth conditions imposed on $g_0$ are required. In fact, in that case the obtained solution is semigroup generalized solution.
\par
When $|h''(s)| \leq c $ variational form of the solution  can be obtained with more relaxed hypotheses imposed on high frequencies of the damping.  For instance,
we can assume
\begin{equation}\label{h}
\liminf_{|s| \rightarrow \infty } \frac{g_0(s)}{s}  > 0 , ~~ |g_0(s) | \leq c [1+ |s|^3 ]
\end{equation}
Under the above condition one can show that the generalized solution satisfies also variational form (\ref{nimic-b}) with the test functions $\phi \in C(0, T; H^1(\Omega) ) \cap C^1(0, T; L_2(\Omega) ) \cap C(\Sigma_T) $ \cite{snowbird}
\end{remark}

\subsubsection{Reference to the  proofs}
Notice that the damping $g_0$ is assumed  active for all values of the parameter $k>1$. This is unlike interior source where only supercritical values of the source require presence of the damping.
In the boundary case, however, Lopatinski condition is not satisfied for the Neumann problem and this necessitates
the presence of the damping which, in some sense, forces Lopatinski condition \cite{sakamoto}.
This is manifested  by the  fact that $H^1(Q) $ solutions of wave equation  with Neumann boundary conditions
do not possess $H^1(\Sigma) $ boundary  regularity (unlike Dirichlet solutions). It is the presence of the damping which, in some sense, recovers certain amount of boundary regularity \cite{l-tbook,vittillaro}.
\par
The proof of  Theorem \ref{t:b} follows similar conceptual lines as in the interior   case.  There are however few subtle differences due to the  presence of undefined traces in the equation.
  In the case when $|h''(s) | \leq C $ and $ m s^2 \leq  g_0(s) s \leq M s^4 $  the proof is given in \cite{snowbird}.
  For the remaining values of the parameters Hadamard  wellposedness  is proved in \cite{bociu,bociu-JDE}.
  The analysis of the boundary case is more demanding than in the interior case.
The corresponding  differences are clearly exposed in the  structure of a counterpart to  Lemma   \ref{l:1}, which in the boundary case
has  a different form  given in Lemma 4.2 in  \cite{bociu}. \par
The last statement  in Theorem  \ref{t:b} regarding finite time blow up of energy is proved in \cite{bociu-1}.
\subsubsection{Generalizations-Extensions}
\begin{enumerate}
\item
Additional regularity of solutions  corresponding to more regular initial conditions when $q \leq 3 $ and $|h''(s) | \leq C$ is given in \cite{snowbird}.
\item
One could obtain a version of Theorem \ref{t:b} that incorporates both dampings: internal and boundary.
This is done in \cite{bociu,bociu-JDE}.
\item
Well-posedness theory for small data taken from potential well ($k \leq 3 $)  is also available \cite{bociu-r,vitillaro-1}.
\end{enumerate}
{\bf Open questions:}
\begin{enumerate}
\item
Well-posedness  of finite energy solutions theory without the boundary damping when $k \leq 2 $ (so that $h(u)
\in L_2(\Gamma) $ for finite energy solutions.
\item
Interplay between boundary and interior damping. Is it possible to show well-posedness of finite energy solutions for supercritical case $ p > 3 $ with  boundary damping  only (i.e. $g =0$)?
\item
The same question asked for finite time  blow results.  Does boundary source alone  leads to blow up of energy
in the presence of  a boundary damping when $ q < k $?
\item
Boundedness of solutions when time goes to infinity and sources are generating energy. Quantitative description of the behavior of global solutions - when $p\leq  m $ and $k \leq  q $.
\end{enumerate}

\subsection{Von Karman equation with boundary damping  - \\ model (\ref{1.5}) }
With reference to the model (\ref{1.5}) and
\eqref{hing-bc}  the following assumption is assumed throughout.
\begin{assumption}\label{as-karman-b}
\begin{enumerate}
\item
Scalar function $g(s)$ and $g_0 (s) $ are  assumed to be continuous and monotone on $\R$ with $0$ at $0$. The rotational damping $G$ (the case $\al>0$) satisfies
Assumption~\ref{as-karman}(2).
\item
The source $P(w) $ is assumed locally Lipschitz from $H^2(\Omega)$ into  $H^{-1}(\Omega)$
 when $\alpha > 0 $ and from $H^2(\Omega)$ into $L_2(\Omega) $ when $\alpha =0$
 (see Remark~\ref{re:F0} concerning a specific choice of interest in applications).
\end{enumerate}
\end{assumption}
We concentrate on a more challenging   case when $\alpha =0$.
In addition to the concept of {\it generalized (semigroup)}
solutions we  also define {\it weak} solutions.
\begin{definition}[Weak solution]
By a \underline{weak solution} of  (\ref{1.5}) and
\eqref{hing-bc}  with $\al=0$ and with initial data $(u_0;u_1)$, defined on some
interval $(0,T) $,  we mean a function $u \in C_w(0, T; H^2(\Omega) \cap H_0^1(\Omega))$ such that $u_t \in C_w(0,T; L_2(\Omega) )$
with $g(u_t)\in L_1( Q_T)$, $g_0((\partial/\partial n)u_t)\in L_1( \Sigma_T)$
and
\begin{enumerate}
\item For all $\phi \in  H^2(\Omega)  \cap H_0^1(\Omega)$,
\begin{align}\nonumber
& \int_0^t \int_{\Om} ( \Delta u \Delta  \phi) \ d\Om d\tau
+ \int_0^t \int_{\Om} g(u_t) \phi \ d\Om d\tau +\int_0^t \int_{\Ga} g_0(\Dn u_t)  \Dn  \phi d\Ga d\tau
 \\ \label{nimic-kb} & = -\int_{\Omega}
(u_t(t)-u_1) \phi d \Omega  +
\int_0^t \int_{\Om} (P(u) + [\cF(u),u] ) \phi \ d\Om d\tau.
\end{align}
\item $\displaystyle \lim_{t \to 0} (u(t)-u_0, \phi)_{2,\Omega } = 0$ and
$\displaystyle \lim_{t \to 0}((u_t(t) - u_1,\phi)) = 0$.
\end{enumerate}
\end{definition}

\index{evolutionary Karman equations without rotational
forces!boundary dissipation, clamped--hinged b.c.}
\begin{theorem}\label{th:wp-karman2} Let $\al=0$.
Under Assumption \ref{as-karman-b} for all initial data
$$ w(0) =w_0 \in H^2(\Omega) \cap H_0^1(\Omega),~~ w_t (0) = w_1  \in  L_2(\Omega) $$
there exist a unique, local (in time) generalized  (semigroup) solution $w$ of finite energy, i.e.,
$$
\exists\, T > 0\; :~~ w \in C([0, T]; H^2(\Omega)),~~ w_t \in C([0, T]; L_2(\Omega)).
$$
Moreover
\begin{itemize}
\item
If, in addition, the damping $g(s)$ and $g_0 (s)$ are  of some polynomial growth and
$(g_0(x) - g_0 (y) ) (x-y) \geq c |x-y|^r $ for some $r \geq 1 $, then a  generalized solution becomes also  weak solution.
 Weak  solution is  also continuously dependent  on the initial data.
\item
The  solutions are global  under  Assumption~\ref{as:karman-P}
 concerning $P(w)$ with the energy  bound independent of  the time horizon.

\end{itemize}
\end{theorem}

\subsubsection{Reference to the proof}

Well-posedness of Von Karman plates with the boundary damping occurring
 in the moments  have been proved in
\cite{cl-book}. See also \cite{l-ji,horn}.

\subsection{Generalizations}
\begin{enumerate}
\item
Additional regularity of  finite-time solutions for more regular initial data, \cite{cl-book,horn}.
\item
Similar problems can be formulated for plate equations with the damping acting in shears and torques \cite{lagnese}.
Thus the boundary conditions under consideration are "free" with the feedback control given by $g_0(w_t)$ acting on the highest third order boundary conditions \cite{chla03,cl-jde-07}.
\item
Combination of interior and boundary damping can also be considered by combining the methods.
\item
Model with rotational inertial  forces  $\alpha > 0$  subject to boundary damping, \cite{cl-book,horn}.
\item
Regular  solutions for infinite time interval can be obtained \cite{cl-book}.
\end{enumerate}
{\bf Open Problem:}
Are  polynomial bounds imposed on $g(s)$ and $g_0(s)$ and  strong coercivity condition imposed on $g_0$   necessary for obtaining weak solutions?

\subsection{
 Kirhhoff-Boussinesq equation with boundary \\ damping  - model (\ref{1.6}) }
\index{Kirchhoff-Boussinesq model!boundary damping}
\subsubsection{Rotational case $\alpha> 0 $ }
In the case when rotational inertia are retained in the model, $\alpha > 0$,
 the nonlinear terms are subcritical and the
well-posedness theory can be carried out
along the same lines as in the case of von Karman equations.
For instance, when  considered (\ref{1.6}) with $g =0$, $G =0$ and the following boundary damping
 \begin{equation}\label{KB-bnd}
 w =0, ~~\Delta w = - g_0(\Dn w_t ) ~~{\rm on}~~\Sigma,
 \end{equation}
 where $g_0(s)$ is continuous and monotone, full Hadamard local well-posedness   can be established.
\begin{theorem}\label{th:KB-b}
Under the assumption \ref{as-karman-b} with $g=0$ and $\al>0$ for all initial data
$$ w(0) =w_0 \in H^2(\Omega) \cap H_0^1(\Omega) , w_t (0) = w_1  \in  H_0^1(\Omega) $$
there exist a unique, local (in time) generalized  (semigroup) solution $w$ of finite energy, i.e.,
$$
\exists\, T > 0\; :~~ w \in C([0, T]; H^2(\Omega)\cap H_0^1(\Omega)),~~ w_t \in C([0, T]; H_0^1 (\Omega)).
$$
Moreover
\begin{itemize}
\item
If, in addition, the damping $g_0 (s)$ are  of some polynomial growth and
$(g_0(x) - g_0 (y) ) (x-y) \geq c |x-y|^r $ for some $r \geq 1 $, then a  generalized solution becomes also  weak solution.
 Weak  solution is  also continuously dependent  on the initial data.
\item
The  solutions are global  under the dissipativity   hypothesis in~\eqref{dis-BK}
\end{itemize}
\end{theorem}
The proof of this theorem is along the same lines as in the case of von Karman equations
with boundary damping.
  \subsubsection{Non-rotational  case: $\alpha =0 $ }
 As we have seen already before, this case is much more delicate even in the case of interior damping.
 The nonlinear source term is supercritical. While this difficulty was overcome in the case of linear interior damping, the  presence of boundary damping brings new set of issues.
 This is both, at the level of uniqueness  and continuous dependence.
 More specifically, in the case $g\equiv 0$:
 \begin{itemize}
 \item
  Existence of finite energy solutions
 $$w \in L_{\infty}([0, T]; H^2(\Omega)\cap H_0^1(\Omega)),~~ w_t \in L_{\infty} ([0, T]; L_2 (\Omega))
$$
defined variationally  and locally in time
with initial data
  $$
  w(0) =w_0 \in H^2(\Omega) \cap H_0^1(\Omega) , w_t (0) = w_1  \in   L_2(\Omega).
  $$
  can be proved by Galerkin method  with a boundary damping  $g_0 (s)$ of some polynomial growth and such that
$(g_0(x) - g_0 (y) ) (x-y) \geq c |x-y|^r $ for some $r \geq 1 $.
\item
Uniqueness of solutions can be established only in the case of  {\it linear}  damping, i.e., for $g_0(s) = a s$. This can be done by adapting Sedenko's method \cite{Sed91b}
as in \cite{sicon,koch} where  full Von Karman system was considered.
The latter displays similar difficulties when dealing with well-posedness.
\item
Continuous dependence on the data can be proved only in a weak topology.  The difficulty lies in the fact that time reversibility of the flow is lost with the presence of boundary damping. This latter property was critical in
proving energy identity for weak solutions in the case of internal linear damping.
\end{itemize}

\subsubsection{Generalizations-Extensions}
\begin{enumerate}
\item
Existence and uniqueness of regular solutions in $H^4(\Omega) \times H^2(\Omega) $
(when $\alpha =0 $). This  can be accomplished by  taking advantage of  higher topologies for the  state, hence  avoiding the  problem of supercriticality. This method has been pursued in \cite{cl-kb2} in the case of interior damping.
However, the arguments are applicable
to the case of boundary damping as well.
 \item
 Other types of boundary damping  -such as  ocurring in "free" \cite{cl-book}  boundary conditions can also be considered.
 In the case  rotational case, well-posedness results are complete - in line with Theorem \ref{th:KB-b}.
 In the non-rotational case, comments presented above apply to this case as well (with the same limitations regarding linearity of the damping).There is however an additional difficulty resulting from the fact that free boundary conditions
 provide   intrinsically nonlinear  contribution on the boundary. This, however,  is of lower order and can be handled as in
 \cite{cl-kb,cl-kb2}
 \item
 Combination of internal damping and boundary damping can be treated by combining the results available for each case separately.
 \item
 Sources that are nonconservative also can be considered, but these are outside the scope of these lectures, see \cite{cl-book}.
 \end{enumerate}
\medskip\par\noindent
{\bf Open Problems:}
\begin{enumerate}
\item
Hadamard well-posedness with  {\it linear  boundary damping}  in the case $\al=0$. Due to the
loss of time reversibility, the arguments used for the internal damping are no longer applicable.
\item
 Ultimately: uniqueness and Hadamard well-posedness in the presence of
 nonlinear damping $g_0$ in the case $\al=0$ and
 boundary conditions \eqref{KB-bnd},
 where $g_0(s)$ is monotone and -say- of linear growth at infinity.
 The arguments used so far  for the uniqueness of weak solutions rely critically on linearity of the damping.
 Thus, the case described above appears to be completely open.
 \end{enumerate}

\section{General tools for   studying attractors}
In this section we describe several  approaches
to  the study of long-time behavior of  hyperbolic-like   systems
described above. For the general discussion of long-time behavior of  systems with dissipation we refer to  the monographs
\cite{BV92,Chu99,Ha88,Lad91,sell,temam}.

\subsection{Basic notions}
By definition a dynamical system is a pair of objects $(X,S_t)$ consisting of
a complete metric space $X$ and a family of continuous mappings
$\{S_t\, :\, t\in \R_+\}$ of $X$ into itself with the semigroup
properties:
\[
S_0=I,\quad S_{t+\tau}=S_t\circ S_\tau.
\]
We also assume that $y(t)=S_ty_0$ is  continuous with respect to
$t$ for any $y_0\in X$. Therewith $X$ is called a {\em  phase space}
(or state space) and $S_t$ is called an {\em evolution semigroup} (or
evolution operator).
\index{dynamical system}
\index{evolution semigroup}
\index{operator!evolution}
\index{dynamical system!phase space}

\begin{definition}\label{de7.1.1} Let $(X, S_t)$ be a dynamical system.
\begin{itemize}
\item
\index{set!absorbing}
A closed set
$B\subset X$  is said to be {\bf  absorbing} for $(X, S_t)$
iff for any bounded set $D\subset X$  there exists
$t_0(D)$ such that
$S_t D\subset B$ for all $t\ge t_0(D)$.
\item
\index{dynamical system!dissipative}
$(X, S_t)$ is said to be (bounded, or ultimately) {\bf dissipative}
iff it possesses a bounded absorbing set $B$.
If $X$ is a Banach space, then
a value $R>0$ is said to be a
radius of dissipativity of  $(X, S_t)$
if $B\subset\{ x\in X\,:\, \Vert x\Vert_X\le R\}$.
\index{radius of dissipativity}
\item
\index{dynamical system!asymptotically smooth}
$(X,S_t)$ is said to be {\bf asymptotically smooth}
iff for any bounded set $D$ such that $S_tD\subset D$
for $t>0$,  there exists  a  compact set $K$ in the closure $\overline{D}$
of $D$, such that \begin{equation}\label{7.1.2}
\lim_{t\to +\infty}
d_X\{S_tD\, |\, K\}=0,
\end{equation}
where $d_X\{A | B\}= \sup_{x\in A} \dist_X (x, B)$.
\end{itemize}
\end{definition}

\par
A set $D\subset X$ is said to be {\em forward} (or positively)  {\em invariant} iff
$S_tD\subseteq D$ for all $t\ge 0$. It is {\em backward } (or negatively) {\em invariant}
iff $S_tD\supseteq D$ for all $t\ge 0$. The set $D$  is said to be
{\em  invariant} iff it is both forward  and backward invariant; that is,
$S_tD= D$ for all $t\ge 0$.
\index{set!invariant}
\index{set!forward invariant}
\index{set!backward invariant}
\index{set!positively invariant}
\index{set!negatively invariant}
\par
Let $D\subset X$. The set
$$
\gamma^t_D\equiv
\bigcup_{\tau\ge t}S_\tau D
$$
is called the {\em tail} (from the moment $t$) of the  trajectories
\index{tail of trajectory} emanating from $D$. It is clear that
$\gamma^t_D=\gamma^0_{S_tD}$.
\par
 If $D=\{ v\}$ is
a single point set, then $\gamma^+_v:=\gamma^0_D$ is said to be
a   {\em positive semitrajectory (or semiorbit)} emanating from $v$.
\index{semitrajectory} \index{semiorbit|see{semitrajectory}} A
continuous curve \index{full trajectory} $\gamma \equiv \{ u(t)\,
:\, t\in\R\}$ in $X$ is said to be a {\em full trajectory}  iff
$S_tu(\tau)=u(t+\tau)$ for any $\tau\in\R$ and $t\ge 0$. \Since $S_t$
is not necessarily
 an invertible operator,
a full trajectory may not exist. Semitrajectories are forward invariant
sets. Full trajectories are invariant sets.
\par
To describe the asymptotic behavior
  we use the concept of an  $\omega$-limit set.
The set
\begin{equation*}
\omega(D)\equiv
\bigcap_{t>0}\overline{\gamma_D^t}
=
\bigcap_{t>0}\overline{\bigcup_{\tau\ge t}S_\tau D}
\end{equation*}
is called the  {\em $\omega$-limit set} of the trajectories
emanating from $D$ (the bar over a set means the closure).\index{set!$\omega$-limit}
It is equivalent to saying that $x\in\omega(D)$ if and only if there exist
sequences $t_n\to +\infty$ and  $x_n\in D$ such that $S_{t_n}x_n\to x$
as $n\to \infty$. It is clear that $\omega$-limit sets (if they exist)
are forward invariant.
\subsection{Criteria for Asymptotic Smoothness}
The following assertion is a  generalization  Ceron-Lopes criteria
(see \cite{Ha88}
and  the references therein).
\begin{theorem}\label{th7.1.1}
Let  $(X, S_t)$ be  a dynamical system on  a Banach space $X$.
Assume that
for any bounded positively invariant set $B$ in $X$
there exist $T>0$,
a continuous nondecreasing  function $g :\R_+\mapsto\R_+$,
and a pseudometric $\varrho^T_B$ on $C(0,T;X)$ such that
\begin{enumerate}
              \item[(i)]
$g(0)=0;~~ g(s)<s,~ s>0$.
\item[(ii)]
The  pseudometric $\varrho^T_B$  is precompact
 (with respect to  the norm of $X$) in the following sense.
Any   sequence $\{x_n\}\subset B$  has a subsequence
$\{x_{n_k}\}$ such that the sequence $\{y_k\}\subset C(0,T; X)$
of elements $y_k(\tau)=S_\tau x_{n_k}$  is  Cauchy
with respect to  $\varrho^T_B$.
\index{precompact pseudometric}
\item[(iii)] The  following estimate
\begin{equation*}
\Vert S_Ty_1-S_Ty_2\Vert\le
g\left(\Vert y_1-y_2\Vert +
\varrho^T_B(\{S_\tau y_1\},\{S_\tau y_2\})\right)
\end{equation*}
holds for every $y_1,y_2\in B$, where we denote by
$\{S_\tau y_i\}$ the element in  the space $C(0,T; X)$ given
by function $y_i(\tau)=S_\tau y_i$.
 \end{enumerate}
 Then,  $(X,S_t)$ is asymptotically smooth dynamical system.
\end{theorem}
Note that a precompact  pseudometric is evaluated on trajectories $S_{\tau}$, rather than on initial conditions (as in the classical treatments, see, e.g., \cite{Ha88}).   This fact becomes quite useful when applying
the criterion to hyperbolic-like dynamics.
\begin{proof}
We refer to \cite{cl-mem}. The main ingredient of the proof  is the relation
\begin{equation}\label{alp-g}
    \al(S_TB)\le g(\al(B)),
\end{equation}
where
$\al(B)$   is
 the Kuratowski $\alpha$-measure of noncompactness which is
defined by the formula
\index{Kuratowski $\alpha$-measure}
\begin{equation}\label{kurat}
\alpha(B)=\inf\{\delta\, :\, B~\mbox{has a finite cover of diameter}~
<\delta\}
\end{equation}
for every  bounded set $B$ of $X$.
For  properties of this metric characteristic
we refer to  \cite{Ha88} or \cite[Lemma 22.2]{sell}.
\par
The property in \eqref{alp-g}
implies that
 for every bounded  forward invariant set $B$
there exists $T>0$ such that
$\alpha(S_TB)<\alpha(B)$ provided $\alpha(B)>0$.
If for some fixed $T>0$ this property
 holds for every bounded set $B$ such that $S_TB$ is also bounded,
then, by the definition (see \cite{Ha88}), $S_T$ is a conditional
$\alpha$-condensing mapping. It is known \cite{Ha88} that
conditional $\alpha$-condensing mappings are
  asymptotically smooth.
Therefore Theorem~\ref{th7.1.1}  can be
considered as a generalization of the results presented in \cite{Ha88}.
\end{proof}
The above criterion is rather general, however it requires "compactness" of  the sources in the equation.
There are two other criteria that  avoid such a requirement of a priori compactness.
 These are:
 \begin{itemize}
     \item {\it compensated compactness} criterion
 -an idea introduced in \cite{turk-jmaa} and later expanded
 in \cite{cl-mem};
     \item {\it J. Ball's energy method}, see \cite{ball} and \cite{MRW}.
   \end{itemize}
The  corresponding results are presented below.

 \subsubsection{ Compensated compactness method}
\begin{theorem}\label{th7-turk}
Let  $(X, S_t)$ be  a dynamical
system on  a complete metric space  $X$ endowed with  a metric $d$.
Assume that  for any bounded positively invariant set $B$ in $X$ and
for any $\epsilon > 0$ there exists $T \equiv T(\epsilon, B) $ such that
\begin{equation}\label{7.ak1}
d( S_T y_1, S_T y_2) \leq \epsilon + \Psi_{\epsilon, B, T}
(y_1, y_2) , y_i \in B,
\end{equation}
where $\Psi_{\epsilon, B, T } (y_1, y_2) $ is  a functional  defined on
$B \times B $ such that
\begin{equation}\label{7.ak2}
\liminf_{m \rightarrow \infty }  \liminf_{n \rightarrow \infty }
\Psi_{\epsilon, B, T } (y_n, y_m ) =0
\end{equation}
for every sequence $\{y_n\} $ from $B$.
 Then $(X,S_t)$ is an asymptotically smooth dynamical system.
\end{theorem}
Note that in a "compact" situation, i.e.,  when the functional $\Psi$ is sequentially compact, the condition (\ref{7.ak2}) is automatically satisfied. The above criterion applies  in the case of critical nonlinearities.
\begin{proof}
  The properties in \eqref{7.ak1} and \eqref{7.ak2} makes  it possible to prove
  that $\al(S_tB)\to 0$ as $t\to+\infty$, where $\alpha(B)$
  is the Kuratowski $\al$-measure defined by \eqref{kurat}. The latter
  property implies the conclusion desired. For details we refer
  \cite{cl-mem} or \cite{cl-book}.
\end{proof}

\subsubsection{ John Ball's  "energy" method}
This  second method  applies  even in the case of supercritical nonlinear terms, but there are other requirements which
restrict applicability of this method.
The main idea behind Ball's method is to construct an appropriate energy type  functional which can be then decomposed into exponentially decaying part and compact part.
While the idea of  decomposition  of  semigroup into uniformly stable and compact part is  behind  almost all
criteria  leading to asymptotic smoothness, the "energy" method described below postulates such a decomposition on
{\it functionals} rather than operators (semigroups). This fact has far reaching consequences and allows application of the method in  supercritical situations.
\par

We   follow presentation of the method given in
\cite{MRW}. Another  exposition of this method in the case
of the damped wave equation can be found in \cite{ball}.

\begin{theorem}\label{th:ball}
Let $S_t$ be an evolution continuous
semigroup of weak\-ly and stron\-gly continuous operators.
Assume that there exist functionals $\Phi$, $\Psi$, $L$, $K$ on phase space
such that the following equality
\begin{multline}
[\Phi(S_tu)+\Psi(S_t u)] +\int_s^t L(S_{\tau}u)e^{-\om(t-\tau)}d\tau
\\
=[\Phi(S_s u)+\Psi(S_s u)]e^{-\om(t-s)}+\int_s^t K(S_{\tau} u)e^{-\om(t-\tau)}d\tau, \label{ball_eq}
\end{multline}
holds for any $u \in X $.
The  functionals $\Phi$, $\Psi$, $L$, $K$ are assumed to have the following  properties:
\begin{itemize}
\item $\Phi:\; X \mapsto \R^+$ is continuous, bounded and
if $\{U_j\}_j$ is bounded in X, $t_j \to +\infty$, $S_{t_j}U_j \rightharpoonup U$ weakly in $X$, and
$\limsup_{n\to\infty}\Phi(S_{t_j}U_j) \le \Phi (U)$, then $S_{t_j}U_j \to U$ strongly in $X$.

\item $\Psi:\; X \mapsto \R$ is
'asymptotically weakly continuous' in the sense
that if $\{U_j\}_j$ is bounded in X, $t_j \to +\infty$,
$S_{t_j}U_j \rightharpoonup U$ weakly in $X$, then $\Psi(S_{t_j}U_j) \to \Psi(U)$.

\item $K:\; X \mapsto \R$ is 'asymptotically weakly
continuous' in the sense that if $\{U_j\}_j$ is bounded in X, $t_j \to
 +\infty$, $S_{t_j}U_j \rightharpoonup U$ weakly in $ X$, then $K(S_sU)\in L_1(0,t)$ and
    \begin{equation*}
        \lim_{j\to\infty}\int_0^t e^{-\om(t-s)}K(S_{s+t_j}U_j) ds =
        \int_0^t e^{-\om(t-s)}K(S_{s}U) ds, \quad \forall t>0.
    \end{equation*}
\item $L$ is 'asymptotically weakly lower semicontinuous' in the sense
that if $\{U_j\}_j$ is bounded in X , $t_j \to +\infty$,
$S_{t_j}U_j \rightharpoonup U$ weakly in $ X $, then $L(S_sU)\in L_1(0,t)$ and
    \begin{equation*}
        \liminf_{j\to\infty}\int_0^t e^{-\om(t-s)}L(S_{s+t_j}U_j) ds
        \ge \int_0^t e^{-\om(t-s)}L(S_{s}U) ds, \quad \forall t>0.
    \end{equation*}
\end{itemize}
Then $S_t$ is asymptotically smooth.
\end{theorem}
\begin{remark}{\rm
\begin{enumerate}
\item
The decomposition of the functionals in (\ref{ball_eq}) depends on   validity of energy {\it identity}
satisfied for weak solutions. This, alone, is a severe condition which may be difficult to verify (in contrast to energy {\it inequality}).
In addition, the proof of {\it equality} in (\ref{ball_eq}) hides behind the fact that the damping in the equation is very structured-typically linear.
\item
The functional  (typically convex) $\Phi$  plays role of the energy of
linearized system - a good topological measure for the solution.
In uniformly  convex spaces  $X$ the assumptions postulated by $\Phi $ are automatically satisfied.
Indeed, this results from the fact that weak convergence and the  convergence
 of the norms to the same element imply strong convergence.
\item
The hypotheses  required from $\Psi$,  $K$ and $L$ represent some compactness property of part of the  nonlinear  energy describing the system.  This is often the case even for  supercritical nonlinearities.  In fact, these terms allow to deal with non-compact sources in the equation provided that the corresponding nonlinear part of the  energy is
sequentially compact. Typical application involves $3D$ wave equation,
where sources $|f(s)| \leq C(1+ |s|^p)$ with $p < 5$ can be handled (in view of compactness of the embedding $H^1(\Om) \subset L_q(\Om)$, $q<6$).
\end{enumerate}
}
\end{remark}

\subsection{Global attractors}

The main objects arising in the analysis  of long-time behavior of
infinite-dimen\-sional dissipative dynamical systems are attractors.
Their study allows us to answer a number of fundamental questions on
the properties of limit regimes that can arise in  the systems  under
consideration.
At present,  there are several   general  approaches and methods that allow us
to prove the existence and finite-dimensionality of global attractors
for a large class of dynamical systems generated by nonlinear partial
differential equations (see, e.g., \cite{BV92,Chu99,Ha88,Lad91,temam}
 and the references listed therein).

 \begin{definition}\label{de7.2.1}\index{global attractor}
\index{attractor!global}
A bounded closed set $A\subset X$
is said to be a {\em global attractor} of the dynamical system
$(X, S_t)$  iff
the following properties hold.
\begin{enumerate}
\item[(i)] $A$ is an invariant set; that is,
 $S_tA =A$ for $t\ge 0$.
\item[(ii)] $A$ is uniformly attracting; that is,  for all bounded set $D\subset X$
\begin{equation*}
\lim_{t\to +\infty}
d_X\{S_tD\, |\, A\}=0,
\end{equation*}
where $d_X\{A | B\}= \sup_{x\in A} \dist_X (x, B)$ is the Hausdorff semidistance.
\index{Hausdorff semidistance}
\end{enumerate}
\end{definition}

It turns out that dissipativity property along with asymptotic smoothness imply an existence of global attractor. The corresponding result is standard by now and reported below (see \cite{Ha88} and also  \cite{BV92,Lad91,temam}).
 \begin{theorem}\label{t7.2.1-as} Any dissipative
 asymptotically smooth  dynamical system  $(X,S_t)$
in a Banach space $X$ possesses a unique  compact global attractor
$A$.  This
attractor is a connected set and can be described as a set of all bounded full
trajectories. Moreover $A=\om(B)$ for any
 bounded absorbing set $B$ of  $(X,S_t)$.
\end{theorem}

In the case when the dynamical system has special property-referred to as {\it gradient system}-
dissipativity property is not needed  (in the explicit form) in order to prove existence of a global attractor.
This is a  very handy property particularly  when the  proof of dissipativity is technically  involved.

\subsubsection{Gradient systems}\label{grad-ds}
The study of the structure of  the global attractors is an important problem
from the point of view of applications. There are no universal approaches
solving  this
problem.  It is well known that even in finite-dimensional cases an attractor can possess
extremely complicated  structure. However, some sets that belong to
the attractor can be easily pointed out. For example, every stationary
point ($S_tx=x$ for all $t>0$) belongs to the attractor of the system.
\index{stationary point}
One can shows that any bounded full trajectory also lies
in the global attractor.
\par
 We  begin with  the following definition.
\begin{definition}\label{de:unst-m}
Let $\cN$ be the set of stationary points of the dynamical system
$(X,S_t)$:\index{set of stationary points}
\[
 \cN=\left\{ v\in X\; :\; S_tv=v\mbox{~for all~} t\ge 0\right\}.
\]
We define  the {\bf unstable manifold}
\index{unstable manifold}
 $\cM^u(\cN)$ emanating from the set
$\cN$ as a set of all $y\in X$ such that there exists a full
trajectory $\gamma=\{ u(t)\, :\, t\in\R\}$ with the properties
\[
 u(0)=y~~ \mbox{and}~~ \lim_{t\to -\infty}{\rm dist}_{X}(u(t),\cN)=0.
\]
\end{definition}
Now we introduce the notions of Lyapunov functions and gradient systems (see, e.g.,
\cite{BV92,Chu99,Ha88,Lad91,temam} and the references therein).

\begin{definition}\label{de7.4.1}
Let $Y\subseteq X$ be a forward invariant set of a dynamical system
$(X,S_t)$.
\begin{itemize}
\item
The continuous functional $\Phi(y)$ defined on $Y$ is said to
be the {\bf Lyapunov function} for the dynamical system
$(X,S_t)$ on $Y$
\index{Lyapunov function}
iff
the function $t\mapsto \Phi(S_ty)$ is a  nonincreasing function
for any $y\in Y$.
\item
The Lyapunov function $\Phi(y)$ is said to be {\bf strict} on $Y$
\index{Lyapunov function!strict}
iff the equation $\Phi(S_{t}y)=\Phi(y)$ for {\em all} $t>0$ and for some
$y\in Y$ implies that $S_{t}y=y$ for all $t>0$; that is,
$y$ is a stationary point of $(X,S_t)$.
\item
The dynamical system $(X,S_t)$ is said to be {\bf gradient} iff there
exists a strict Lyapunov function for $(X,S_t)$ on the whole phase space $X$.
\end{itemize}
\end{definition}
\index{gradient system}
\index{dynamical system!gradient}
A connection between gradient systems and existence of compact attractors is given below.
The main  result  stating    existence and properties of
 attractors for gradient systems  is the following theorem (for the proof we refer to
 \cite{cl-mem,cl-book} and the references therein; see also
 \cite[Theorem 4.6]{Raugel} for a similar assertion).
\begin{theorem}\label{th:grad-attr}
Assume that  $(X,S_t)$ is a gradient asymptotically smooth dynamical system.
Assume its  Lyapunov function $\Phi(x)$
 is bounded from above on any bounded subset of $X$ and
 the set  $\Phi_R=\{ x\, :\, \Phi(x)\le R\}$ is bounded for every $R$.
If the set $\cN$  of stationary points of $(X,S_t)$ is bounded,
then $(X,S_t)$ possesses a compact global attractor $A=\cM^u(\cN)$.
 Moreover,
 \begin{itemize}
   \item The global attractor $A$ consists of  full trajectories
 $\gamma=\{ u(t)\, :\, t\in\R\}$ such that
\begin{equation*}
 \lim_{t\to -\infty}{\rm dist}_{X}(u(t),\cN)=0 ~~
 \mbox{ and } ~~ \lim_{t\to +\infty}{\rm dist}_{X}(u(t),\cN)=0.
\end{equation*}
   \item For any $x \in X$ we have
 \begin{equation*}
  \lim_{t\to +\infty}{\rm dist}_{X}(S_tx,\cN)=0,
 \end{equation*}
that is, any trajectory stabilizes to the set $\cN$ of  stationary points.
In particular, this means that  the global minimal attractor $A_{\rm min}$
coincides  with the set of the stationary points, $A_{\rm min}=\cN$.
   \item If $\cN=\{z_1,\ldots, z_n\}$ is a finite set,
then $A=\cup_{i=1}^n \cM^u(z_i)$, where
$\cM^u(z_i)$ is the unstable manifold of the stationary point $z_i$, and also
\begin{enumerate}
  \item[(i)] The global attractor $A$ consists of  full trajectories
   $\gamma=\{ u(t)\, :\, t\in\R\}$  connecting  pairs of
   stationary points:  any $u\in A$ belongs to  some full trajectory
   $\gamma$ and
   for any $\gamma\subset A$ there exists a pair
   $\{ z, z^*\}\subset\cN$ such that
  \[
    u(t)\to z \mbox{~as~} t\to -\infty
   \mbox{ and }
    u(t)\to z^* \mbox{~as~} t\to +\infty .
      \]
  \item[(ii)] For any $v\in X$ there exists a stationary point $z$ such that
   $S_tv\to z$ as $t\to +\infty$.
 \end{enumerate}
 \end{itemize}
\end{theorem}

\subsubsection{Dimension of global attractor}
\label{s7.3}
Finite-dimensionality is an important property of global  attractors
that can be  established for many  dynamical systems,  including
those arising in significant  applications.
There are several approaches that provide effective
estimates for the dimension of attractors
of dynamical systems generated by PDEs (see, e.g., \cite{BV92,Lad91,temam}).
Here we present  an approach
that  does not require
$C^1$-smoothness of the evolutionary operator
(as in  \cite{BV92,temam}).
The reason for this focus is that dynamical systems
of hyperbolic-like nature do not display smoothing effects,  unlike parabolic
equations. Therefore, the $C^1$ smoothness of the flows is
most often beyond   question, particularly in problems with a nonlinear dissipation.
Instead, we  present a method     which
can be applied to more general   locally Lipschitz flows.
This method generalizes
  Ladyzhenskaya's theorem (see, e.g.,
\cite{Lad91}) on finite
dimension of invariant sets.
We also refer to \cite{prazak}
for a closely related  approach
based on some kind of
 squeezing property.
 However, we wish to point out that
the estimates of the dimension
based on the theorem below usually tend to  be conservative.

\begin{definition}\label{de7.3.1}\index{fractal dimension}
Let $M$ be a compact set in a metric space $X$.
\begin{itemize}
\item
The {\bf fractal (box-counting) dimension} $\dim_fM$ of $M$ is defined by
\[
\dim_fM=\limsup_{\eps\to 0}\frac{\ln n(M,\eps)}{\ln (1/\eps)}\;,
\]
where $n(M,\eps)$ is the minimal number of closed balls of the
radius $\eps$ which cover the set $M$.
\end{itemize}
\end{definition}
We can also consider the Hausdorff dimension $\dim_H$ to describe  complexity
and  embeddings properties
 of compact sets.  We do not give a formal definition of this dimension characteristic
 (see, e.g., \cite{Fal90} for some details and references)
 and we  only note
 that (i) the Hausdorff dimension does not exceed (but is not equal, in general)
 the fractal one;
(ii) fractal dimension is more convenient in calculations.
\smallskip\par
The following result which generalizes \cite{Lad91}  was proved in \cite{cl-jdde}, see also
\cite{cl-mem}.

\begin{theorem}\label{t7.3.1a}
Let $H$ be a separable Banach space and
$M$ be a bounded closed set in $H$. Assume that
there exists a mapping $V\, :\, M\mapsto H$ such that
\begin{itemize}
\item[(i)] ${}\,{}$ $M\subseteq VM$.
\item[(ii)] ${}\;{}$
$V$ is Lipschitz on $M$; that is,  there exists $L>0$ such that
\begin{equation*}
\| Vv_1-Vv_2\|\le L  \| v_1-v_2\|,~~ v_1,v_2\in M.
\end{equation*}
\item[(iii)] ${}\;{}\;{}$ There exist compact seminorms $n_1(x)$ and $n_2(x)$
on $H$ such that
\begin{equation*}
\| Vv_1-Vv_2\|\le \eta  \| v_1-v_2\|+ K\cdot [ n_1(v_1-v_2)+ n_2(Vv_1-Vv_2)]
\end{equation*}
for any $v_1,v_2\in M$,
where $0<\eta <1$ and $K>0$ are constants (a seminorm $n(x)$ on $H$
is said to be compact iff $n(x_m)\to 0$ for any sequence
$\{x_m\}\subset H$ such that $x_m\to 0$ weakly in $H$).
\index{compact seminorm}
\end{itemize}
Then $M$ is a compact set in $H$ of a finite fractal dimension.
Moreover,
if $H$ is a Hilbert space and
the seminorms $n_1$ and $n_2$ have the form $n_i(v)=\| P_iv\|$,
$i=1,2$, where $P_1$ and $P_2$ are finite-dimensional
orthoprojectors, then
\begin{equation*}
\dim_f M\le
\left(
\dim P_1+\dim P_2\right)\cdot\ln\left( 1+
\frac{8(1+L)\sqrt{2}K}{1-\eta}\right)
\cdot\left[\ln\frac{2}{1+\eta}\right]^{-1}.
\end{equation*}
\end{theorem}

\begin{remark}
 We note that under the hypotheses of Theorem~\ref{t7.3.1a}
the mapping $V$  possesses the property (see Lemma 2.18 in \cite{cl-mem})
\[
\al(VB)\le \eta \al(B)~~\mbox{for any}~~B\subset M,
\]
where $\al(B)$ is the Kuratowski $\al$-measure given by \eqref{kurat}.
This means that $V$ is $\al$-contraction on $M$ (in the terminology of
\cite{Ha88}). The note that the latter property is not sufficient for finite-dimensionality of the set $M$.
Indeed, let
\[
X=l_2=\left\{ x=(x_1;x_2;\ldots )\, :\, \sum_{i=1}^\infty x^2_i<\infty
\right\}
\]
 and
 \[
M=\left\{ x=(x_1;x_2;\ldots )\in l_2\, :\, |x_i|\le i^{-2},~~ i=1,2,\ldots
\right\}
\]
We define a mapping $V$ in $X$ by the formula
\[
\big[Vx\big]_i= f_i(x_i),~~~ i=1,2,\ldots
\]
where $f_i(s)=s$ for $|s|\le i^{-2}$, $f_i(s)=i^{-2}$ for $s\ge i^{-2}$, and
$f_i(s)=-i^{-2}$ for $s\le -i^{-2}$. One can see that $V$ is globally Lipschitz
in $X$ and $VX=M=VM$. Since $M$ is a compact set, the mapping is
$\al$-contraction (with $\eta=0$). On the other hand it is
clear that $\dim_f M=\infty$.
\par
We also not that this example means that the statement of Theorem~2.8.1 in \cite{Ha88}
is not true and thus it requires some additional hypotheses concerning
the mapping.
\end{remark}

In what follows we shall present a unified criterion which allows to prove both finite-dimensionality and smoothness of attractors.  This criterion  involves a special class of systems that we call {\it quasi-stable}.

\subsection{Quasi-stable systems }\label{sec7.9-stab}
In this section following the presentation given in \cite[Section 7.9]{cl-book}
we introduce a class of dissipative dynamical systems which display rather
 special  long time behavior   dynamic properties. This class will be referred to as {\it quasi-stable systems} which
 are defined via {\it quasi-stability inequality}  given in Definition \ref{de:ms-stable}.
 It turns out that this is quite large class of systems  naturally occurring in
 nonlinear  PDE's
of hyperbolic type of  second order in time  possibly interacting with parabolic equation.
The interest in this class of systems stems from the fact that  {\it quasi-stability inequality}
almost automatically implies   -in one shot-  number of desirable properties such as
{\bf asymptotic smoothness, finite dimensionality of attractors, regularity of attractors, exponential  attraction} etc.
In what follows we shall provide brief introduction to the theory of quasi-stable systems.

We begin with the following assumption.
\begin{assumption}\label{ch7.9.A}
Let $X$, $Y$, and $Z$ be reflexive Banach spaces; $X$ is compactly embedded
in $Y$. We endow the space $H=X\times Y\times Z$ with the norm
\[
|y |^2_H=|u_0|^2_X+|u_1|^2_Y+|\theta|^2_Z,\quad y=(u_0;u_1;\theta_0).
\]
The trivial case $Z=\{0\}$ is allowed.
We assume that
$(H,S_t)$ is a dynamical system on
$H=X\times Y\times Z$   with the evolution operator
of the form
\begin{equation}\label{7.9.1}
S_ty=(u(t); u_t(t); \theta(t)), \quad y=(u_0;u_1;\theta_0)\in H,
\end{equation}
where the functions $u(t)$ and $\theta(t)$ possess the properties
\[
u\in C(\R_+, X)\cap C^1(\R_+, Y), \quad \theta\in C(\R_+,Z).
\]
\end{assumption}
The structure of the phase space $H$ and the evolution operator $S_t$ in
Assumption~\ref{ch7.9.A} is motivated by the study
 of the system generated by
equation of the second order in time in $X\times Y$
possibly interacting with some first-order evolution equation in space $Z$.
This type of interaction arises in modeling of  thermoelastic plates and structural acoustic sytems.
\begin{definition}\label{de:ms-stable}
\index{dynamical system!quasi-stable}
A dynamical system of the form \eqref{7.9.1} is said to be
stable modulo compact terms ({\bf quasi-stable}, for short) on a set $\cB\subset H$ if
there exist a compact seminorm $\mu_X(\cdot)$ on the space $X$ and
 nonnegative scalar functions $a(t)$, $b(t)$, and $c(t)$ on $\R_+$
such that
(i) $a(t)$ and  $c(t)$ are locally bounded on $[0,\infty)$,
(ii)~$b(t)\in L_1(\R_+)$ possesses
the property $\lim_{t\to\infty}b(t)=0$,
and  (iii) for every $y_1,y_2\in \cB$ and $t>0$ the following relations
\begin{equation}\label{8.4.1mc}
| S_ty_1-S_ty_2|^2_H\le  a(t)\cdot | y_1-y_2|^2_H
\end{equation}
and
\begin{equation}\label{8.4.2mc}
| S_ty_1-S_ty_2|^2_H\le  b(t)\cdot | y_1-y_2|^2_H+
c(t)\cdot \sup_{0\le s\le t}\left[ \mu_X(u^1(s)-u^2(s))\right]^2
\end{equation}
hold. Here we denote
$S_ty_i=(u^i(t); u^i_t(t);\theta^i(t))$, $i=1,2$.
\end{definition}
\begin{remark}{\rm
The definition of {\it quasi-stability} is rather natural from the point of view of
 long-time behavior. It pertains to decomposition of the
flow into exponentially stable and compact part.
This represents some sort of analogy with the ``splitting" method \cite{BV92,temam},
however,
the decomposition refers to the difference of two trajectories,
rather than a single trajectory. In addition, the {\it quadratic} dependence with respect to the semi-norm in (\ref{8.4.2mc})
is essential in  the definition.
The relation \eqref{8.4.2mc} is called a \emph{stabilizability estimate}
\index{stabilizability estimate}
(or {\it quasi-stability inequality})
\index{quasi-stability inequality}
and, in the context of long-time dynamics, was originally  introduced in \cite{cl-jdde}
(see also \cite{cel2,chla03} and the discussion in \cite{cl-mem}  and \cite{cl-book}).
To obtain such an estimate proves fairly technical (in critical problems)
and requires rather subtle PDE tools to prove it.
Illustrations of the method are given in \cite{cl-book} for some abstract models and
 also for a variety of von Karman models.
We also refer to
\cite{bu-ch0,bu-ch,cel2,cl-jdde,cl-mt,snowbird,cl-mem,clt-dcds08,clt-jdde09,naboka09}
for similar considerations for other models.
\par
 The notion of quasi-stability introduced in Definition~\ref{de:ms-stable}
requires  a special structure of the semiflow and a special type
of a (quasi-stability) inequality. However, the idea behind this notion
can be applied in many other cases (see, e.g., \cite{chla03,cl-mem,cl-kb2}
and also \cite{cl-book}.
Systems with delay/memory terms can be also included in this framework
(see, e.g., \cite{fastovska07,fastovska09,potomkin10,ryz05} and also \cite{cl-book}).
The same idea was recently applied
in \cite{kolbasin} for  analysis
of long-time dynamics in a parabolic-type model
(see below in Section~\ref{sect:other}).
}
\end{remark}
\begin{remark} {\rm
In order to write down  a more explicit form of quasi-stability inequality let us consider dynamical system (say of gradient form) generated by
some second order evolution equation in the space
$H=X\times Y$  with an associated
energy functional -say $E_y(t)$ and the damping integral
denoted by $D_y(t)$, so that the  energy identity  reads
$$
E_y(t) + \int_s^t D_y(\tau) d \tau = E_y(s), ~~ s < t.
$$
In this case
 a sufficient condition for quasi-stability is the following  \cite{cel2}{\it sta\-bi\-liza\-bi\-lity inequality}:
there exists a  parameter  $  T > 0 $,    such that
 the  difference of two any trajectories $z(t) \equiv y_1 (t) - y_2 (t) $
 satisfies the relation
\begin{equation}\label{stabil}
E_{0, z}(T) + \int_0^{T} E_{0,z} (\tau) d \tau \leq C_{T} \int_0^T \widetilde{D}_z(\tau) + C_{T} LOT_z,
\end{equation}
where $E_{0,z} (t) $ denotes a  positive part of the energy,
$\widetilde{D}_z$
is the damping functional for the difference $z$
and
$LOT_z$ is a lower order term of the form
$LOT_z = \sup_{\tau \in (0,T)}||z(\tau)||^2_{H_1}$,
 where $H \subset H_1 $ with a compact embedding. If one is already equipped with an existence of compact attractor, then
stabilizability inequality that is required is of milder form:
there exist two parameters $ -\infty < s < T $ such that
\begin{equation}\label{stabil-1}
E_{0, z}(T) + \int_s^{T} E_{0,z} (\tau) d \tau \leq C_{s,T} \int_s^T \widetilde{D}_z(\tau) d\tau + C_{s,T} LOT_z.
\end{equation}
Here $LOT_z$ denotes lower order terms measured by
\[
 LOT_z = \sup_{\tau \in (s,T)}||z(\tau)||^2_{H_1},
 \]
and parameters $s, T$ may depend on the  specific trajectories $y_1, y_2$ .
\par
The advantage of the above formulation is that these are inequalities that are obtainable by multipliers
when studying stabilization and controllability
(see, e.g., the monographs \cite{cmbs,l-tbook} and the references therein). The stabilizability inequality in (\ref{stabil}) provides a direct link
between controllability and long time behavior.
Indeed, when considering stabilization to zero equilibria, the lower order terms $LOT_z$ are dismissed by applying compactness-uniqueness argument. Then the inequality applied to a single trajectory $y_1 = y $
\begin{equation}\label{controll} E_{ y}(T)  \leq C_{0,T} \int_0^T D_y(\tau) = C_{0,T}[ E_{y}(0) - E_y(T) ] \end{equation}
implies, by standard evolution method,  exponential decay to zero equilibrium.

Similarly, in the case of controllability the observability inequality required  \cite{cmbs} is
precisely (\ref{controll}) with $D_y(t) $ denoting the quantities observed (for instance, velocity in the spatial domain or velocity on the subdomain). In either case, the crux of the matter is to establish such inequality.
}
\end{remark}

\par
In what follows our aim is to show that  quasi-stable systems have far reaching consequences and
enjoy many nice   properties that include  {\bf (i)  existence of global attractors
that is both finite-dimensional and smooth,
(ii) exponential attractors}, and so on.

\subsubsection{Asymptotic smoothness}
\par
\begin{proposition}\label{pr:7.9.1sc}
Let Assumption~\ref{ch7.9.A} be in force. Assume that the dynamical system
$(H,S_t)$ is quasi-stable on every bounded forward invariant set $\cB$ in $H$.
Then,  $(H,S_t)$ is asymptotically smooth.
\end{proposition}
\begin{proof}
Let
\begin{equation}\label{7.9.2}
\widetilde{X}=\mbox{Closure}\,\left\{v\in X\, :\; |v|_{\widetilde{X}}\equiv
\mu_X(v)+|v|_Y<\infty\right\}.
\end{equation}
One can see that $X$ is compactly embedded in the Banach space $\widetilde{X}$.
We have that the space
\[
W^1_{\infty,2}(0,T;X,Y)=\left\{ f\in L_\infty(0,T;X)\, :\;  f'\in L_2(0,T;Y)\right\}
\]
is compactly embedded in $C(0,T;\widetilde{X})$.
This implies that
the  pseudometric $\varrho^t_{\cB}$ in $C(0,t;H)$
defined by the formula
\[
\varrho^t_{\cB}(\{S_\tau y_1\},\{S_\tau y_2\})=
c(t) \sup_{\tau\in [0, t]}
\mu_X(u^1(t)-u^2(t))
\]
is precompact (\wrt$H$).  Here we  denote by
$\{S_\tau y_i\}$ the element from  $C(0,t; H)$ given
by function $y_i(\tau)=S_\tau y_i\equiv (u^i(t);u_{t}^i(t);\theta^i(t))$.
By \eqref{8.4.2mc} $\varrho^t_{\cB}$ satisfies the hypotheses
of  Theorem~\ref{th7.1.1}.
This implies the result.
\end{proof}
\begin{corollary}\label{co:pr7.9.1}
If the system $(H,S_t)$ is dissipative and satisfies  the hypotheses of
Proposition~\rref{pr:7.9.1sc}, then it possesses a compact global attractor.
\end{corollary}
\index{dynamical system!quasi-stable!global attractor}
\begin{proof}
By Proposition~\ref{pr:7.9.1sc} the system $(H,S_t)$ is asymptotically smooth. \!Thus
the result follows from Theorem~\ref{t7.2.1-as}.
\end{proof}
\subsubsection{Finite dimension of global attractors}
We start with the following general assertion.
\begin{theorem}\label{th7.9dim} Let Assumption~\rref{ch7.9.A} be valid.
Assume that the dynamical system $(H,S_t)$ possesses a compact
global attractor $A$ and is quasi-stable on  $A$ (see Definition
\rref{de:ms-stable}). Then  the  attractor $A$  has a finite
fractal dimension $dim_f^H A$ (this also
 implies the finiteness of the Hausdorff dimension).
\end{theorem}
\index{dynamical system!quasi-stable!dimension of global attractor}
\begin{proof}
The idea of the proof is based on the method of ``short" trajectories
(see, e.g., \cite{malek-ne,malek} and the references therein and
also \cite{cl-mem}).
\par
We apply Theorem~\ref{t7.3.1a} in the space $H_T=H\times W_1(0,T)$
with an appropriate $T$. Here
\begin{equation}\label{8.4.2a}
W_1(0,T)=\left\{ z\in L_2(0,T; X)  : |z|^2_{W_1(0,T)}\equiv\!
\int_0^T\!\left( |z(t)|_X^2+ |z_t(t)|_Y^2\right) dt<\infty\right\}.
\end{equation}
The norm in $H_T$ is given by
\begin{equation}\label{8.4.2b}
\| U\|^2_{H_T}= |y|_{H}^2+|z|^2_{W_1(0,T)},~~
U=(y;z),\; y=(u_0; u_1; \theta_0).
\end{equation}
Let $y_i=(u^i_0;  u^i_1; \tht_0^i)$, $i=1,2$, be two elements from the
attractor $A$. We denote
\[
S_ty_i=(u^i(t);u^i_t(t); \tht^i(t)),\quad t\ge 0,\; i=1,2,
\]
and $Z(t)=S_ty_1-S_ty_2\equiv (z(t); z_t(t); \xi(t))$, where
\[
(z(t); z_t(t); \xi(t))\equiv
(u^1(t)- u^2(t);u^1_t(t)- u^2_t(t); \tht^1(t)-\tht^2(t)).
\]
 Integrating (\ref{8.4.2mc})  from $T$ to $2T$ with respect to $t$,
we obtain that
\begin{equation}\label{8.4.4}
\int_T^{2T}| S_ty_1-S_ty_2|^2_Hdt  \leq
\widetilde{b}_T |y_1-y_2|^2_H  +
 \widetilde{c}_T  \sup_{0\le s\le 2T }
\left[\mu_X (z(s))\right]^2,
\end{equation}
where
\[
\widetilde{b}_T=   \int_T^{2T}b(t)dt~~ \mbox{and} ~~
\widetilde{c}_T=   \int_T^{2T}c(t)dt.
\]
It also follows from (\ref{8.4.2mc}) that
\[
| S_Ty_1-S_Ty_2|^2_H\le  b(T)\cdot | y_1- y_2|^2_H +
c(T)\cdot \sup_{0\le s\le T}\left[\mu_X (z(s))\right]^2
\]
and combining with  (\ref{8.4.4}) yields
\begin{equation}\label{8.4.5}
| S_Ty_1 - S_Ty_2|^2_H   +
\int_T^{2T}| S_ty_1-S_ty_2|^2_Hdt
\le
b_T |y_1- y_2|^2_H   +
c_T  \sup\limits_{0\le s\le 2T }
\left[\mu_X (z(s))\right]^2,
\end{equation}
where
\begin{equation}\label{7.9.3}
b_T=   b(T)+ \int_T^{2T}b(t)dt ~~ \mbox{and} ~~
c_T= c(T)+\int_T^{2T}c(t)dt.
\end{equation}
Let $A$ be the global attractor. Consider in the space $H_T$
the set
\[
A_T:=\left\{ U\equiv (u(0); u_t(0);\tht(0); u(t), t\in [0,T])\, :\,
(u(0); u_t(0);\tht(0))\in A\right\},
\]
where $u(t)$ is the first component  of $S_ty(0)$ with $y(0)=(u(0); u_t(0);\tht(0))$,
and define operator $V\,:\, A_T\mapsto H_T$ by the formula
\[
 V\, :\, (u(0); u_t(0);\tht(0); u(t))\mapsto (S_Ty(0); u(T+t)).
\]
It is clear that $V$ is Lipschitz on $A_T$ and $VA_T=A_T$.
\par
\Since the space $\widetilde{X}$ given by \eqref{7.9.2}
possesses the properties
\[
X\subset \widetilde{X}\subset Y~~\mbox{ and }~~
X\subset \widetilde{X}~\mbox{is compact},
\]
contradiction argument yields
\begin{equation}\label{7.9mu}
\left[\mu_X(u)\right]^2\le \eps |u|_X^2+ C_\eps |u|_Y^2~~\mbox{ for any }~~
\eps>0.
\end{equation}
Therefore it follows from (\ref{8.4.1mc}) that
\[
\sup_{0\le s\le 2T }
\left[\mu_X (z(s))\right]^2\le \frac{b_T}{c_T} |y_1-y_2|^2_H + C(a_T,b_T,c_T)
\sup_{0\le s\le 2T }
|z(s)|_Y^2,
\]
where $a_T=\sup_{0\le s\le 2T }a(s)$ and $b_T,c_T$ given by \eqref{7.9.3}.
Consequently, from (\ref{8.4.5}) we obtain
\[
\|VU_1-VU_2\|_{H_T}\le \eta_T \|U_1-U_2\|_{H_T} +K_T\cdot ( n_T(U_1-U_2)+
 n_T(VU_1-VU_2)),\]
for any $ U_1,U_2\in A_T$,
where   $K_T>0$ is a constant
(depending on $a_T,b_T,c_T$,  the embedding properties
of $X$ into $Y$,  the seminorm $\mu_X$),
and
\begin{equation}\label{8.4.5a}
\eta^2_T= b_T=  b(T)+ \int_T^{2T}b(t)dt.
\end{equation}
The seminorm $n_T$ has the form $n_T(U):=\sup_{0\le s \le T }| u(s)|_Y$.
\Since $W_1(0,T)$ is compactly embedded into
$C(0,T;Y)$, $n_T(U)$ is a compact seminorm on $H_T$ and we can choose $T>1$
 such that $\eta_T<1$. We also have from \eqref{8.4.1mc} that
 \[
 \|VU_1-VU_2\|_{H_T}\le L_T \|U_1-U_2\|_{\cH_T}~~\mbox{ for }~~U_1,U_2\in A_T,
\]
where $L_T^2=a(T)+ \int_T^{2T}a(t)dt$.
 Therefore we can
apply Theorem~\ref{t7.3.1a} which
 implies that $A_T$ is a compact set in $H_T$ of finite fractal dimension.
\par
Let  ${\cal P}\,:\, H_T\mapsto H$
be the operator defined
by the formula
\[
 {\cal P}\, :\, (u_0; u_1; \theta_0; z(t))\mapsto (u_0; u_1; \theta_0).
\]
$A={\cal P} A_T$ and ${\cal P}$ is Lipshitz continuous, thus
$
\mbox{dim}_{f}^{H}A\le
\mbox{dim}_{f}^{H_T}A_T<\infty.
$
Here $\mbox{dim}_{f}^{W}$ stands for the fractal dimension of a set in
the space $W$.
This concludes the proof of Theorem~\ref{th7.9dim}.
\end{proof}
\begin{remark}\label{re7.9-dim}
{\rm
By \cite{cl-mem} the dimension
  $\mbox{dim}_f^H A$ of the attractor admits the estimate
\begin{equation}\label{dim-bound}
\mbox{dim}_{f}^{H} A\le
\left[\ln\frac{2}{1+\eta_T}\right]^{-1}\cdot
\ln\, m_0\left(
\frac{4K_T (1+L_T^2)^{1/2}}{1-\eta_T}\right),
\end{equation}
Here  $m_0(R)$ is the maximal number  of pairs $(x_i;y_i)$ in
$H_T\times H_T$ possessing the properties
\[
\|x_i\|^2_{H_T}+\|y_i\|^2_{H_T}\le R^2, \quad n_T(x_i-x_j)+n_T(y_i-y_j)>1,\quad
i\neq j.
\]
It is  clear that $m_0(R)$ can be estimated by the maximal number  of pairs $(x_i;y_i)$ in
$W_1(0,T)\times W_1(0,T)$ possessing the properties
$
\|x_i\|^2_{W_1(0,T)}+\|y_i\|^2_{W_1(0,T)}\le R^2$ and $n_T(x_i-x_j)+n_T(y_i-y_j)>1$
for all $i\neq j$.
Thus the bound in \eqref{dim-bound}  depends on
the functions $a$, $b$, $c$ and seminorm $\mu_X$ in Definition~\ref{de:ms-stable}
and also on the embedding properties of $X$ into $Y$.
}
\end{remark}
\subsubsection{Regularity of trajectories from the attractor}
In this section we show how  stabilizability estimates can be used  in order
to obtain additional
regularity of trajectories lying on the  global attractor.
The theorem below provides regularity for time derivatives. The needed ``space"
regularity follows from the analysis of the respective PDE.
It typically involves application of elliptic theory
(see the corresponding  results in \cite[Chapter~9]{cl-book}).
\begin{theorem}\label{th7.9reg}
Let Assumption~\rref{ch7.9.A} be valid.
Assume that the dynamical system $(H,S_t)$ possesses a compact global attractor
$A$ and is quasi-stable on the  attractor $A$.
Moreover, we assume that \eqref{8.4.2mc} holds with the function
$c(t)$ possessing the property $c_\infty=\sup_{t\in\R_+}c(t)<\infty$.
 Then any full trajectory $\{ (u(t);u_t(t);\theta(t))\, :\, t\in\R\}$
that belongs to the global attractor enjoys
 the following regularity  properties,
\begin{equation}\label{7.9reg1}
u_t \in  L_\infty(\R; X)\cap C(\R; Y),  \quad u_{tt} \in  L_\infty(\R; Y),
 \quad \tht_{t} \in  L_\infty(\R; Z)
\end{equation}
Moreover, there exists $R>0$ such that
\begin{equation}\label{7.9reg2}
| u_t(t)|_X^{2}+ | u_{tt}(t)|_Y^{2} +| \tht_t(t)|_Z^{2}  \le  R^{2}, \quad t\in\R,
\end{equation}
where $R$ depends on the constant $c_\infty$, on the seminorm $\mu_X$
in Definition~\rref{de:ms-stable},
and also on the embedding properties of $X$ into $Y$.
\end{theorem}
\index{dynamical system!quasi-stable!smoothness of trajectories from attractor}
\begin{proof}
It follows from     \eqref{8.4.2mc}   that for any two trajectories
\begin{eqnarray*}
\gamma &= & \{U(t)\equiv (u(t);u_t(t);\theta(t))\, :\, t\in\R\},
\\
\gamma^* &= & \{U^*(t)\equiv (u^*(t);u^*_t(t);\theta^*(t))\, :\, t\in\R\}
\end{eqnarray*}
from the global attractor we have that
\begin{equation}\label{8z-ineq-atr}
|Z(t)|^2_H  \leq   b(t-s) |Z(s)|^2_H+
  c(t-s) \sup_{s\le \tau \leq t }  \left[ \mu_X(z(\tau))\right]^2
\end{equation}
for all $s\le t$, $s,t\in\R$, where $Z(t)=U^*(t)-U(t)$ and  $z(t)=u^*(t)-u(t)$.
In the limit $s\to-\infty$ relation \eqref{8z-ineq-atr} gives us that
\[
|Z(t)|^2_H  \leq    c_\infty \sup_{-\infty\le \tau \leq t }
 \left[ \mu_X(z(\tau))\right]^2
\]
for every $t\in\R$ and for every couple of trajectories $\gamma$ and $\gamma^*$.
Using relation \eqref{7.9mu}   we can conclude that
\begin{equation}\label{8z-ineq-atr3}
 \sup_{-\infty\le \tau \leq t } |Z(\tau)|^2_H  \leq
 C\sup_{-\infty\le \tau \leq t }   |  z(\tau)|_Y^2,
\end{equation}
for every $t\in\R$ and for every couple of trajectories $\gamma$ and $\gamma^*$
from the attractor.
\par
Now we fix the trajectory $\gamma$ and
for $0<|h|<1$ we consider the shifted trajectory
$\gamma^*\equiv\gamma_h= \{y(t+h)\, :\, t\in\R\}$.
Applying \eqref{8z-ineq-atr3} for this pair  of trajectories
and using the fact that all terms  \eqref{8z-ineq-atr3}
are quadratic \wrt $Z$ we obtain that
\begin{equation}\label{8z-ineq-atr4}
 \sup_{-\infty\le \tau \leq t }\left\{ |u^h(\tau)|^2_X+ |u^h_t(\tau)|^2_Y+
|\tht^h_t(\tau)|^2_Z\right\}   \leq
 C\sup_{-\infty\le \tau \leq t }   |  u^h(\tau)|_Y^2,
\end{equation}
where $u^h(t)= h^{-1}\cdot [u(t+h)-u(t)]$ and
$\tht^h(t)= h^{-1}\cdot [\tht(t+h)-\tht(t)]$. On the attractor we obviously
have that
\[
|u^h(t)|_Y\le\frac1h\cdot\int_0^h |u_t(\tau+t)|_Y d\tau\leq C,\quad t \in \R ,
\]
with uniformity in $h$. Therefore  \eqref{8z-ineq-atr4} implies
that
\[
|u^h(t)|^2_X+ |u^h_t(t)|^2_Y+
|\tht^h_t(t)|_Z^2 \leq C,\quad t \in \R.
\]
Passing with the limit on $h$ then yields relations \eqref{7.9reg1} and
\eqref{7.9reg2}.
\end{proof}
\subsubsection{Fractal exponential attractors}
For quasi-stable systems we can also establish some result
pertaining to   (generalized) fractal exponential attractors.
\par
We first recall the following definition.
\begin{definition}\label{de7.3.2}\index{inertial set}
\index{fractal exponential attractor}
\index{attractor!fractal exponential}
A compact set $A_{\rm exp}\subset X$
is said to be  inertial (or a {\bf fractal exponential
attractor})
of the dynamical system $(X, S_t)$  iff  $A$ is a positively invariant set
of finite fractal dimension (in $X$) and
for every bounded set $D\subset X$ there exist positive constants
$t_D$, $C_D$ and $\gamma_D$ such that
\[
d_X\{S_tD\, |\, A_{\rm exp}\}\equiv
\sup_{x\in D} \dist_X (S_tx,\, A_{\rm exp})\le C_D\cdot e^{-\gamma_D(t-t_D)},
\quad t\ge t_D.
\]
If the dimension of $A$ is finite in some extended space we
call $A$ {\em generalized} fractal exponential
attractor.
\end{definition}
For more  details  concerning
 fractal exponential attractors
we refer to \cite{EFNT94} and also to recent survey \cite{MirZel-08}.
We only note that the standard technical tool in the construction of
 fractal exponential attractors
is the so-called squeezing property which says \cite{MirZel-08}, roughly
speaking, that either the higher modes are dominated by the lower ones or that the
semiflow is contracted exponentially. Instead the approach we use is  based
on the quasi-stability property which says that
the semiflow is asymptotically contracted   up to a homogeneous compact additive term.

\index{dynamical system!quasi-stable!fractal exponential attractor}
\begin{theorem}\label{th7.9exp}
Let Assumption~\rref{ch7.9.A} be valid.
Assume that the dynamical system $(H,S_t)$  is dissipative
and  quasi-stable on some bounded absorbing set $\cB$.
We also assume that there exists a space $\widetilde{H}\supseteq H$ such that
$t\mapsto S_ty$ is H\"{o}lder continuous in $\widetilde{H}$ for every $y\in\cB$;
that is, there exist $0<\gamma\le 1$ and $C_{\cB,T}>0$ such that
\begin{equation}\label{Hold}
|S_{t_1}y-S_{t_2}y|_{\widetilde{H}}\le C_{\cB,T} |t_1-t_2|^\gamma,\quad
t_1,t_2\in [0,T],\; y\in\cB.
\end{equation}
Then the dynamical system $(H,S_t)$  possesses a (generalized) fractal exponential attractor
whose dimension is finite in the space $\widetilde{H}$.
\end{theorem}

In contrast with Theorem~3.5 \cite{MirZel-08}    Theorem~\ref{th7.9exp} does
not assume H\"{o}lder continuity in the phase space. At the expense of this
we can guarantee finiteness of the dimension in some extended space only.
This is  why we  use the notion of  {\em generalized } exponential attractors.

\begin{proof}
We apply the same idea as in the proof of Theorem~\ref{th7.9dim}.
\par
We can assume that absorbing set   $\cB$ is closed and forward invariant
(otherwise, instead of $\cB$, we consider
$\cB'=\overline{\cup_{t\ge t_0}S_t\cB}$ for $t_0$ large enough, which lies in $\cB$).
\par
In the space $H_T=H\times W_1(0,T)$ equipped with the norm
(\ref{8.4.2b}), where $W_1(0,T)$ is given by
(\ref{8.4.2a}), we consider the set
\[
\cB_T:=\left\{ U\equiv (u(0); u_t(0);\tht(0); u(t), t\in [0,T])\, :\,
(u(0); u_t(0);\tht(0))\in \cB\right\},
\]
where $u(t)$ is the first component  of $S_ty(0)$ with $y(0)=(u(0); u_t(0);\tht(0))$.
As above we
 define the shift operator $V\,:\, \cB_T\mapsto H_T$ by the formula
\[
 V\, :\, (u(0); u_t(0);\tht(0); u(t))\mapsto (S_Ty(0); u(T+t)).
\]
It is clear that $\cB_T$ is a closed bounded set in $H_T$ which is
forward invariant \wrt $V$.
\par
It follows from (\ref{8.4.1mc}) that
\[
\|VU_1-VU_2\|^2_{H_T}\le \left( a(T)+\int_T^{2T}a(t)dt\right)
 \|U_1-U_2\|^2_{H_T},\quad U_1,U_2\in \cB_T.
\]
As in the proof of Theorem~\ref{th7.9dim} we can obtain that
\[
\|VU_1-VU_2\|_{H_T}\le \eta_T \|U_1-U_2\|_{H_T} +
K_T\cdot ( n_T(U_1-U_2)+ n_T(VU_1-VU_2))
\]
for any $ U_1,U_2\in \cB_T$ and for some $T>0$,
where   $K_T>0$ is a constant, $n_T(U):=\sup_{0\le s \le T }| u(s)|_Y$
and $\eta_T<1$ is given by (\ref{8.4.5a}).
Therefore by Theorem~7.4.2\cite{cl-book} the mapping $V$ possesses a fractal
exponential attractor; that is,  there exists a compact set $\cA_T\subset\cB_T$
and a number $0<q<1$ such that $\dim_f^{H_T}\cA_T<\infty$, $V\cA_T\subset\cA_T$,
and
\[
\sup\left\{ {\rm dist}_{H_T} (V^kU, \cA_T)\; :\; U\in\cB_T\right\}\le C q^k,
\quad k=1,2,\ldots,
\]
for some constant $C>0$.
In particular, this relation implies that
\begin{equation}\label{exp-atr1}
\sup\left\{ {\rm dist}_{H} (S_{kT}y, \cA)\; :\; u\in\cB\right\}\le C q^k,
\quad k=1,2,\ldots,
\end{equation}
where $\cA$ is the projection of $\cA_T$ of the first  components:
\[
\cA=\left\{ (u(0);u_t(0);\tht(0))\in \cB\; :\;
(u(0); u_t(0);\tht(0); u(t), t\in [0,T])\in \cA_T\right\}.
\]
It is clear that $\cA$ is a compact
forward invariant set
\wrt $S_T$; that is, $S_T\cA\subset\cA$. Moreover $\dim^H_f\cA\le\dim^{H_T}_f\cA_T<\infty$.
One can also see that
\[
A_{\rm exp}=\cup\left\{ S_t\cA\, :\; t\in [0,T]\right\}
\]
 is a compact forward invariant set
\wrt $S_t$; that is,  $S_tA_{\rm exp}\subset A_{\rm exp}$.
Moreover, it follows from
(\ref{Hold}) that
\[
\dim_f^{\widetilde{H}}A_{\rm exp}\le c\left[ 1+\dim^H_f\cA\right]<\infty.
\]
We also have from (\ref{exp-atr1}) and \eqref{8.4.1mc} that
\[
\sup\left\{ {\rm dist}_{H} (S_{t}y, A_{\rm exp})\; :\; u\in\cB\right\}\le
C e^{-\gamma t},\quad t\ge 0,
\]
for some $\gamma>0$. Thus $A_{\rm exp}$ is a (generalized) fractal exponential attractor.
\end{proof}
\begin{remark}\label{re:7.9hldr}{\rm
   H\'older continuity (\ref{Hold}) is needed in order
 to derive the finiteness
  of the fractal dimension $\dim_f^{\widetilde{H}}A_{\rm exp}$
  from the finiteness of $\dim^H_f\cA$. We do not know whether
  the same holds true without property (\ref{Hold})
imposed in some vicinity of $A_{\rm exp}$. This is  because $A_{\rm exp}$
is a {\em uncountable} union of (finite-dimensional) sets $S_t\cA$.
We also emphasize that fractal dimension depends on the topology.
Indeed, there is an example
 of a set with finite fractal dimension in one space and
infinite fractal dimension in another (smaller) space.
}
\end{remark}


\section[Long time behavior for canonical models]{Long time behavior for canonical models described in  Section~\ref{sect1.2}}
 The goal of this section is to show how to apply abstract methods presented in the previous section in order to
 establish long time behavior properties of dynamical systems  generated by PDE's described in Section~\ref{sect1.2}.
 An interesting feature is that though all the models considered are of hyperbolic type-without  an inherent smoothing mechanism, the long time behavior  can be made   "smooth" and  characterized by finite dimensional structures.
 Our plan is to focus on the following topics.
\begin{enumerate}
\item
Control to finite dimensional and smooth attractors of {\it nonlinear wave} equation
\begin{itemize}
\item
with interior  fully supported dissipation,
\item
with boundary or partially localized dissipation.
\end{itemize}
\item
 Control to finite dimensional and smooth attractors of  {\it nonlinear plates} dynamics
\begin{itemize}
\item
von Karman plates with a  {\it nonlinear  fully supported interior} feedback control,
\item
 von Karman and Berger plates  with boundary and partially localized damping,
 \item
 Kirchhoff-Boussinesq plates with a  {\it linear interior feedback}  control.
\end{itemize}
\end{enumerate}
For each model considered  we  follow the  plan: (i) state the assumptions and the results, (ii) discuss
possible generalizations, (iii) state open problems,
(iv)~provide a sketch of the proofs.

\subsection{Wave dynamics}
In this section we consider  the existence problem
for finite dimensional and smooth attractors for semilinear wave equation with critical
(feedback) sources with both {\it interior  and boundary damping}.
\par
We start with interior damping and source model as described in (\ref{1.1}).
\subsubsection{Interior damping}
 \index{wave equation!interior  damping}
In studying long time behavior we find convenient to collect all the assumptions required for the source and damping:
\begin{assumption}\label{fg}
\begin{enumerate}
\item
The source $f\in C^1(\R)$ satisfies the dissipativity condition in \eqref{dis-f-wave} and the growth condition
$$ |f'(s)|\leq M |s|^2~~ for~~  |s| \geq 1. $$
\item
The damping $g(s) $ is increasing function, $g(0)=0$ and it satisfies the following asymptotic  growth conditions:
\begin{equation}\label{g-attractor}
0< m_g  |s|^2 \le g(s) s \leq M_g |s|^6~~ for~~  |s| \geq 1.
\end{equation}
\item The damping coefficient $a$ satisfies $a(x) \geq a_0 > 0 ,
x\in \Omega $  .
\end{enumerate}
\end{assumption}
 With reference to the model (\ref{1.1}) we can take
  \begin{equation}\label{ex:f-g}
    f(s) = -s^3 + cs^2~~\mbox{and}~~g(s) = g_1 s + |s|^4s~~
    \mbox{with  $g_1 \geq 0 $ and $c \in \R$.}
  \end{equation}

\begin{remark}{\rm
1.The power $3$ associated with the source $f$ and the power
$5$ associated with the damping $g$ are
often referred as "double critical" parameters for the wave equation with $n =3$.
The reason for this is that the maps
$u \rightarrow f(u) $ from $H^1(\Omega)$ into $L_2(\Omega) $
and $u \rightarrow g(u) $ from $H^1(\Omega)$ into $H^{-1}(\Omega)$
are bounded but {\it not compact.}\\
2. More general forms  of space dependent  and localized   damping coefficient $a(x)$ are considered in \cite{clt-dcds08}
}
\end{remark}

Under the above conditions the wellposedness results of Theorem \ref{th:wp-wave1} assert existence of a continuous semiflow $S_t$  which defines dynamical system $(\cH, S_t) $ with $\cH= H_0^1(\Omega) \times L_2(\Omega) $.
We are ready to undertake the study of long-time behavior.
Our first main result is given below:
\begin{theorem}[Compactness]\label{T:1}${}\!\!\!{}$ \!
With reference to the dynamics described by (\ref{1.1}) and zero Dirichlet boundary conditions,
   we assume  Assumption \ref{fg} for the source $f$ and the damping $g$.
Then, there exists a global compact attractor
$\cA\subset \cH$ for the dynamical system $( \cH, S_t)$.
In addition $( \cH, S_t)$ is gradient system.
\end{theorem}
To describe the structure of the attractor, we introduce the set
of stationary points of $S_t$ denoted by $\cN$,
\begin{equation}\label{N-set-wave}
\cN=\left\{ V\in\cH\; :\; S_tV=V\mbox{~for all~}
t\ge 0\right\}.
\end{equation}
Every stationary point $W\in\cN$ has the form $W=(w;0) $, where
$w=w(x)$
solves the problem
\begin{equation}\label{eqn-111}
\Delta w + f(w) = 0 \mbox{ in }  \Omega \mbox{ and }
 w  = 0 \mbox{ in }  \Gamma.
\end{equation}
Under the Assumption \ref{fg}, standard elliptic theory implies that
every stationary solution satisfies $w \in H^{2} (\Omega)$
 and the set of all stationary
solutions
is bounded in $H^1(\Om)$.
In fact, more regularity of stationary solutions  can be shown,
but the above is sufficient for the analysis to follow.

\par
The next result is a consequence of  general properties
of gradient systems (see, e.g., \cite{BV92,Ha88} and also Section~\ref{grad-ds}
above) and  asserts
that the attractor $\cA$ coincides with this unstable manifold.
\begin{theorem}[{\bf Structure}]\label{T:3}
Under the assumptions of Theorem \ref{T:1} we have that
\begin{itemize}
\item
$\cA=M^u(\cN)$;
\item
$\lim_{t\to +\infty}{\rm dist}_{\cH}(S_tW,\cN)=0$ for any
$W\in\cH\equiv H^1(\Omega) \times L_2(\Omega)$;
\item if  (\ref{eqn-111}) has
a finite number of solutions\footnote{We note that
the property that  (\ref{eqn-111}) has finitely many solutions is {\em generic}.}, then
for any $V\in\cH$ there exists a stationary point $Z=(z,0)\in\cN$
such that $S_tV\to Z$ in $\cH$ as     $t\to +\infty$.
\end{itemize}
\end{theorem}

\smallskip\par
Our second main result read  refers to the dimensionality of attractor.
\begin{theorem}[\bf Finite Dimensionality]\label{Main}
Let $f\in C^2(\R)$, $g\in C^1(\R)$ and  Assumption \ref{fg} be in force.
In addition, we assume that for  all $s \in \R$:
\begin{enumerate}
\item $|f''(s)| \leq C |s|$,
\item
$ 0 < m \leq g'(s) \leq C[ 1+sg(s)]^{2/3}$.
\end{enumerate}
Then the fractal dimension of
the global attractor $\cA$ of the dynamical system
$(\cH, S_t) $ is finite.
\end{theorem}
With reference to the model (\ref{1.1}) we still can take
$f$ and $g$ as in \eqref{ex:f-g}, but  with $g_1 >0$.

\par
 The proof of Theorem~\ref{Main}
follows from quasi-stability property satisfied for the model
(see Definition \ref{de:ms-stable}).
The exactly  same lemma allows  to obtain "for free"  the following regularity result
for elements from the attractor.

\begin{theorem}[\bf Regularity]\label{sm0}
Under the assumptions of Theorem \ref{Main} the attractor is  a bounded  set in
\[
V = \left\{ (u,v)\in H_0^1(\Omega) \times H_0^1(\Omega)\,  \left|
\begin{array}{c}
-\Delta  u  +  g(v) \in L_2(\Omega),\\
\end{array}
\right.
 \right\},
\]
i.e. there exist constants $c_i>0$ such that\footnote{See also
Theorem \ref{sm}  below which asserts an  additional regularity of attractor.
}
\[
\cA\subset \left\{ (u,v)\in H^1(\Omega) \times H^1(\Omega)\, \left|
\begin{array}{c}
\|u\|_{1,\Om}+ \|v\|_{1,\Om}\le c_1; \\
\|\Delta  u  -  g(v)\|\le c_2,\;
\end{array}
\right. \right\}.
\]
\end{theorem}

A natural question that can be asked in this context is that of existence
of {\it strong} (i.e., corresponding to topology of strong solutions) attractors.
Though this  latter property is technically related  to smoothness of attractors,
however the corresponding result  does not follow  from Theorem \ref{sm0}, unless the damping is
subcritical. Additional analysis is needed for  that.
A first question to address in this direction is the  existence of attractors for {\it strong solutions}.
This is to say that when restricting solutions to regular initial data, the corresponding trajectories
converge asymptotically to an  attractor $\cA_1$ which typically may be smaller than $\cA$ the latter corresponds
to weak or generalized solutions. A first step toward this goal is to show dissipativity of strong solutions
(which of course does not follow at once  from dissipativity of weak solutions, unless a  problem is linear).
Fortunately, dsissipativity of strong solution is, again, direct consequence of quasi-stability.
Thus, we have  this property for "free".
\begin{theorem}[\bf Dissipativity of strong solutions]\label{th:strong}
Let the assumptions
of  Theorem \ref{Main} be in force.
Then there exists a number $R_0>0$ such that for any $R>0$ we can find $t_R>0$
such that
\begin{equation}\label{str-dis}
\|w_{tt}(t)\|^2+\|w_{t}(t)\|_{1,\Om}^2\le R_0^2\quad\mbox{for all}\quad t\ge t_R
\end{equation}
for any strong solution $w(t)$ to problem \eqref{1.1} with
initial data $(w_0;w_1)$ from the set
\begin{equation}\label{W-R}
\cW_R = \left\{ (w_0,w_1)\in H_0^1(\Omega) \times H_0^1(\Omega)\, :\;
\|w_0\|_{1,\Om}+ \|w_1\|_{1,\Om}\le R,\;
\|w_{2}\|\le R
 \right\},
\end{equation}
where $w_2\equiv g(w_1) -\Delta w_0 $.
\end{theorem}
Theorem \ref{th:strong} refers to   strong solutions.
The existence of such, is guaranteed
once initial data  are taken from
 $\cW_R$.
With additional calculations   (using the  multiplier $\Delta w_t $)
 one improves the statement
 and obtains the dissipativity for $\|w(t)\|^2_{2,\Omega} \leq R_0^2$, $t \geq t_R $ with initial data from $ H^2(\Omega)$.
 This, in particular, shows that  strong attractors (for strong solutions)
 are in  $H^2(\Omega) \cap H^1_0(\Omega) \times H_0^1(\Omega)$. An interesting question is whether
 the same regularity is enjoyed by {\it weak} attractors. In other words,  whether
 $H^2(\Omega) \cap H^1_0(\Omega) \times H_0^1(\Omega)  \supset \cA $?  The answer to this question is given by
 \begin{theorem}[\bf Regularity]\label{sm}
Under the same assumptions as in Theorem \ref{Main}
the attractor is a bounded set in $H^2(\Omega) \times L_2(\Omega)$.
\end{theorem}
The proof of this theorem given in \cite{A-K} exploits the so called
"backward smoothness on attractor".
\medskip\par\noindent
 {\bf  Decay rates to equilibrium:} Knowing that  every trajectory converges
 to some equilibrium point one would like to know "how fast"?
 The answer to this question is given by decay rates of asymptotic convergence of solutions to point of equilibria.
 This is the topic we present next.

We introduce concave, strictly increasing, continuous  function $ k_0 : \R_+
\mapsto \R_+ $  which captures the behavior of $g(s)$ at the
origin possessing the properties
\begin{equation}\label{h-0}
k_0(0) =0\quad\mbox{and}\quad s^{2} + g^2 (s) \leq k_0 (s g(s) )
\quad\mbox{for}\quad |s| \leq 1.
\end{equation}
Such a function can always be
constructed due to the monotonicity of $g$,  see \cite{las-tat} or Appendix B in \cite{cl-book}.
Moreover, on the strength of  Assumption \ref{fg}  there exists a constant $c>0$
such that $k(s) \equiv k_0(s) + c s $ is a
concave, strictly increasing, continuous  function $ k : \R_+
\mapsto \R_+ $ such that
\begin{equation}\label{kkk}
k(0) =0\quad\mbox{and}\quad |s|^{2}  \leq k (s g(s))
\quad\mbox{for all}\quad s\in\R.
\end{equation}
Given function $k$ we define
\begin{align}\label{q-G0}
H_0 (s)  = &  k\left( \frac{s}{c_3}\right), \
G_0(s) = c_1 (I + H_0 )^{-1} (c_2s), \
\\
Q(s)  = &
s - (I + G_0 )^{-1}(s), \nonumber
\end{align}
where $c_i$  are some positive
constants. It is obvious that
$Q(s)$ is strictly monotone.  Thus, the  differential
equation
\begin{equation}\label{ode1}
\frac{d \sigma}{dt} + Q (\sigma) =0,\quad t > 0,\quad
\sigma(0) = \sigma_0 \in \R,
\end{equation}
admits  global, unique  solution $\sigma(t) $ which, moreover,  decays
asymptotically to zero
as $t \rightarrow \infty $.
With these definitions we are ready to state our  next result.
\begin{theorem}[\bf Rate of convergence to equilibria]\label{Main1}
In addition to  Assu\-m\-ption \ref{fg}, we assume that  the set of stationary points $\cN$ is finite, and every
equilibrium $V=(v;0)$ is hyperbolic in the  sense that the
problem
\begin{equation}\label{hyp1}
 \Delta w + f'(v)\cdot w = 0  \mbox{ in }  \Omega ~ \mbox{ and } ~
 w  = 0 \mbox{ on }  \Gamma
\end{equation}
has no non-trivial solutions.
Then, for any $W_0 =(w_0;w_1) \in \cH $,  there exists a stationary
point $V =(v;0)$ such that
\begin{equation}\label{ode}
\|S(t) W_0 - V\|_{\cH} \leq C \left( \sigma( [ t
T^{-1}])\right),\quad t > 0,
\end{equation}
where $C$ and $T$  are positive constants depending on $W_0$, $[a]$ denotes the
integer part of $a$ and $\sigma(t) $ satisfies
\eqref{ode1} with $\sigma_0 =  C(W_0, V) $
where  $C(W_0, V ) $ is a constant depending on $ ||W_0||_{\cH}$ and
$||V||_{\cH}$ and $Q$ is defined in (\ref{q-G0}).
In particular, if  $ g'(0)> 0 $, then
$$
  \|S(t) W_0  -V \|_{\cH} \leq C e^{-\omega t }
$$
for some positive constants $C$ and $\omega$ depending on $W_0$.
\end{theorem}
\begin{remark}{\rm
Since $Q(s)$ is strictly increasing and $Q(0) =0$, the  rates
described by the ODE in (\ref{ode1}) (see, e.g., \cite{las-tat} or
\cite[Appendix B]{cl-book})
decay uniformly to zero. The ``speed" of decay depends on the
behavior of $g'(s) $ at the origin. If $g'(s)$  decays to zero
polynomially, then by solving the ODE in (\ref{ode1}), one obtains
algebraic decay rates for the solutions to $\sigma(t)$. If,
instead  $\delta =0$ and $g'(0) > 0$, then $Q(\sigma) = a \sigma $ for some
$ a >0 $ and,  consequently, the decay rates  derived from (\ref{ode1})
are exponential (see \cite{las-tat} for details).
In the latter case using the dissipitavity of strong solutions
(see Theorem~\ref{th:strong}) one can prove by interpolation that
the exponential decay rate holds for strong solutions in a stronger
(than in $\cH$) topology.
}
\end{remark}
\subsubsection{ Boundary damping }
 \index{wave equation!boundary  damping}
In the case when boundary dissipation is  the {\it only active mode  of damping }
(model in (\ref{1.4}) and \eqref{w-bc-n} with
$ g =0$) one can still prove that the long time behavior is essentially the same as in the case of internal damping. The task of achieving this goal is much more technical, due to the necessity of propagating the damping from the boundary into the interior. This is  done by using familiar by now "flux" multipliers. However, the resulting analysis
and PDE estimates are considerably more complicated and often resort to specific unique continuation property as well as Carleman's estimates (when dealing with critical cases).
This topic has been considered in \cite{cel,cel2,clt-jdde09}.
The corresponding results is given below.
\begin{theorem}[\bf  Boundary Dissipation]\label{Main-boundary}
With reference to  the equations in (\ref{1.4})and \eqref{w-bc-n}, where we take $ g =0 $, $h(w) = - w $,  the source $f$ satisfies dissipativity   property imposed  in
\eqref{dis-f-wave},
where $\lambda_1$ corresponds to the first eigenvalue of the Robin problem
and such that $|f''(s)|\le C(1+|s|)$ for all $s\in \R$.
 The
 boundary  damping  $g_0$  is an increasing, differentiable  function, $g_0(0) = 0 $  and satisfies the asymptotic growth condition:
 \begin{equation}\label{g0}
 m \leq  g_0'(s)  \leq M,~~ |s| > 1,~~~ m, M > 0,
 \end{equation}
 Then
 \begin{enumerate}
 \item
 Under the above assumption the statements of Theorem \ref{T:1} and \ref{T:3} hold.
 \item
 If, in addition,
  $g_0'(0)  > 0 $ then the statements of Theorem \ref{Main} and also Theorem \ref{sm},
 Theorem \ref{Main1} (with appropriate modifications) hold true.
 \end{enumerate}
\end{theorem}
One fundamental difference between boundary damping and interior damping is that
the restriction of linear bound at infinity for the boundary damping is essential. One dimensional examples \cite{vancostenoble}  disprove validity of uniform stability to a single equilibrium when the damping is superlinear at infinity.

\subsubsection{Generalizations}
{\bf With reference to interior damping}
\begin{enumerate}
\item
Wave equation with homogenous  Neumann or Robin  boundary conditions can be considered in the same way.
\item
The same wave model with additional  boundary damping and source, as described in (\ref{1.4}) and (\ref{w-bc-n}).
The corresponding analysis is more involved. Details are given in \cite{snowbird}.
\item
Possibility of degenerate damping $g(s)$, where the density function $a(x)$ describes possibility of degeneracy of the damping. This situation is also studied in \cite{snowbird}.
\item
Strong attractors. These are attractors corresponding to strong solutions.
In the case when boundary source and  damping are involved and the damping may degenerate,
the description of strong attractors is more  involved. These are   bounded sets  $ V \times H^1(\Omega) $
where $ H^2(\Omega) \subset V $. The details are given  in \cite{snowbird}.
\item
Under additional assumptions imposed on regularity of the damping $g(s)$ and the source $f(s) $ one can prove that
the attractors corresponding to Theorem \ref{sm} are infinitely smooth $C^{\infty}$. This can be done by the usual boot-strap argument.
\item
Attractors with localized dissipation only. This corresponds to having $a(x)$ localized in the layer near boundary.
\cite{clt-dcds08}.  As in the case of boundary damping, requires linearly bounded damping.
\end{enumerate}
{\bf With reference to boundary damping}
\begin{enumerate}
\item
Attractors with boundary damping  which is further localized to a  suitable portion of  the boundary. Requires
appropriate geometric conditions, see \cite{clt-jdde09} for details.
\item
Relaxed regularity assumptions imposed on the damping function $g_0(s)$, see \cite{clt-jdde09} again.
\end{enumerate}
\subsubsection{Open Problems}
\begin{enumerate}
\item
Higher regularity of attractors with minimal hypotheses on the damping and source.
\item
Under which conditions weak attractors coincide with strong attractors in the case of boundary or partially localized damping?
\item
Supercritical sources (e.g., $f(s) = s^5$). Ball's method will provide existence of attractors with linear damping. However, linear damping
is not sufficient for uniqueness of solutions. One needs to consider non-unique solutions.
\item
Nonmonotone damping. Some particular results in this
direction can be found in  \cite{cl-mem}.
\item
Non-gradient structures. For instance models involving   first order differentials as the potential  sources.
\end{enumerate}

\subsubsection{Remarks }

\begin{enumerate}
\item
{\it Theorem \ref {T:1} -Compactness.} Difficulty: double critical exponent -rules out previous methods.   Use Theorem \ref{th7-turk}
after an appropriate rewriting of the source which then fits into the structure of the functional $\Psi$. This point is explained below.
\item
{\it Theorems \ref{Main} and  Theorem \ref{sm}-finite dimensionality and smoothness}. Relies on proving quasi-stability inequality. The main trick is to provide appropriate  decomposition of the source  that  takes advantage
of dissipativity integral  which is $L_1(\R) $. See \cite{snowbird,clt-jdde09}.
\item
{\it Full statement of Theorem \ref{sm}} -  \cite{A-K}.
 Smoothness of the trajectories on the attractor near $-\infty$.
 The  existence of strong attractors (for strong solutions) is de facto consequence of quasi-stability inequality. However, in many situation proving additional spatial regularity ($H^2(\Omega)$, for instance) is problematic. It is for this task
 that special considerations  related to backward smoothness are important.
\item
{\it Decay rates - Theorem \ref{Main1}}.  The difficulty relates to the fact that convergence to equilibria is a very "unstable" process. It is not possible to localize the analysis to the neighborhood  of the equilibria.  Decay rates are derived by using convex analysis and reducing estimates to nonlinear ODE's. See \cite{snowbird}.
\item
{\it Boundary damping - Theorem \ref{Main-boundary}}.
First results  are given in \cite{cel,cel2} where finite-dimensionality was proved  for subcritical  sources only.
Complete proof of Theorem \ref{Main-boundary} (in fact in a more general version) is given in \cite{clt-jdde09},
where flux multipliers are used along with Carleman's estimates and also backward smoothness of trajectories.
\end{enumerate}

\subsubsection{Guide through  the Proofs.}

We are not in a position to present complete proofs of the results. These are given in  the cited references.
Here, our aim instead is to orient the reader what are the main points one needs to go over
in order to prove these results. Special emphasis is given to more subtle technical details where "some" tricks
-most of which recently introduced - need to be applied.
It is our hope that  an attentive reader will be able to grasp the essence of the proof, the tools required so that when guided
to more specific reference  he will be in a position to reconstruct
  a complete  proof.
\medskip\par\noindent
{\bf Theorem \ref{T:1} - Compactness}
\smallskip\par\noindent
{\sc Step 1: The corresponding system is gradient system. }
This follows from the fact  that  Lyapunov function $V(t) \equiv \cE(t) $ is bounded from above and below by the topology of the phase space. The latter  results from  dissipativity condition imposed on $f(s)$ in Assumption \ref{fg}
(see \eqref{dis-f-wave}).
Moreover  $V_t(t) =0$ implies $((g(w_t), w_t )) =0 $ hence   $w_t(x,t) =0$ in $\Omega $ reduces the dynamics to a stationary set defined by \eqref{N-set-wave} and \eqref{eqn-111}.
 The set $\cN$ of stationary points  is bounded due, again,  to dissipativity condition imposed on $f$. This part of the  argument is standard.
\smallskip\par\noindent
{\sc
Step 2:  Asymptotic smoothness.} The  main difficulty in carrying  the proof  of asymptotic smoothness is double criticality of both the source and the damping.  This is the reason why previous approaches (based on splitting or squeezing) did require additional assumptions.
To overcome this  difficulty  we  use Theorem \ref{th7-turk}.
In preparation for this theorem we are  using  the equation
 \begin{equation}\label{z-dfi-1.1}
z_{tt} - \Delta z +  a(x) [g(w_t)-g(u_t)] =f(w)-f(u) ~~\mbox{in}~~\Omega  \times \R_+
\end{equation}
with $z=0$ on $\Gamma$,
 written for the difference of two solutions $z = w-u $,
 where both $w$ and $ u$ are solutions to the original equations
 confined to a bounded set $\cB \in H $, and also
\begin{itemize}
  \item {\it  energy identity} (multiply equation \eqref{z-dfi-1.1} by $z_t$);
  \item {\it  equipartition} of energy (multiply equation \eqref{z-dfi-1.1} by $z$).
\end{itemize}
In addition we recall dissipation  relations for each solution $u$ and $w$:
\begin{equation}\label{disip}
\int_0^t (( g(u_t) , u_t) )) + (( g(w_t), w_t )) \leq C_B,~~~\forall\, t > 0,
\end{equation}
and also
\[
\|u_t(t)\|^2 +\|u(t)\|^2_{1,\Om} + \|w_t(t)\|^2 +\|w(t)\|^2_{1,\Om} \leq C_B,~~~\forall\, t > 0.
\]
Critical step consists of obtaining "recovery" estimate for the energy of $z$ in terms of the {\it damping} and the {\it source}.  The important part is that the "recovery" estimate  does not involve any initial data.
This step is achieved by using  {\it equipartition} of the energy.
The result of  which is  the  following inequality \cite{snowbird}:
\begin{equation}\label{asym-1}
\hf T E_z(T) + \int _0^T E_z(t) \leq
\int_0^T\left[ ||z_t||^2 + D(z_t) +  |G(z_t,z)|\right] dt + \Psi(z, T),
\end{equation}
for $T\ge 4$,
where $E_z(t) \equiv ||z_t(t)||^2 + ||\nabla z(t)||^2$,
\begin{equation}\label{gzzt}
D(z_t) \equiv (( g(w_t) - g(u_t), z_t)),~~~
G(z_t, z) \equiv  (( g(w_t) - g(u_t), z)),
\end{equation}
and
\begin{multline*}
\Psi (z, T) \equiv- \int_0^T ((f(w) - f(u), z_t)) \\ + \int_0^T ds \int_s^T d\tau ((f(w) - f(u) , z_t))
+ \int_0^T | ((f(w) - f(u) , z))| dt
\end{multline*}
Our goal  is to bound the right hand side of the above relation by terms that are (i) either bounded uniformly in $T$,  or (ii) compact or (iii) small multiples of energy integral.
\par
From the energy relation for $z$ we have that
\begin{equation}\label{dis-z}
\int_0^T  D(z_t) d\tau \leq C_{B}+  \int_0^T ((f(w) - f(u), z_t)) d\tau.
\end{equation}
One can also wee from Proposition B.1.2 in \cite{cl-book} that
\[
\|z_t\|^2\le \eta + C_\eta D(z_t),~~~\forall\, \eta>0.
\]
Thus the terms $\|z_t\|^2$ and  $D(z_t)$ produce
in \eqref{asym-1} an admissible contribution.
\par
The last integral in the definition of $\Psi$ is compact. Thus the noncompact terms are the second one
involving $G$ and the first two integrals in definition of $\Psi$. Since these two terms are similar, it suffices
to analyze the main non-compact terms in (\ref{asym-1}) which are:
$$
\int_0^T |G(z_t,z)|  dt ~~\mbox{and}~~  \int_0^T d s\int_s^T ((f(w) - f(u), z_t ))d\tau.
$$
These terms reflect  double criticality  of  the damping-source exponents. We  analyze these next.
\par
For the {\bf damping term} $G$ one uses split $\Om$ into
$\Omega_1= \{|u_t(x,t)| > 1\}$ and the complement $\Omega_2 =
\Omega
\setminus \Omega_1$. With the use of critical growth condition for $g$
and the fact that $\|z\| \le C_\B$ we obtain
\begin{align*}
  \int_{\Omega_1} |  g(u_t)| |z |dx  \leq &
  \left[\int_{\Omega_1}|g(u_t)|^{6/5}dx \right]^{5/6}\|z\|_{L_6(\Omega)}   \\
\leq & \left[ \int_{\Omega_1} g(u_t) u_t dx \right]^{5/6} \|z\|_{1, \Omega}   \leq C_{B} \int_{\Omega} g(u_t) u_t dx.
\end{align*}
 On $\Omega_2 $ the damping is bounded -contributing to  lower order terms.
 Thus we obtain for $G(z_t,z) $
\begin{align*}
\int_0^T G(z_t,z)  dt  \leq & C_B \int_0^T \left[
 (g(u_t), u_t) + (g(w_t), w_t) \right] dt  +
C_B T \sup_{t\in[0,T]} ||z(t)|| \\
\leq & C_B  + C_B T \sup_{t\in[0,T]} ||z(t)||,
\end{align*}
which gives the inequality in terms of the universal constant (independent on time $T$) and a lower order term
$||z(t)||$.
\par
For the {\bf source}, the key point is the following decomposition
that allows to prove sequential convergence of the functional $\Psi$ in   Theorem \ref{th7-turk}:
\begin{multline*}
\int_0^T (( f(w) - f(u), z_t ))  \\
= \int_{\Omega} \left[ F(w(T) ) + F(u(T) ) - F(w(0)) - F(u(0))\right]
\\ -
\int_0^T\left[ ((f(u), w_t )) + ((f(w), u_t)) \right] dt,
\end{multline*}
where $F(s)$ denotes the antiderivative of $f$. In line with Theorem \ref{th7-turk} we consider these terms
on sequences $ u=w_n, w= w_m $ converging weakly in $H^1(\Omega) $  to a given element. We want to conclude that the corresponding functional converges sequentially to zero.
Convergence on the first four  terms is strong due to compactness while convergence
on the last term is sequential with $w_n \rightarrow w$  weakly~in   $H^1(\Omega)$
and $w_{mt}
\rightarrow w_t$  weakly in $L_2(\Omega)$.
Applying (\ref{asym-1}) gives
$$
E_z(T) \leq \frac{\epsilon}2+\frac{C_B}{T} + C_\epsilon
\left[\Psi (z,T) + LOT(z,T)\right]  \leq \epsilon + C_\epsilon
\left[\Psi (z,T) + LOT(z,T)\right],
$$
for $T=T_\eps$ large enough,
where $LOT(z,T)$ comprise of all compact terms.
Now we use the fact that
\begin{align*}
  \hat \Psi (w_n,w_m, T)
\equiv  &-\int_0^T ((f(w_n) - f(w_m), w_{nt} - w_{mt}))
\\ &
+ \int_0^T ds \int_s^T d\tau ((f(w_n) - f(w_m),w_{nt} - w_{mt}))
\end{align*}
is weakly sequentially compact - see the details in \cite{snowbird,cl-mem}.
\smallskip\par\noindent
{\sc Step 3: completion.}  We apply Theorem \ref{th:grad-attr} to deduce the results stated in Theorem \ref{T:1} and Theorem \ref{T:3}.

\medskip\par
\noindent{\bf Theorem \ref{T:3} -Structure} -follows from asymptotic smoothness and gradient structure via Theorem
 \ref{th:grad-attr}.

\medskip\par
\noindent {\bf Theorems \ref{Main} and \ref{sm} - Finite Dimensionality  and Regularity.}
 The key tool is the quasi-stability inequality formulated
   in Definition \ref{de:ms-stable}.
 This inequality, as in the case of asymptotic smoothness, should hold for the difference  of two solutions $z =w-u$
 taken from a bounded set-say an  absorbing ball $\cB$.
  As before, we apply two multipliers  $z$ and $z_t$.
  \par

 For the damping term $G(z_t,z)$ given in \eqref{gzzt} we have the following estimate
 (see Proposition 5.3 in \cite{snowbird} for details):
  \begin{lemma}[Damping]\label{d} For every $\epsilon>0$
 there exists $C_\eps$ such that
 \begin{align*}
 \int_0^T |G(z_t,z)|dt \leq & \epsilon \sup_{[0,T]}\|z\|^2_{1,\Om}
 \left( T + \int_0^T \big[((g(w_t), w_t)) + ((g(u_t), u_t))\big] d\tau\right)
 \\
 & {}\, + C_{\epsilon} \int_0^T ((g(u_t) -g(w_t), z_t )) dt.
 \end{align*}
   \end{lemma}
The following decomposition of the source is also  essential (see Proposition
 5.4  in \cite{snowbird}).
  \begin{lemma}[Source]\label{s} The following estimate holds true:
    $$
  \int_t^T (( f(u) - f(w) , z_t )) \leq C_{B,T} ||z||^2 + \epsilon\int_0^T \|z\|^2_{1,\Omega}
  +C_{\epsilon, B } \int_0^T \|z(t)\|^2_{1,\Omega} K(t)dt,
   $$
  where $K(t) = ||w_t(t)||^2 + ||u_t(t)||^2$ which  under the condition
  $g'(s)\ge m>0$, $\forall\, s$, by (\ref{disip}) satisfies
  the following dissipativity property
  \begin{equation}\label{K-dis}
  \int_0^{\infty} K(s) ds \equiv \int_0^{\infty}\left[||w_t(s)||^2 + ||u_t(s)||^2 \right] ds  < C_B.
 \end{equation}
  \end{lemma}
 By using dissipation properties in  (\ref{disip})
   and in \eqref{K-dis}
   and combining with the results of
 Lemma \ref{d} and Lemma \ref{s}   one obtains
 (see \cite{snowbird} for details)
 the inequality
 \begin{align*}
E_z(t) \leq & C_1 \left[ e^{-\omega t} E_z(0) + ||z||^2_{L_{\infty}(0,t; H^1_0(\Om)}
\right] \exp\left\{C_2\int_0^t \tilde K(s) ds \right\} \\
 \leq & C \left[ e^{-\omega t} E_z(0) + ||z||^2_{L_{\infty}(0,t; H^1_0(\Om)}
\right]
 \end{align*}
 where
 \[
 \tilde K(t) = K(t) + ((g(u_t(t)), u_t(t))) + ((g(w_t(t)), w_t(t))).
 \]
 This leads to the desired quasi-stability inequality.
\medskip\par
 \noindent{\bf  Theorem \ref{sm} -Strong attractors} - proof  explores smoothness of backward trajectories on the attractor, see \cite{A-K}.
\medskip\par
 \noindent{\bf Theorem \ref{Main1}-Decay rates}. The details given in
 \cite[Section 8]{snowbird}.
\medskip\par
 \noindent{\bf Theorem \ref{Main-boundary} -Boundary Dissipation.}
 The proof is more technical here. Let us point out the main differences.
 \medskip\par\noindent
 {\bf Compactness:}
  \smallskip\par\noindent
 {\sc Gradient structure:} In the boundary case the question of gradient structure  has one additional aspect
 with respect to the interior damping. It involves  a familiar in control theory question of {\it unique continuation} through the boundary.  This property is known to hold due to various extensions of Holmgren-type  theorems. This aspect of the problem is central in inverse problems when one is to reconstruct
 the source from the boundary measurements.  Under suitable geometric conditions unique continuation results hold
 for the wave equation  with a potential and  $H^1$ solutions \cite{isakov,ruiz,eller}.
 \smallskip\par\noindent
 {\sc Reconstruction of the energy:} Since the damping is localized near the boundary, reconstruction of the energy involves propagation of the damping from a boundary into the interior.  While in the case of interior damping
 equipartition of energy was sufficient for recovery of potential energy, this is not enough in the present case.
 Additional multiplier is needed which is referred as "flow multiplier" given by $h \nabla w $, where $h$ is a suitable vector field. This multiplier -of critical order for the wave equation- is successful in reconstructing the energy of a single solution -provided time $T$ is sufficiently large and  the support of the damping on the boundary sufficiently large.
 However, when dealing with the differences of two trajectories (functions $z$), the method is no longer successful.
 There is a competition of two critical terms. A major  detour  is needed when
 one estimates  the  difference of two solutions.
  In  the boundary case we have two critical multipliers $h \nabla z$ and $z_t$ both of energy level:
 $$
    (( f(u) - f(w) , h \nabla z )) ~~\mbox{and}~~ ((f(u) - f(w) , z_t )).
 $$
While this is not a problem with {\it subcritical } sources, it becomes a major issue in the case of {\it critical} sources.
 The remedy  for  the first term to this issue is by invoking Carleman's estimates with large parameter which allows for balancing the terms. This step is very technical-details can be found in \cite{clt-jdde09}.
  \smallskip\par\noindent
 {\sc Dissipativity integrals:} Note that in this case (\ref{dis-z}) requires additional integration over the boundary in the right hand side.
 However, flux multiplier - again - with Carleman's estimates  allows to  get the estimate on this term.
 At the end of the day, with the use of restrictive growth conditions assumed on $g$ (linear bound at infinity)
  one obtain reconstruction of the form as in (\ref{asym-1} without $||z_t||^2 $ term.
  The estimate of $G(z,z_t)$ is reasonably straight due to linear bound hypotheses assumed.
   \smallskip\par\noindent
{\sc  Functional $\Psi$:}
  The estimate for the compactness part  of $\Psi$  is the same. It depends on the source only. This allows for the applicability of the same Theorem \ref{th7-turk} to conclude compactness of attractors with boundary dissipation.

   \medskip\par\noindent
  {\bf Quasi-stability Consequences: Finite dimensionality and Smoothness of Attractors.}
  Here the presence of the boundary damping presents  another major challenge.
  This is at the level of the inequality in Lemma~\ref{s}. The result stated there  depends critically on finiteness of  $K(t)$.
  This is no longer valid in the pure boundary case, where dissipativity integral involves boundary integrals only.  In that case  one resorts to the analysis of "backward" smoothness on the attractor.
  This procedure consists of the following four steps.
   \smallskip\par\noindent
{\sc  Step 1:} Prove quasi-stability inequality for  solutions on the attractor near stationary points.
   This leads to consideration of   negative time $t \rightarrow -\infty$.
  Such estimate is possible due to the fact that the system is gradient system and the  velocities $w_t(t)$ and $u_t(t) $ in (\ref{K-dis}) are "small"
  near equilibria.
   \smallskip\par\noindent
{\sc  Step 2:}
  Quasi-stability inequality provides additional smoothness of the trajectories at some  negative time $T_f < 0$.
  This smoothness is propagated  forward  on the strength of regularity theorem. The consequence of this
  is that the  elements in the attractor $\cA$  are  contained in $ H^2(\Omega) \times H^1(\Omega)$.
     \smallskip\par\noindent
{\sc  Step 3:}
The  ultimate smoothness of attractors is achieved by claiming that the attractor is bounded in the topology
  of $H^2(\Omega) \times H^1(\Omega)$.
  The above conclusion is obtained  by exploiting compactness of attractors and finite net coverage.
     \smallskip\par\noindent
{\sc  Step 4:}
Having obtained  smoothness of attractor,  one proceed to prove finite-dimensionality.
  This becomes a straightforward consequence of quasi-stability property
  on the attractor resulting from the aforementioned smoothness.
   \smallskip\par
    It should be noticed that the method described above is fundamentally based on gradient property of the dynamics.
  The  full details are given in \cite{clt-jdde09}.

  \begin{remark}{\rm
  In the case of boundary damping (linearly bounded)  there is another method that leads to compactness of attractors
  also in the presence of {\it critical} sources. This  method relies on a suitable split of  the dynamics \cite{cel}.
  In such case there is no need for Carleman's estimates.  However,  properties such as regularity and finite dimensionality of attractors intrinsically
   depend on  quasi-stability estimate. The latter requires consideration
  of the differences of two solutions.
  It is this step that forces Carleman's estimates to operate.
  In the subcritical case this is not necessary and more direct estimates provide full spectrum of results also in the  boundary case, see \cite{cel2}.
  }
  \end{remark}

  \subsection{Von Karman plate dynamics}

\subsubsection{Internal damping}

We consider  system \eqref{1.2} with clamped boundary conditions \eqref{Cl} and
discuss  long time behavior under the following hypothesis.
\begin{assumption}\label{as:g-attr}
 We assume that
Assumption~\ref{as-karman} holds and, in addition,
 $g_i\in C^1$ and $g_i'(s)\ge m$ for $|s|\ge s_0$ in the case $\al>0$
and   in the case $\al=0$: $g\in C^1$ and $g'(s)\ge m$ for $|s|\ge s_0$.
\end{assumption}
We also want to be more specific about the nature of the source $P(w)$.
For sake of simplicity, we do not  consider non-conservative force cases which lead to systems that are no longer of gradient structure. Instead we  limit ourselves to conservatively forced models. We suppose that Assumption~\ref{as:karman-P}
holds with $P_1(w)\equiv 0$, i.e.,
 $P(w)$ is a Fr\'echet  derivative of the functional $- \Pi_0(w)$
and the   following property holds: there exist $a \in \R$ and $\epsilon >0$
such that
\begin{equation}\label{Pi}
 \Pi_0(w)  +   a\|w\|_{2-\epsilon,\Omega} ^2 \geq 0,~~  ((P_0(w), w))  +   a\|w\|_{2-\epsilon,\Omega} ^2 \geq 0,~~\forall w\in H_0^2(\Om).
 \end{equation}
  \begin{remark}{\rm
In the specific example given in Remark~\ref{re:F0} the functional $\Pi_0$  takes the form
$$\Pi_0(u) =-\frac{1}{2}  (( [F_0,u], u))-((p,u)), ~~~ F_0\in H^{3+\delta}(\Om)\cap H_0^1(\Om),~~ p\in L_2(\Om).
 $$
Thus, the assumption (\ref{Pi}) is satisfied (in the case $\al>0$ we can take
less regular $F_0$).
}
\end{remark}
Under these conditions concerning $P$
 the energy function can be written as
\[
\cE(u,u_t)=\frac12 \left[ \|u_t\|^2+\alpha  \|\nabla
 u_t \|^2\right] +\frac{1}{4} ||\Delta \cF(u) ||^2+ \Pi_0(u).
\]
Under the general Assumption \ref{as-karman} we have shown that von Karman system generates a continuous semiflow  in the space
$H_\al \equiv H^2_0(\Om) \times \cH_\al(\Omega)$, where $\cH_\al(\Omega)$ is given by
\eqref{h-alpha}, i.e.,
 $\cH_\al(\Om)=H^1_0(\Om)$ in the case $\al>0$ and
$\cH_\al(\Om)=L_2(\Om)$ in the case $\al=0$.
This
leads to a well-defined dynamical system $(H_\al,S_t)$.
\par
The existence of a global attractor  under the assumptions specified for model (\ref{1.2})
can be derived now  from   Theorem \ref{th7-turk}. The main long-time result  reads as follows:

\begin{theorem}[Compactness]\label{th8.2.2} Let Assumption~\ref{as:g-attr} be in force.
Assume that  $P(w)=[F_0,w]+p$, where $F_0\in H^{3+\delta}(\Om)\cap H_0^1(\Om)$ and $p\in L_2(\Om)$.
Then the dynamical system $(H_\al,S_t)$ generated by
 equations \eqref{1.2} with the  clamped boundary conditions \eqref{Cl}
 possesses a global  compact attractor $A $
   in the space
$H_\al \equiv H^2_0(\Om) \times \cH_\al(\Omega)$,
where  $H_\al$ is given by
\eqref{h-alpha}.
\end{theorem}
In the case $\al>0$ the proof of Theorem~\ref{th8.2.2}
follows from general abstract results
for second order equations with subcritical source  (see \cite{cl-mem}).
In the case $\al=0$ the force is critical and the
proof relies  on Theorem~\ref{th7-turk} in the same manner as  in the case of
wave equations.  We refer to \cite{cl-book} for details.
\begin{remark}\label{re:P-gen}
{\rm
One can easily obtain the same result for a more general class of  nonlinear sources $P(w)$  which can be axiomatized
by (\ref{Pi}) with an additional requirement that the map is compact.
}
\end{remark}
\par

It is easy to verify (internal damping) that the energy
$\cE(u,u_t)$
is a strict Lyapunov  function for $(H_\al, S_t)$. Hence  $(H_\al, S_t)$
is a gradient system (see Definition \ref{de7.4.1} in Section~\ref{grad-ds}),
and the
application of Theorem~\ref{th:grad-attr} yields  the
following result.
\index{evolutionary Karman equations with rotational
forces!internal dissipation!regular structure of attractor}
\index{evolutionary Karman equations with rotational
forces!internal dissipation!global attractor}
\index{evolutionary Karman equations without rotational
forces!internal dissipation!global attractor}
\begin{theorem}[Structure]\label{th9.2.1}
Let Assumption~\ref{as:g-attr} be in force.
Assume that $P$ in \eqref{1.2} satisfies (\ref{Pi}).
 Let  $(H_\al,S_t)$
be the dynamical system  generated by
 equations \eqref{1.2} with clamped boundary conditions
Assume that  $(H_\al,S_t)$ possesses a compact global attractor $A$.
Let $\cN_*$ be  a set of stationary solutions
to \eqref{1.2} (see Proposition~\ref{pr:stationary}).
Then
\begin{itemize}
  \item $A=\cM^u(\cN)$, where $\cN=\{ (w;0)\, :\, w\in\cN_*\}$ is
  the set of stationary point of $(H_\al,S_t)$ and $\cM^u(\cN)$
is the  the unstable manifold $\cM^u(\cN)$ emanating from  $\cN$
which is defined
as a set of all $U\in H_\al$ such that there exists a full
trajectory $\gamma=\{ U(t)=(u(t);u_t(t))\, :\, t\in\R\}$ with the properties
\[
 U(0)=U~~ \mbox{and}~~ \lim_{t\to -\infty}{\rm dist}_{H_\al}(U(t),\cN)=0.
\]
\item The global attractor $A$ consists of full
trajectories $\gamma=\{(u(t);u_t(t)): t\in\R\}$ such
that
\[
\lim_{t\to \pm\infty}\left\{  \|u_t(t)\|_{\cH_\al}^2+ \inf_{w\in\cN_*}
\|u(t)-w\|^2_2\right\}=0.
\]
\item Any
 generalized solution  $u(t)$ to problem
\eqref{1.2} stabilizes to the set of stationary points; that is,
\begin{equation}\label{9.2.3}
  \lim_{t\to +\infty}\left\{  \|u_t(t)\|_{\cH_\al}^2+ \inf_{w\in\cN_*}
\|u(t)-w\|^2_2\right\}=0.
 \end{equation}
\end{itemize}
\end{theorem}

This theorem along with generic-type  results on the finiteness of the number
of solutions to the stationary problem allow us to obtain the following result.
\begin{corollary}\label{co9.2.1}
Under the hypotheses of Theorem~\rref{th9.2.1} there exists an open
dense set $\sR_0$ in $L_2(\Om)$ such that for every  load function $p\in \sR_0$
the set $\cN$ of stationary points for $(H_\al,S_t)$ is finite. In
this case $A=\cup_{i=1}^N\cM^u(z_i)$, where $z_i=(w_i;0)$ and
$w_i$ is a solution  to the stationary problem Moreover,
\begin{itemize}
  \item
The global attractor $A$ consists of  full  trajectories
$
\gamma=\{ (u(t);u_t(t))\, :\, t\in\R\}
$
connecting  pairs of   stationary points; that is,  any $W\in A$ belongs to some full
trajectory $\gamma$ and  for any $\gamma\subset A$ there exists a pair
$\{ w_{-}, w_{+}\}\subset\cN^*$ such that
\[
 \lim_{t\to -\infty}\left\{  \|u_t(t)\|_{\cH_\al}^2+
\|u(t)-w_-\|^2_2\right\}=0
\]
and
\[
 \lim_{t\to +\infty}\left\{  \|u_t(t)\|_{\cH_\al}^2+
\|u(t)-w_+\|^2_2\right\}=0.
\]
\item For any $(u_0;u_1)\in H_\al$ there exists a stationary solution $w\in\cN^*$
such that
\begin{equation}\label{9.2.4}
 \lim_{t\to +\infty}\left\{  \|u_t(t)\|_{\cH_\al}^2+
\|u(t)-w\|^2_2\right\}=0,
\end{equation}
where $u(t)$ is a
 generalized solution  to problem \eqref{1.2}
with  initial data $(u_0;u_1)$ and with the clamped boundary
conditions.
 \end{itemize}
\end{corollary}

\begin{remark}\label{re:lojasivich} {\rm
We note that the {\it generic} class ${\cal R}_0$ of loads in the statement
of Corollary~\ref{co9.2.1}
 can be rather thin.
The standard example (see, e.g., \cite[Section 2.5]{Chu99}) is the set ${\cal R}_0$
in the interval $[0,1]$ of the form
\[
{\cal R}_0= \left\{ x\in (0,1)\, :~~\exists\; r_k\in\Q, ~~ |x-r_k|\le \eps 2^{-k-1}\right\},
\]
where $\Q$ is the sequence
of all rational numbers.  This set  is open and
dense in  $[0,1]$, but its Lebesgue measure less than $\eps$.
However using analyticity of the von Karman force terms
and the Lojasiewicz-Simon inequality we can obtain a result
(see \cite{Chu11-loja}) on convergence
for {\em all} loads $p$ under some additional  non-degeneracy at the origin  type
conditions concerning damping terms. By assuming this additional  non-degeneracy condition
(always satisfied when $g'(0) > 0 $ or in the hyperbolic case when $g(s)s \geq m s^2$, $|s|\leq 1 $)
 the result of Corollary \ref{co9.2.1} holds for all loads $p\in L_2(\Omega) $.
For general discussion of    Lojasiewicz-Simon method we refer
to \cite{HarJen99, HarJen09}  and to the references therein.
}
\end{remark}

It follows from Assumption~\ref{as:g-attr} (see \cite{cl-book}, for instance)
that
there exists  a strictly increasing continuous concave
function $H_0(s)$ such that
\begin{equation*}
s_1^2+s_2^2 \leq   H_0(s_1 g_i(s_1)+s_2 g_i(s_2))
~~\mbox{ for all }
~s_1,s_2 \in \R.
\end{equation*}
in the case $\al>0$ and
$s^2 \leq   H_0(s g(s))$ for $s \in \R$
in the case $\al=0$.

 The above  facts  are used in formulation
 of the result on {\bf the rate of convergence to
an equilibrium}.
\index{evolutionary Karman equations with rotational
forces!internal dissipation!rate of convergence to equilibrium}
\begin{theorem}[Rates of convergence to equilibria]\label{th9.2.2}
Assume that $P$ satisfies (\ref{Pi})  in \eqref{1.2} and
the hypotheses above that guarantee
the existence of a compact global attractor $A$
for the dynamical system $(H_\al,S_t)$ generated by problem
\eqref{1.2}  with  the clamped boundary
conditions hold.
\par
If   the set of stationary solutions  consists of finitely many
isolated equilibrium points that are hyperbolic,\footnote{ In the
sense that the linearization of the stationary problem
 with the corresponding boundary conditions around
every stationary solution has a trivial solution only. }
then for
any initial condition  $ y \in H_\al$ there exists an equilibrium point
$ e =(w; 0)$, $w\in H^2_0(\Om)$, such that
\begin{equation}\label{8.6.2a-mdf}
\| S_ty -e\|_{H_\al}^2\le C \cdot\sigma\left(\left[ tT^{-1}\right]\right),\quad t>0,
\end{equation}
where $C$ and $T$ are positive constants depending on $y$ and  $e$,
$[a]$ denotes the integer part of $a$ and $\sigma(t)$ satisfies the
following ODE,
\begin{equation}\label{8.6.2b-mdf}
\frac{d\sigma}{dt}+Q(\sigma)=0,~ t>0,\quad \sigma(0)=C(y,e).
\end{equation}
Here $C(y,e)$ is a constant depending on $y$ and  $e$,
\begin{equation}\label{10.1q-G0}
Q(s) =s-\left(I+G_0\right)^{-1}(s)~~~
with ~~G_0(s) =c_1\left(I+H_0\right)^{-1}(c_2s),
\end{equation}
 where
positive numbers $c_1$ and $c_2$ depend on $y$ and $e$.
If in Assumption~\ref{as:g-attr} $s_0=0$,
then the rate of
convergence in (\ref{8.6.2a-mdf}) is exponential.
\end{theorem}
\begin{remark}\label{re:loja-2}{\rm
It is also possible to state a result on the convergence rate without
assuming hyperbolicity of equilibrium points (see \cite{Chu11-loja}).
However as in the case described in Remark~\ref{re:lojasivich}
 this requires
additional  non-degeneracy type properties of the
 damping terms. Moreover, the lack of hyperbolicity is compromised by slowing down  the decay rates.
 }
\end{remark}
The following result describes {\bf finite dimensionality} of attractors.

\begin{theorem}[Dimension]\label{th:dim-rot}
\!\!Assume
the hypotheses
 which gu\-arantee
the existence of a compact global attractor $A$
for the  system $(H_\al,S_t)$ generated by problem
\eqref{1.2} with the boundary
conditions \eqref{Cl} hold (see  Theorem~\ref{th8.2.2}) and $\alpha \geq 0 $.
In addition, we assume that
\begin{enumerate}
  \item[{\rm (i)}] In the case $\al=0$
there exist $m,M_0 >0 $ such that
\begin{equation}\label{g0-prim}
0 < m \le g'(s) \leq M_0 [1 +  s g(s) ],\quad s\in\R,
\end{equation}
 \item[{\rm (ii)}]  When $\alpha > 0 $
 we assume that (i) $g(s)$ is either of  polynomial growth at infinity
or else satisfies \eqref{g0-prim},
 and (ii)
 there exists $0\le\gamma<1$ such that the functions $g_i$ satisfy the
inequality
\begin{equation}\label{gi-prim}
0<m\le g_i'(s) \leq M [1 + sg_i(s)]^\ga ,\quad s\in\R,\; i=1,2,
\end{equation}
\end{enumerate}
Then the fractal dimension of the attractor $\cA$ is finite.
\end{theorem}

\begin{remark}{\rm
One can see that in the case
$\alpha > 0$ the condition imposed on $g_i$ in (\ref{gi-prim}) are valid
if in addition to Assumption ~\ref{as:g-attr} we assume that
that there exists $m.M_1,M_2>0$ such  that
\[
0 \leq g_i'(s) \leq M_1 [ 1 + |s|^{p-1}],
~~sg_i(s)\ge m  |s|^{(p-1)/\ga}-M_2,~~~
 s\in \R,
\]
 for some $p \geq 1$ (see Remark~ 9.2.7 \cite{cl-book}).
}
\end{remark}

The following assertion on regularity of elements in the attractor  is a straightforward consequence
of  Theorem~\ref{th7.9reg}.
\begin{theorem}[{\bf Regularity}]\label{th9.2.3}
Let
the hypotheses  of Theorem~\ref{th:dim-rot} be valid.
Then any full trajectory $\gamma=\{(u(t);u_t(t))\, :\, t\in\R\}$
from the attractor $A$ of the dynamical system  $(H_\al,S_t)$  generated by  problem
(\rref{1.5})
 with  the boundary
conditions \eqref{Cl} possesses
the properties
\begin{equation}\label{9.2.8}
u(t)\in  C_r(\R; W),\quad  (u_t(t),u_{tt}(t))\in C_r(\R; H_\al),
\end{equation}
where $C_r$ means right-continuous functions and
\begin{itemize}
  \item in the case $\al>0$:
 $H_\al=H^2_0(\Om)\times H^1_0(\Om)$  and $W=H^3(\Om)\cap H^2_0(\Om)$.
  \item in the case $\al=0$:
 $H_\al=H^2_0(\Om)\times L_2(\Om)$  and $W=H^4(\Om)\cap H^2_0(\Om)$.
\end{itemize}
Moreover, \begin{itemize}
  \item in the case $\al>0$:    $A\subset H^3(\Om)\times H^2(\Om)$ and
\[
\sup_{t\in \R}\left( \|u(t)\|^2_3+\|u_t(t)\|^2_2
+\|u_{tt}(t)\|^2_1\right)\le C_A<\infty,
\]
\item
in the case $\al=0$:   $A\subset H^4(\Om)\times H^2(\Om)$ and
\[
\sup_{t\in \R}\left( \|u(t)\|^2_4+\|u_t(t)\|^2_2
+\|u_{tt}(t)\|^2_1\right)\le C_A<\infty.
\]
\end{itemize}
\end{theorem}

The validity of quasi-stability property for the system allows us to
deduce, without an additional effort, another important property of dynamical system.
This is existence of  {\bf strong attractors }  that is,
attractors in a strong topology determined by the
generators of the (linearized) dynamical system.
The corresponding result is formulated below.
\index{evolutionary Karman equations with rotational
forces!internal dissipation!strong attractor}
\begin{theorem}[Strong attractors]\label{th9.2.3str}
We assume that the hypotheses  of Theorem~\rref{th9.2.3} are in force.
\par
{\bf Case $\al>0$:}
Let the rotational damping be linear; that is,
 $g_i(s)=g_i\cdot s$ for $i=1,2$.
 Then
the global attractor $A$ is also strong: for any bounded set $B$
from $H_{st}=(H^3(\Om)\cap H^2_0(\Om))\times H^2_0(\Om)$ we have that
\[
\lim_{t\to \infty}\sup_{y\in B}{\rm dist}_{H_{st}}(S_ty, A)=0.
\]
\par{\bf  Case: $\al=0$:}
Let $H_{st}=(H^4(\Om)\cap H^2_0(\Om))\times H^2_0(\Om)$.
The  global attractor $A$
of the system
$(H_0,S_t)$ generated by the generalized solutions to problem~\eqref{1.2}
with  clamped  b.c. \eqref{Cl}  is also strong,
i.e. $A$ is a global attractor for the system
$(H_{st},S_t)$.
\par
Moreover, in both cases $\al=0$ and $\al>0$  the  global attractor $A$ has a finite dimension as a compact set in  $H_{st}$.
\end{theorem}
\begin{proof}
We apply the idea of Theorem~8.8.4\cite{cl-book}. The hypotheses of Theorem~\ref{th9.2.3}
guarantees that the system $(H_\al,S_t)$ is  quasi-stable.
\end{proof}
\index{evolutionary Karman equations with rotational
forces!internal dissipation!fractal exponential attractor}
An important property of the dynamical system is an existence of exponential
attractors, which attract trajectories at the exponential rate.
The quasi-stability property
 allows us to deduce an {\bf existence of exponential attractors}.
The corresponding result is formulated below.
\begin{theorem}[Exponential attractors]\label{th9.2.5}
Let $(H_\al,S_t)$ be the
dynamical system  generated by  problem \eqref{1.2} and
 \eqref{Cl}.
\par
{\bf Case $\al>0$}:
 Under the condition
 \eqref{gi-prim} the dynamical system $(H_\al,S_t)$
has a (generalized) fractal exponential attractor $A$ (see
Definition~\rref{de7.3.2}) whose
dimension is finite in the space $H^1(\Om)\times W'$, where
 $W'$ is a completion
of   $H_0^1(\Om)$
 with respect to the norm $\|\cdot\|_{W'}= \| (1-\alpha\Delta)\cdot\|_{-2}$.
 \par
{\bf Case $\al=0$:}
Assume the validity of the hypotheses above which
guarantee the existence of a global finite-dimensional attractor.
Assume in addition that
\begin{equation}\label{exp-new-k}
 |g(s)|\le C\left( 1+sg(s)\right)^\gamma, \quad s\in\R,
\end{equation}
for some $0\le\gamma<1$. Then the system $(H_0,S_t)$ has a
(generalized) fractal exponential attractor $A$ whose dimension is
finite in the space $L_2(\Om)\times H^{-2}(\Om)$.
\end{theorem}
{\bf Main ingredients for the proofs.}
The rotational case $\alpha > 0 $ corresponds to the situation where von Karman nonlinearity is {\it compact},
hence {\it subcritical}.
This alleviates number of  technical issues when proving  validity of {\it quasi-stability inequality}
which, in turn, leads to the  existence  of smooth and finite-dimensional attractors.
Indeed, build in compactness turns critical integrals into  lower order terms.
\par
We   thus focus here on the more subtle  {\it non-rotational case}, $\alpha =0 $.
Here, the main challenge is to establish {\it quasi-stability inequality}. Once this is proved,  finite-dimensionality of attractors- Theorem \ref{th:dim-rot}, smoothness of attractor -Theorem \ref{th9.2.3} and  existence of strong and exponential  attractors
-Theorem \ref{th9.2.3str} and also Corollary \ref{th9.2.5}  follow from  the abstract tools presented in Section~\ref{sec7.9-stab}.
\par
Thus, the core of the difficulty relates to the quasi-stability estimate obtained for a difference of two solutions
$z = u-w $ defined on $Q = \Omega \times (0, T), \Sigma = \Gamma \times (0, T) $.
The solutions  under consideration are confined to a bounded invariant set
(attractor, for instance), so we can assume
$$
\|u_t(t)\|^2 +\|u(t)\|^2_{2,\Omega} \leq R^2,~~~ \|w_t(t)\|^2
+\|w(t)\|^2_{2,\Omega} \leq R^2.
 $$
 To present the main idea we assume that $P(u)\equiv 0$ in \eqref{1.2}.
 In this case
 the equation for $z$ can be written as:
\begin{equation}\label{8b.1-z.h2.n}
z_{tt}  + \Delta^2 z  =f
~~\mbox{ in }~ \Omega, \quad
z  =0,\; \Dn z  = 0 ~~\mbox{ on }~ \Gamma,
\end{equation}
where
$f=-a(g(u_t) -g(w_t)) +
[v(u),u ] -[v(w),w ]$ .
We have the following energy equality (satisfied for strong solutions)
\begin{equation}\label{10en-z.h2}
E_z(T)+D_t^T(z)
=E_z(t)+\int_t^T (\cR(z),z_t)_\Om dt,
\end{equation}
where
\[
E_z(t) = \frac12 \int_{\Om }[ | z_t|^2+ |\Delta z|^2] dx,
\quad
\cR(z) = [v(u),u ] -[v(w),w ],
\]
and
\[
D_t^T(z)=\int_t^T\int_{\Om}a [g(u_t)-g(w_t)]z_t
dx  d\tau.
\]
 Our goal is to prove the following recovery inequality:
\begin{equation}\label{goal}
E_z(T) + \int_0^T E_z(t) dt \leq C_R D_0^T(z)  + LOT(z),
\end{equation}
where $$LOT(z) \leq C_R \sup_{t\in [0, T]} \|z(t)\|^2_{2-\epsilon,\Omega}.
 $$
Inequality (\ref{goal}), via reiteration on the intervals  $[mT, (m+1) T ]$,    leads to  quasi-stability inequality.
\par
 To prove (\ref{goal}) one applies  standard multipliers $z_t $ and $z$.
 After some calculations (as for the wave equation)
one obtains \begin{multline}\label{goal1}
E_z(T) + \int_0^T E_z(t) dt \\
\leq C_R  D_0^T(z)
+ C_R LOT(z)
+C_T \int_0^T (\cR(z), z_t) dt.
\end{multline}
 The difficulty, however,  is in handling the  critical term
\begin{equation*}
\int_0^T ( \cR(z),z_t)_{\Omega}  dt.
\end{equation*}
In order to obtain quasi-stability inequality, this critical term should be represented in terms of the damping
$D_0^T(z)$, small multiples of the integrals of the energy, i.e.,
$ \epsilon \int_0^T E_z(t) dt $   and lower order terms (subcritical quantities)  in  a quadratic form.
The key to this argument is "compensated compactness structure" of the von Karman bracket
 which leads to the following  estimate:
 \begin{equation}\label{magic}
 |\int_0^T (\cR(z), z_t) d\tau| \leq C_R \max_{t\in [0,T]} \|z(t)\|^2_{2-\epsilon} +
 C_R \int_0^T (\|u_t\| + \|w_t\|) \|z\|^2_{2,\Omega}  dt
 \end{equation}
 The critical role is played by the presence  in (\ref{magic}) of the velocities $\|u_t\|$, $\|w_t\|$ which represent the damping
 and obey the estimate
 \begin{equation}\label{magic0}
 \int_0^{+\infty}( \|u_t\|^2 + \|w_t\|^2 ) dt \leq C_R
 \end{equation}
Indeed,  once (\ref{magic})  is proved and inserted into (\ref{goal1}),  then (\ref{magic0})
along with Cronwall's inequality leads to the desired quasi-stability estimate.
Thus, the key is  to be able to prove  (\ref{magic}).
 To this end  the following "compensated compactness" decomposition  of the von Karman bracket is used:
\[
 (\cR(z), z_t) = \frac14 \frac{d}{dt} Q(z) + \hf P(t)
 \]
 where
 \begin{eqnarray*}
 Q(z) &= & (v(u) + v(w) , [z,z] ) - |\Delta v(u+w,z)|^2, \nonumber\\
 P(z)& = &-(u_t, [u, v(z,z)])  -(w_t, [w, v(z,z)])  \\
 & & -(u_t + w_t, [z, v(u+w,z)]),
\end{eqnarray*}
where $v(u,w)=\Delta^{-2}\big([u,w]\big)$.
By using sharp Airy's regularity
\[
|v(u,v)|_{W^{2,\infty}} \leq C \|u\|_{2,\Omega} \|v\|_{2,\Omega}
\]
given by Lemma~\ref{l:kar-br},
one obtains the estimates
\[
|Q(z)|\leq C \|z\|_{2-\epsilon,\Omega}^2~~~\mbox{and}~~~
|P(z)| \leq C_R ( \|u_t\| + \|w_t\|) \|z\|^2_{2, \Omega}.
\]
The above yields
\begin{eqnarray*}
\int_0^T (\cR(z),z_t) dt &\leq & C(R)  [ Q(T) - Q(0)] + C(R) \int_0^T P(z) dt \nonumber \\
&\leq&  C_R  \int_0^T ( \|u_t\| + \|w_t\|) \|z\|^2_{2, \Omega} +C_R \max_{t\in [0,T]} \|z(t)\|^2_{2-\epsilon}
\end{eqnarray*}
as desired for (\ref{magic}), hence for quasi-stability inequality.

\subsubsection{Boundary damping}

We consider now the model in \eqref{1.5} with the boundary conditions in
\eqref{hing-bc}. This type of problems is often  referred to as  control
problem with  reduced number of controls.
\par
We impose the following hypotheses.
\begin{assumption}\label{as:8b.1.0.h}
\begin{itemize}
  \item Assumption \ref{as-karman}(1,2) is in force,
$a \in L_\infty(\Om)$, $a(x)\ge 0$ almost everywhere and
$ P(w)\equiv p \in L_2(\Om)$.
  \item The boundary damping $g_0\in C^1(\R)$  is  increasing,
   with the property $g_0(0) =0$ and there
 exist positive constants  $m$, $M$, and $s_0$ such that
\begin{equation}\label{gi-lb.h}
0<m \le g_0'(s) \leq M   ~~ \mbox{ for }~ |s| \geq s_0.
\end{equation}
\end{itemize}
\end{assumption}

\smallskip\par\noindent
{\bf Case $\al>0$}:
\smallskip\par\noindent
We  use  the state space
$H \equiv \big[ H^2(\Om) \cap H^1_0(\Om)\big] \times  H^1_{0}(\Omega)$.
\par
 Using energy equality  one
can easily prove the following assertion.
\begin{theorem}\label{t:10.1-grd.h}
Let  Assumption~\rref{as:8b.1.0.h} be in force and $\al>0$.
Then
\begin{itemize}
  \item
There exists $R_*>0$ such that the set
\begin{equation}\label{10.1wr-set-h}
{\cal W}_R=\left\{ y=(u_0;u_1)\in H\, :\;  \cE(u_0, u_1)\le R \right\}
\end{equation}
is a nonempty invariant set \wrt semiflow  $S_t$ generated by
equations \eqref{1.5}  with hinged b.c. \eqref{hing-bc} for every
$R\ge R_*$. Moreover, the set ${\cal W}_R$ is bounded for every
$R\ge R_*$  and any bounded set
is contained in ${\cal W}_R$ for some $R$.
\item
If in addition  $a(x)>0$ almost everywhere in $\Om$,
the system $(H,S_t)$
is gradient.
  \item The set $\cN$
of all stationary points of the semiflow $S_t$ is bounded in $H$. Thus
there exists $R_{**}>0$ such that  $\cN\subset  \cW_R$ for all $R\ge R_{**}$.
\end{itemize}
\end{theorem}
We recall that
\[
\cN=\left\{ V\in H\; :\; S_tV=V\mbox{~for all~} ~~ t\ge 0\right\}\,.
\]
and every stationary point $W$ has the form $W=(u;0)$,
where $u=u(x)\in  H^2(\Om)$  is a weak (variational) solution
to  the problem
\begin{equation}\label{1-stat.h}
\Delta^2 u   = [v(u),u ] + p~~{\rm in}~~
\Omega; ~~~
u=0, ~~ \Delta u  = 0~~{\rm on}~~
  \Gamma,
\end{equation}
with  function $v(u) $ satisfying (\ref{airy}) with $w=u$, i.e.,
 $v(u)=\Delta^{-2}\big([u,u]\big)$.

\begin{theorem}[Compact attractors]\label{t:8b.3.h}
Let $\al>0$ and Assumption \rref{as:8b.1.0.h}  be
in force. Then \index{evolutionary Karman equations with rotational
forces!boundary dissipation, clamped--hinged b.c.!global attractor}
\begin{itemize}
  \item The restriction $(\cW_R,S_t)$ of the dynamical system $(H,S_t)$
  on $\cW_R$ given by \eqref{10.1wr-set-h} has  a compact global
attractor $A_R\subset \cW_R$ for every $R\ge R_*$, where $R_*$ is the
same as in
Theorem~\rref{t:10.1-grd.h}.
  \item If,  in addition,   $a(x)>0$ almost everywhere in $\Om$,
then there exists a  compact global attractor $A$  for the system
$(H,S_t)$. Moreover, we have $A=\cM^u(\cN)$, \index{evolutionary
Karman equations with rotational forces!boundary dissipation,
clamped--hinged b.c.!structure of attractor} where $\cM^u(\cN)$ is
the unstable manifold (see the definition in Section~\rref{grad-ds})
emanating from the set $\cN$   of equilibria for the semiflow $S_t$.
\item The two  attractors $A_R$ and $A$ have
finite fractal dimension provided relation \eqref{gi-lb.h} holds for
all $s\in \R$. Moreover, in this latter  case the attractors
are bounded sets in the space $H^3(\Om)\times H^2(\Omega)$
\index{evolutionary Karman equations with rotational forces!boundary
dissipation, clamped--hinged b.c.!smoothness of elements from
attractor} and for any trajectory $(u(t);u_t(t))$ we have the
relation
\[
\|u_{tt}(t)\|_{1,\Om}+\|u_{t}(t)\|_{2,\Om}+\|u(t)\|_{3,\Om}\le C,\quad t\in\R.
\]
\end{itemize}
\end{theorem}

An important ingredient of the argument is

\begin{proposition}[Observability estimate]\label{pr:10.1-os.h}
\index{evolutionary Karman equations with rotational forces!boundary
dissipation, clamped--hinged b.c.!observability estimate} Assume
that $\al>0$ and Assumption \rref{as:8b.1.0.h} is  in
force. Let
\[
U(t) = (u(t); u_t(t)) = S_ty_1 ~~\mbox{and}~~ W(t) = (w(t);
w_t(t)) = S_ty_2
\]
 be two solutions corresponding to initial
conditions $y_1$ and  $y_2$ from the set ${\cal W}_R$ given by
\eqref{10.1wr-set-h}. Then there exist $T_0>0$ and constants
$C_1(T)$ and $C_2(R,T)$ such that
\begin{equation}\label{10.1-os.h}
TE_z(T)+\int_0^T E_z(t) dt\le C_1(T)\int_{\Sigma_1}\left|\Dn z_t
\right|^2 d\Sigma +
C_2(R,T)\cdot LOT(z)
\end{equation}
for any $T\ge T_0$,
where  $z \equiv u - w$,
\[
E_z(t) = \frac12 \int_{\Om }\left[
| z_t|^2+\alpha |\g z_t|^2  + |\Delta z|^2\right] dx,
\]
and the lower-order terms have the form
\begin{equation}\label{10.1lot-z.h}
LOT(z) =  \sup_{0\le \tau\le T }\| z(\tau)\|^2_{0, \Omega}
+ \int_0^T\|z_t(\tau)\|^2_{0, \Om} d\tau.
\end{equation}
\end{proposition}
\par
We can obtain results
on {\bf convergence rates } of individual trajectories to equilibrium under
the condition that the set $\cN$ of stationary points is discrete.
\par
As in the previous section we introduce a  concave, strictly increasing,
continuous  function
$H_0 : \R_+
\mapsto \R_+ $  which captures the behavior of $g_0(s)$ at the
origin possessing the properties
\begin{equation}\label{10.1h-0.h}
H_0(0) =0 ~~ \mbox{and} ~~ s^2 \leq H_0 ( sg_0(s))
~~ \mbox{for} ~ |s| \leq 1.
\end{equation}
We define a
function $Q(s)$ by relations (\ref{10.1q-G0})
and consider the  differential equation
\begin{equation}\label{10.1ode1.h}
\frac{d \sigma}{dt} + Q (\sigma) =0,\quad t > 0,\quad
\sigma(0) = \sigma_0 \in \R_+,
\end{equation}
which
admits the global unique  solution $\sigma(t) $ decaying asymptotically to zero
as $t \to \infty$.
\par
With these preparations  we are ready to state our result.

\begin{theorem}[{\bf Rate of stabilization}]\label{t:10.1-cnv.h}
\index{evolutionary Karman equations with rotational forces!boundary
dissipation, clamped--hinged b.c.!rate of stabilization
 to equilibrium}
Let the hypotheses of    Theorem~\rref{t:8b.3.h} be valid
with $a(x)>0$ a.e. Assume that there exist $\ga>0$ and $s_0>0$
such that the interior damping $g(s)$ satisfies the relation  $sg(s)\ge \ga s^2$ for $|s|\le s_0$.
In addition assume that  problem \eqref{1-stat.h}  has
a finite number of solutions.
  Then
for any $V\in H$ there exists a stationary point $E=(e;0)$ such that
$S_tV\to E$ as $t\to +\infty$.
Moreover, if the equilibrium $E$ is hyperbolic
in the sense that the linearization of \eqref{1-stat.h} around each of  its
solutions has the trivial solution only,
then there exist $C,T>0$ depending on  $ V,E$ such that
 the following rates of stabilization
\[
\|S_t V - E\|_{H} \leq C \sigma( [ t T^{-1}]),\quad t > 0,
\]
hold,
where  $[a]$ denotes the
integer part of $a$ and $\sigma(t) $ satisfies
\eqref{10.1ode1.h} with $\sigma_0$
 depending on $ V,E \in H$
(the constants $c_i$ in the definition of $Q$ also depend on $V$ and $E$).
In particular, if  $ g_0'(0)> 0$,  then
$$
  \|S_t V  -E \|_{H} \leq C e^{-\omega t }
$$
for some positive constants $C$ and $\omega$ depending on
$ V, E \in H$.
\end{theorem}
\medskip\par\noindent
{\bf Case $\al=0$ and $H \equiv \big[ H^2(\Om) \cap H^1_0(\Om)\big] \times  L_2(\Omega)$:}
\medskip\par\noindent
We consider  a model with
 nonlinear boundary
dissipation acting via     hinged    boundary conditions
which
does not account a for regularizing effects of
rotational inertia. Thus,  the corresponding solutions
are less regular than in the case of rotational models.
\par
The following analog of Theorem~\ref{t:10.1-grd.h} is valid.

\begin{theorem}\label{t:10.2-grd.h}
Let Assumption~\rref{as:8b.1.0.h} hold and $\al=0$. Then
\begin{itemize}
  \item  There exists $R_*>0$ such that the set
\begin{equation}\label{wr-set-h}
{\cal W}_R=\left\{ y=(u_0;u_1)\in H\, :\;  \cE(u_0, u_1)\le R \right\}
\end{equation}
is a nonempty bounded set in $H$ for all $R\ge R_*$. Moreover,
any bounded set $B\subset H$ is contained in ${\cal W}_R$ for some $R$ and
the set ${\cal W}_R$ is invariant \wrt the semiflow $S_t$.
  \item If in addition $a(x)>0$
  then
  the system $(H,S_t)$  is gradient.
\end{itemize}
\end{theorem}
Our  main goal  is to  prove the global
attractiveness property
for the dynamical system $(H,S_t)$. This
property requires additional
hypotheses imposed on the data of the problem.
\begin{assumption}\label{as.h}
 The interior damping function $g(s)$ is  globally Lipschitz,
i.e., $|g(s_1)- g(s_2)| \leq M |s_1-s_2|$
for all $s_1,s_2\in\R$.
\end{assumption}
This  global Lipschitz requirement  is due to the fact that
only one boundary condition is used as a source of dissipation.
\par

Our main  result is the following.

\begin{theorem}[Compact attractors]\label{t:main.1.h}
We suppose that  Assumptions \rref{as:8b.1.0.h} and \rref{as.h} are in force.  Then
\begin{itemize}
\item
For any $R\ge R_*$ there exists a global compact attractor $\cA_R$
for the restriction $({\cal W}_R, S_t)$ of the dynamical system
$(H, S_t)$ on ${\cal W}_R$, where ${\cal W}_R$ is given by \eqref{wr-set-h}.
\item
If we assume additionally that $a(x)>0$ a.e. in $\Om$ and
$g(s)s>0$ for all $s\neq 0$, then there is $R_0>0$ such that $\cA_R$
does not depend on $R$ for all $R\ge R_0$. In this case
$\cA\equiv\cA_{R_0}$ is a global attractor for  $(H, S_t)$ and $\cA$
coincides with the  unstable manifold $\cM^u(\cN)$ emanating from the
set $\cN$ of stationary points for $S_t$. Moreover, $\lim_{t\to
+\infty}{\rm dist}_{H}(S_tW,\cN)=0$ for any $W \in H$.
\index{evolutionary Karman equations without rotational
forces!boundary dissipation, clamped--hinged b.c.!global attractor}
\index{evolutionary Karman equations without rotational
forces!boundary dissipation, clamped--hinged b.c.!structure of
attractor}
\item
The  global attractors $\cA_R$ and $\cA$ are bounded sets in
$H^3(\Om)\times H^2(\Om)$ and have a finite fractal dimension
provided  the  relation in \eqref{gi-lb.h} holds for   all
$s\in\R$. \index{evolutionary Karman equations without rotational
forces!boundary dissipation, clamped--hinged b.c.!finite dimension
of attractor} \index{evolutionary Karman equations without
rotational forces!boundary dissipation, clamped--hinged
b.c.!smoothness of elements from attractor}
\end{itemize}
\end{theorem}
 From Theorem \ref{t:main.1.h} we  obtain the
following corollary.
\begin{corollary}\label{C:1.h} Let the hypotheses of Theorem~\rref{t:main.1.h}
be in force. Assume that  $a(x)>0$ a.e. in $\Om$
and $g(s)s>0$ for all $s\neq 0$. Then
the global attractor $\cA$ consists of  full
trajectories
$\gamma=\{ W(t)\, :\, t\in\R\}$ such that
\[
\lim_{t\to -\infty}{\rm dist}_{H}(W(t),\cN)
=0 ~~
\mbox{ and }  ~~\lim_{t\to +\infty}{\rm dist}_{H}(W(t),\cN)=0.
\]
In particular, if we assume that equation  \eqref{1-stat.h}  has
a finite number of solutions, then the global attractor $\cA$ consists of  full
   trajectories $\gamma=\{ W(t)\, :\, t\in\R\}$  connecting  pairs of
   stationary points:  any $W\in\cA$ belongs to some full
   trajectory $\gamma$ and  for any $\gamma\subset\cA$ there exists a pair
   $\{ Z, Z^*\}\subset\cN$ such that
    $W(t)\to Z \mbox{~as~} t\to -\infty$  and
    $W(t)\to Z^*$ as $t\to +\infty$.
In the latter case  for any $V\in H$ there exists a
stationary point $Z$ such that
$S_tV\to Z$ as $t\to +\infty$.
\end{corollary}

In   analogy with  the  previous cases  one can also provide
statements for the rate of stabilizations to equilibria. The
statement (and the proofs) do not depend on the specific boundary
conditions. However,
in this case we do not need to assume  additional   coercivity estimates
for the interior damping.
The point is that our basic Assumptions
\ref{as:8b.1.0.h} and \ref{as.h} are sufficient to conclude that
generalized solutions are weak (variational).
\medskip\par\noindent
{\bf Main ingredients for the proofs:}
As in the case of interior damping, the rotational case $\alpha > 0 $ is less involved. We shall thus focus
on  the non-rotational model $\alpha =0 $ where additional subtleties are present.
Similar to the interior damping case,    Theorem~\ref{th7-turk}  provides    the main tool for   establishing  asymptotic smoothness. For sake of avoiding repetitions, we shall mainly emphasize   the parts of the proof  that are more specific to the  boundary damping.
The proof is divided into several steps
which are presented
below. As before, we use  notation
$Q \equiv \Omega \times (0,T)$, $\Sigma \equiv \Gamma \times (0,T)$
and assume that $P(u)\equiv 0$ in \eqref{1.5}.
\par
We can  assume that $u$ and $w$
are strong (smooth) solutions.
Then the  difference $z \equiv u - w$ solves
the following problem
\begin{equation}\label{8b.1-z.h2}
z_{tt}  + \Delta^2 z  =f
~~\mbox{ in }~ \Omega, \quad
z  =0,\; \Delta z = \psi ~~\mbox{ on }~ \Gamma,
\end{equation}
where
$f=-a(g(u_t) -g(w_t)) +
[v(u),u ] -[v(w),w ]$ and
the boundary conditions are given by
\[
\psi= -   \left[ g_0 \left(\Dn u_t\right) -    g_0 \left(\Dn w_t\right)\right].
\]
We have the following energy equality (satisfied for strong solutions)
\begin{equation}\label{10en-z.h2+bndr}
E_z(T)+D_t^T(z)
=E_z(t)+\int_t^T ((\cR(z),z_t)) dt,
\end{equation}
where
\[
E_z(t) = \frac12 \int_{\Om }[ | z_t|^2+ |\Delta z|^2] dx,
\quad
\cR(z) = [v(u),u ] -[v(w),w ],
\]
and
\[
D_t^T(z)=\int_t^T\int_{\Om}a [g(u_t)-g(w_t)]z_t
dx  d\tau+\widetilde{D}_t^T(z),
\]
with
\begin{eqnarray*}
\widetilde{D}_t^T(z)& =&  \int_t^T\int_{\Gamma}
\left[ g_0 \left(\Dn u_t\right) -    g_0 \left(\Dn w_t\right)\right]
\Dn z_t d\Gamma d \tau.
\end{eqnarray*}
\smallskip\par\noindent
 We also use the following assertion.
\begin{lemma}\label{l:h2}
Let  $T > 0 $ be given.
Let $\phi \in C^2(\R)$ be a given function
with support in $[\delta, T-\delta]$ , where $\delta \leq  {T}/{4}$,
such that $0\le\phi\le 1$ and $\phi \equiv 1$ on $[2\delta, T-2 \delta]$.
Then any strong solution $z$ to problem \eqref{8b.1-z.h2} satisfies the
following inequality
\begin{eqnarray}\label{10.1-pos2}
\int_0^T E_z(t) \phi(t) dt
\leq C_1 \int_0^T E_z(t) |\phi'(t)| dt
+ C_2 \int_{\Sigma}\left[ |\psi|^2+
\left|\Dn  z\right|^2\right]d\Sigma \nonumber \\
+ \int_0^T \int_\Om f h \nabla z \phi  dxdt +\int_0^T \int_{\Omega} \cR(z) z_t dx  dt +
C_3 \cdot  B_T(z),\qquad {}
\end{eqnarray}
where  the constants $C_i$ do not depend on $T$
and the boundary terms\footnote{
They are not defined on the energy space.
 These are higher-order boundary  traces of solutions.
}
 $B_T(z)$ is given by
\begin{equation}\label{10.1-pos12}
B_T(z) \equiv \int_{\Sigma} \left[ \left|\left(\Dn\right)^2 z\right|^2 +
\left|\Dn \Dt z\right|^2 \right] \phi d\Sigma
+\int_0^T\left\|\Dn \Delta z  \right\|^2_{-1,\Gamma}\phi dt.
\end{equation}
\end{lemma}

On the next step we eliminate the second- and third-order boundary
traces on the boundary
in the expression \eqref{10.1-pos12} for $B_T(z)$.
These  second-order ``supercritical"  boundary traces are
due to the fact that the dissipation is allowed to affect the system via only one boundary condition.
Traces of the second and third order (see (\ref{10.1-pos12}))
 are   above the energy level, so these cannot enter the estimates.
It is the microlocal analysis argument, again, that allows for the
elimination of these terms.
Handling of the boundary terms requires delicate trace estimates   used already  in the case $\alpha>0$.
However the treatment of critical source (giving rise to the term $\cR(z)$ ) is more troublesome now.
The dissipativity integral is supported on the boundary and not in the interior.
Here are few details.
For the trace   result we use  a  more general trace
estimate proved
in \cite{l-ji} (see Proposition 1 and  Lemma 4)
valid for the  linear Kirchoff  problem (note that the estimates in \cite{l-ji}
are independent of the parameter representing rotational forces). Thus,
these estimates are applicable to both Kirchoff
and Euler--Bernoulli (\ref{8b.1-z.h2}) models.
As a consequence we have the following
\begin{lemma}\label{trace0-h}
Let $z$ be a solution to  linear  problem \eqref{8b.1-z.h2}
with given $f$ and $\psi$ and $B_T(z)$ be given by \eqref{10.1-pos12}.
Then there exist constants $ C_{T} >0 $ such that for any
$  0<\eta < \hf   $  the following estimate holds
\begin{multline*}
B_T(z)\leq
C_{1,T} \int_{\Sigma} \left( |\psi |^2 + \left| \Dn z_t\right|^2
\right) d \Sigma
 \\
 +\; C_{2,T}\left[\|z\|^2_{C([0,T]; H^{2-\eta}(\Omega))} +
\|z_t\|^2_{L_2([0,T];
H^{-\eta}(\Omega))}
+ \int_0^T \|f\|^2_{-\eta, \Omega} dt\right].
\end{multline*}
\end{lemma}
The above calculations along with the Compactness Theorem \ref{th7-turk}  imply existence of compact attractor
-first statement in Theorem \ref{t:main.1.h}. As for smoothness and finite dimensionality, one needs to prove quasi-stability estimate -which amounts to the proof of the estimate in Proposition \ref{pr:10.1-os.h}
-but with $\alpha =0$.
This means that von Karman bracket is no longer subcritical and the term $(\cR(z), z_t) $ needs to be estimated by appropriate damping on the boundary. To achieve this we proceed as in the internal case-
by decomposing the bracket into $Q(z)$ and $P(z)$. However, in the boundary case the additional difficulty is due to
the fact that dissipativity estimate (\ref{magic0}) no longer holds.
What one has, instead, is similar quantity on the boundary. In order to take advantage of this relaxed dissipation, one proceeds as in the case of wave equation by considering  (i) smoothing effect of backward trajectories on the attractor,
(ii) propagating it forward and (iii) using the compactness of the attractor. This is technical part of the argument
with full
 details  given in  \cite[Section~4.2]{cl-jde-07}.

\subsubsection{Generalizations}
\begin{enumerate}
\item
The case of nonconservative forces  can also be considered. This is to say that the the force $P(u)$ is  not potential. In that case the system has no longer gradient structure and one needs first to establish existence of an absorbing set. This can be done under some additional restrictions on the damping,  see \cite{cl-book},
 where the case ``potential operator + $(\psi,\g) w$" is considered.
\item
More general sources $P(w)$ can be considered, including nonlinear sources. However,
in general, one needs to make appropriate compactness hypotheses, see the
Remark~\ref{re:P-gen}.
\item
Dissipation in other (e.g., free)  boundary conditions, see
 \cite{cl-book,chla03,cl-jde-07}.
\end{enumerate}

\subsubsection{Open questions}
\begin{enumerate}
\item
General theory of non-conservative forces which destroy gradient structure.
\item
Localized dissipation, i.e., dissipation localized in a suitable layer near the boundary.
\item
Boundary damping problems {\it without} the  light damping.This has to do with the existence of strict Lyapunov function
and related unique continuation across the boundary. While various unique continuation results from the boundary are
available also for plates \cite{albano}, the non-local nature of the von Karman bracket prevents Carleman's estimates
 \cite{albano} from applicability.
\end{enumerate}

\subsection{Kirchhoff -Boussinesq plate}

\subsubsection{Interior damping.  Case $\al>0$:}
\index{Kirchhoff-Boussinesq model!interior damping}
In this case our main results presented in
the following  theorem (for the proofs we refer to \cite[Chapter 7]{cl-mem}).

\begin{theorem}\label{t:main.ki.1} Let  $\al>0$ and $P(w)=\Delta[ w^2]-\vr |w|^{l-1} w$ for some $\vr\ge0$ and $l\ge 1$.
Assume that the damping functions $g$ and $G$ have the structure\footnote{see Remark~\ref{re:damp-n-opt}(2) below.}
described in \eqref{dmp-str}.
Then problem
(\ref{1.3}) with $\alpha > 0$
and the clamped boundary conditions \eqref{Cl}
generate   a continuous  semiflow $S_t$ in the space
$H \equiv H^2_0(\Om) \times H^1_0(\Omega)$.
\par
Assume in addition that either
\[
\left\{ m\le 3,~~ \vr>0,~~l\le m,~~ \inf_{x\in\Om}\{a(x)\}>0 \right\}
\]
or else
\[
\left\{ m<3,~~ \vr=0,~~  \inf_{x\in\Om}\{a(x)\}>0~\mbox{is  sufficiently large} \right\}.
\]
Then
the semiflow $S_t$ possesses a global  compact attractor $A$.
Moreover,
if the constant $G_1$ in \eqref{dmp-str} is positive,
then
\begin{itemize}
\item the fractal dimension of the attractor $A$ is finite;
\item
  the system possesses a fractal exponential attractor (see Definition~\ref{de7.3.2})
whose dimension is finite in the space $H_0^1(\Om)\times W$, where $W$ is a completion
of $H_0^1(\Om)$ with respect to the norm
$\|\cdot\|_W= \| (1-\alpha\Delta)\cdot\|_{-2}$.
\end{itemize}
\end{theorem}

\begin{remark}\label{re:damp-n-opt}{\rm
(1) As we see the additional  potential term  $\rho |u|^{l-1}u$ in the force $P(w)$
 allows us to dispense with
a necessity of assuming large values for the  damping parameter.
\par
(2)
The requirements  concerning the damping terms in
Theorem ~\ref{t:main.ki.1} are not optimal. We choose them for the sake
some transparency. The same results under much more general
damping functions can be found in \cite[Chapter~7]{cl-mem}.
}
\end{remark}

\subsubsection{Interior damping. Case $\al=0$ with linear damping function:}

We   assume that $g(s)=s$ and consider the source term of the form
\begin{equation}\label{KB-source-attr}
     P(w)=\si \Delta[ w^2]-\vr |w|^{l-1} w ~~\mbox{with}~~\si,\vr\ge0.
\end{equation}
As a phase space we take $H \equiv H^2_0(\Om) \times L_2(\Omega)$.
\begin{theorem}[Compact attractor]\label{t:c}
Let
\begin{equation}\label{k-sigma}
    \si^2<\frac14 k \min\{1,k\} ~~{with}~~ k\equiv \inf_{\Om} a(x)>0.
\end{equation}
Then   the dynamical system $(H,S_t)$ generated by
equations (\ref{1.3}) with $\al=0$
possesses a  compact global attractor $\mathfrak{A}$.
\par If $\si=0$, then the system $(H,S_t)$ is gradient and thus
$\mathfrak{A}=\mathcal{M}^u(\cN)$ is unstable manifold  in $H$ emerging from
the
 set
$\cN$ of equilibria.
\end{theorem}
We note that the system under consideration is
not a gradient system when $\si\neq 0$.
As a consequence,
the proof of existence of global attractor  can
not be  just reduced to the proof of asymptotic smoothness.
One needs to establish first  existence of an absorbing ball.
For this we use Lyapunov type function of the form
\begin{equation}\label{KB-V-eps}
V_\eps(t) = \cE(t) + \eps \int_{\Omega} w(t) w_t(t) dx,
\end{equation}
with $\eps=\frac14  \min\{1,k\}$, see \cite{cl-kb,cl-kb2} for details
in the case when $a(x)\equiv k$ is a constant.
Here the energy functional has the form
\begin{multline*}
\cE(t) \equiv \frac{1}{2}  ||w_t(t)||^2 +  \frac{1}{2}  || \Delta w(t) ||^2 +
 \frac{1}{4} \int_{\Omega}  | \nabla w(x,t)|^4 dx \\ + \si
 \int_{\Omega}  w(x,t)| \nabla w(x,t)|^2 dx+\frac{\vr}{l+1}\int_{\Omega}  |w(x,t)|^{l+1} dx.
  \end{multline*}
\par

In order to complete the proof of the existence of global attractor,
hence of Theorem \ref{t:c}, we use Ball's method (see Theorem~\ref{th:ball}).
\par
Let us take $V_\eps(t)$ with $\epsilon =k/2$
and denote $\Psi(t)=V_{k/2}(t)$.
Using the fact that {\it energy
identity} holds for all  weak solutions  one
can easily  see that the following {\it equality} is satisfied for
this Lyapunov's function $ \Psi(t)$:
\[
\pd_t\Psi(t)+k \Psi(t)+\int_\Om(a(x)-k) |w_t|^2dx =K(w(t)),
\]
where
\begin{multline*}
K(w(t))
 =    \int_{\Omega}\big(\si |\nabla w|^2 -\frac{k}2((a(x)-k)w  \big) w_t dx
\\
- k
\si \int_{\Omega} w | \nabla w |^2 dx
- \frac{k}4 \int_{\Omega} |\nabla w|^4 dx
-
\frac{k\vr(l-1)}{2l+2}\int_{\Omega}  |w(x,t)|^{l+1} dx.
\end{multline*}
It is clear that the term $K(w(t))$ is subcritical with respect
to strong energy topology,
therefore using the representation
\[
\Psi(T)+\int_0^T e^{-k(T-t)}L(t) dt = e^{-kT} \Psi(0)+\int_0^T e^{-k(T-t)}K(w(t)) dt,
\]
with
\[
L(t)=\int_\Om(a(x)-k) |w_t(t)|^2dx,
\]
we can apply   Theorem \ref{th:ball}
to prove the existence of
global attractor.
\par
To establish the statement in the case $\si=0$ we note that
in this case $\cE(w,w')$ is a strict Lyapunov function
and thus the system is gradient.
\medskip\par
Now we consider long-time dynamics of strong solutions.
\par
Let $H_{st}= \big[H^4(\Omega) \cap H_0^2(\Om)\big]\times  H_0^2(\Om)$.
The argument given in the well-posedness section shows that the restriction
$S_t^{st}$ of the semiflow $S_t$ on $H_{st}$ is a continuous (nonlinear) semigroup
of continuous mappings in $H_{st}$.

\begin{theorem}[{\bf Compact
 attractor for strong solutions}] \label{pr:str-attr}
Let $\si=0$ and $a(x)\equiv k>0$ be a constant (for simplicity).  Then
semiflow $S_t^{st}$ has a compact global attractor $\mathfrak{A}^{st}$ in $H_{st}$.
This attractor $\mathfrak{A}^{st}$ possesses the properties:
\begin{itemize}
  \item $\mathfrak{A}^{st}\subseteq \mathfrak{A}$;
  \item $\mathfrak{A}^{st}=\mathcal{M}_{st}^u(\cN)$ is unstable manifold  in
$H_{st}$
  emerging from the set $\cN$ of equilibria;
  \item $\mathfrak{A}^{st}$ has a finite fractal dimension as a compact set
  in $H_{st}$.
\end{itemize}
Moreover,
if we assume in addition $l\equiv2m-1$ is an odd integer, then any trajectory $(w(t);w_t(t))$ from the attractor
$\mathfrak{A}^{st}$ possesses the property
\begin{equation}\label{sm-time-atr}
    \sup_{t\in\R} \|w^{(n)}(t)\|_{m,\Om}\le C_{n,m}<\infty,\quad n,m=0,1,2,\ldots,
\end{equation}
where $w^{(n)}(t)=\partial_t^n w(t)$.
In particular,  $\mathfrak{A}^{st}$ is a bounded set
in $C^m(\overline{\Om})\times C^m(\overline{\Om})$ for  each $m=0,1,2,\ldots$,
and thus $\mathfrak{A}^{st}\subset C^\infty(\overline{\Om})\times C^\infty(\overline{\Om})$.
\end{theorem}

The proof of Theorem \ref{pr:str-attr} consists of several steps.
The situation  corresponding to gradient flows (we deal with the case $\sigma =0$).
\smallskip\par\noindent
{\sc Step 1 (dissipativity in $H_{st}$) by the ``barrier'' method:}
Due to
dissipativity in $H$ and energy relation it is sufficient to consider
dissipativity of strong solutions $w(t)$ possessing properties
\begin{equation}\label{abs-0}
\|w_t(t)\|^2+\|w(t)\|_2^2\le R,\; t\ge 0,\quad
\int_0^\infty\|w_t(\tau)\|^2d\tau\le C_R<\infty.
\end{equation}
Let $w(t)$ be strong solution satisfying \eqref{abs-0} and $v(t)=w_t(t)$.
Consider the functional
\begin{multline}\label{G-repr-2}
G(v,v')\equiv  \frac{1}{2}\Big( \| v'\|^2+\|\Delta w\|^2+\|v\|^2
\\
+ \int_{\Omega} \left[
|\nabla w |^2 |\nabla v(x) |^2 +
2 \left|(\nabla w,\nabla v)_{\R^2}\right|^2\right] dx\Big),
\end{multline}
and denote
\begin{equation}\label{h-eps}
 H(t)\equiv H(v,v_{t})= a_0+ G(v,v_{t})+\eps(v,v_{t}),
\end{equation}
where $a_0$ and $\eps$ are positive parameters.
Then using  Br\'esis--Gallouet--Sedenko  type   inequality
(see Lemma~\ref{LEMMA 2.2} and Remark~\ref{re:sedenko})
 and choosing $\eps$ small enough, we obtain  that
\[
\frac{dH(t)}{dt}+\gamma H(t)\le C^1_{R}
\| w_t(t)\|^2  H(t)\ln\left( 1+H(t)\right) + C^2_{R}
\]
   with positive $\gamma$ (see \cite{cl-kb2} for details).
The above inequality leads to the  dissipativity property
by barrier's method.
\smallskip\par\noindent
{\sc Step 2 (attractor in $H_{st}$):}  To prove the existence of the
attractor we need to establish asymptotic compactness of the system
in $\cL$. For this we use again Ball's method.
Let $\eps=k/2$ and $a_0=0$ in \eqref{h-eps}. Then one can see that
\[
\frac{dH(t)}{dt}+k H(t) =K_{st}(w(t), w_t(t)),
\]
where
\begin{multline*}
    K_{st}(w,v)=3\int_\Om(\nabla w,\nabla v)_{\R^2}|\g v|^2 dx \\
 -\int_\Om [f'(w)v-v]v_{t} dx-\frac{k}2\int_\Om [f'(w)v-v]v dx,
\end{multline*}
where $f(s)=\vr |s|^{l-1}s$.
The term $K_{st}(w,v)$ is obviously  subcritical with respect to the
topology in $H_{st}$, therefore we can apply the same Ball's argument
 to prove the existence of global attractor
$\mathfrak{A}^{st}$. The inclusion $\mathfrak{A}^{st}\subseteq
\mathfrak{A}$ is evident. One can also see that
$\mathfrak{A}^{st}=\mathcal{M}_{st}^u(\cN)$.
\smallskip\par\noindent
{\sc Step 3  (Finite dimension of the attractor in $H_{st}$):}
To prove  finite dimensionality of the strong attractor in $H_{st}$
 we use smoothness properties of trajectories and
again the methods   developed in \cite{cl-mem}. For this we need
to establish the
quasi-stability estimate in $H_{st}$.

\begin{proposition}[quasi-stability inequality  in $H_{st}$]\label{pr:stab-st}
Let the hypotheses of Theorem~\ref{pr:str-attr} be in force.
 Let $w(t)$ and $w^*(t)$ be  strong solutions
to the problem in question satisfying the estimates
\begin{equation*}
\|w(t)\|_4+\|w_t(t)\|_2 +\|w_{tt}(t)\|+
\|w^*(t)\|_4+\|w^*_t(t)\|_2+\|w^*_{tt}(t)\|\le R
\end{equation*}
for all $t\in \R$  and for some  constant $R>0$.
Let $\tilde w(t)=w(t)-w^*(t)$. Then
\begin{multline*}
\|\lw(t)\|^2_4+\|\lw_t(t)\|^2_2 +\|\lw_{tt}(t)\|^2
 \le
C_1 e^{-\gamma t}\left( \|\lw(0)\|^2_4+\|\lw_t(0)\|^2_2\right) \\
+C_2\int_0^t e^{-\gamma(t-\tau)}\left( \|\lw(\tau)\|^2_2+\|\lw_t(\tau)\|^2\right) d\tau
\end{multline*}
for all $t>0$, where $\gamma>0$ and $C_1,C_2>0$ may depend on $R$.
\end{proposition}
The proof (see \cite{cl-kb2}) follows the strategy applied to von Karman evolution equations
and   presented in   \cite[Sect.9.5.3]{cl-book}.
\par
By Theorem~\ref{th7.9dim} (see also Theorem~4.3  in \cite{cl-mem})  the relation in Proposition~\ref{pr:stab-st}  is
sufficient to prove that the fractal dimension of $\mathfrak{A}^{st}$
as a compact set in $H_{st}$ is finite.

\subsubsection{Boundary damping}
\index{Kirchhoff-Boussinesq model!boundary damping}
In the case of
$\al>0$ we can use the same methods (see Section 10.3 in \cite{cl-book}
for the details) as for von Karman
model \eqref{1.5} with boundary damping \eqref{hing-bc}.
The case of the boundary damping  with $\al=0$ is still open.

\section{Other models covered by the methods\\ presented}\label{sect:other}

In this section we shortly describe several models
whose long-time dynamics can  be studied  within the framework
presented above. These include: (i)
nonlocal Kirchhoff wave model with strong and structural damping;
(ii) finite dimensional and smooth attractors for  thermoelastic plates;
(iii) control to finite dimensional attractors  in thermal-structure,
  flow-structure and  fluid-plate interactions.
We also consider dissipative wave models which arise in plasma physics.

\subsection{Kirchhoff wave models}
\index{Kirchhoff wave models}
In a bounded smooth domain $\Om\subset\R^d$ we consider the following
Kirchhoff wave model with a strong nonlinear damping:
\begin{equation}\label{main_eq}
    \left\{ \begin{array}{l}
        u_{tt} -\si(\|\g u\|^2)\De u_t - \phi(\|\g u\|^2)\De u+ f(u)=h(x),~~ x\in\Om,\; t>0,\\ [2mm]
u|_{\d\Om}=0,~~ u(0)=u_0, \quad u_t(0)=u_1.  \\
    \end{array} \right.
\end{equation}
Here $\De$ is the Laplace operator, $\si$ and $\phi$ are scalar
functions  specified later, $f(u)$ is a  given source term,
$h$ is a given function in $L^2(\Om)$.
\par
This kind of wave
 models goes back to G. Kirchhoff ($d=1$, $\phi(s)=\phi_0+\phi_1 s$,
$\si(s)\equiv 0$, $f(u)\equiv 0$) and  has been studied by many authors
under different sets of hypotheses (see, e.g., \cite{Lions77} and also
\cite{Chu11} and the references therein). The model in \eqref{main_eq} is
characterized by  the presence of three nonlinearities:  the source,  the damping and  the stiffness.
\par
We assume that the source nonlinearity   $f(u)$ is a $C^1$ function
possessing the  following properties: $f(0)=0$ (without loss of generality),
\begin{equation}\label{f-coercive}
\mu_f:=\liminf_{|s|\to\infty}\left\{ s^{-1}f(s)\right\}>-\infty,
\end{equation}
and also
(a)
if $d=1$, then $f$ is arbitrary; (b) if $d=2$ then
\begin{equation*}
|f'(u)|\le C\left(1+|u|^{p-1}\right) ~~\mbox{for some}~~p\ge 1;
\end{equation*}
(c) if $d\ge 3$ then either
\begin{equation}\label{f-crit}
|f'(u)|\le C\left(1+|u|^{p-1}\right)\quad\mbox{with some}~~
1\le p\le p_*\equiv \frac{d+2}{d-2},
\end{equation}
or else
\begin{equation}\label{f-supercrit}
 c_0|u|^{p-1}  -c_1 \le f'(u)\le c_2\left(1+|u|^{p-1}\right)~~
\mbox{with some}~~ p_*< p< p_{**},
\end{equation}
where $p_{**}\equiv  \frac{d+4}{(d-4)_+}$, $c_i>0$ are constants and $s_+=(s+|s|)/2$.
\par
We note that the conditions above covers subcritical, critical and supercritical cases.
\par
Under rather mild hypotheses concerning $C^1$ functions $\phi$ and $\si$
we can prove that the problem is well-posed.
This holds  even without  the requirement that   $\phi$ is non-negative
(see the details in \cite{Chu11}). However  in order to study
long-time dynamics of the problem  (\ref{main_eq})
we need to assume that
the functions  $\si$ and   $\phi$  from  $C^1(\R_+)$ are positive
and
either
\[
 \phi(s)s\to +\infty~~\mbox{as} ~~ s
\to +\infty~~~\mbox{and}~~~\mu_f=\liminf_{|s|\to\infty}\left\{ s^{-1}f(s)\right\}>0.
\]
or else
\[
\hat\mu_\phi= \liminf_{s\to+\infty} \phi(s)>0~~~\mbox{and}~~
 ~\hat\mu_\phi\la_1 +\mu_f>0,
\]
where
$\la_1$ is the first eigenvalue of  the minus Laplace operator
 in $\Om$ with  Dirichlet boundary conditions.
\par
As a phase space we consider $\cH=[H_0^1(\Om)\cap L_{p+1}(\Om)]\times L_2(\Om)$
endowed with partially strong\footnote{
It is obvious that the partially strong convergence becomes strong
below supercritical level ($H_0^1(\Om)\subset L_{p+1}(\Om)$).
}
 topology:
a sequence  $\{(u^n_0;u^n_1)\}\subset \cH$ is said to be
{\em partially strongly
convergent} to $(u_0;u_1)\in \cH$ if
$u^n_0\to u_0$ strongly in $H^1_0(\Om)$, $u^n_0\to u_0$
weakly in $L_{p+1}(\Om)$ and $u^n_1\to u_1$ strongly in $L_{2}(\Om)$
as $n\to\infty$ (in the case when $d\le 2$ we take $1<p<\infty$ arbitrary).
\par
Under the  conditions stated  above (see \cite{Chu11})
 problem \eqref{main_eq} generates an
evolution semigroup   $S_t$ in
the space $\cH$. The action of the semigroup is given   by  the formula
$S_ty=(u(t);u_t(t))$,
 where $y=(u_0;u_1)\in\cH$ and $u(t)$ is a weak solution to (\ref{main_eq}).
\par
To describe the dynamical properties  of  $S_t$
it is convenient to introduce the following notion.
\par
A bounded  set $\Ac\subset \cH$   is said to be
a {\em global partially strong attractor} for $S_t$ if
(i) $\Ac$ is closed with respect to the partially strong
 topology,
(ii) $\Ac$ is strictly invariant ($S_t\Ac=\Ac$ for all $t>0$),
and (iii) $\Ac$  uniformly  attracts in the partially strong topology
all other bounded  sets: for any (partially strong)
vicinity $\cO$ of $\Ac$ and for any bounded set
$B$ in $\cH$ there exists $t_*=t_*(\cO,B)$
such that $S_t B\subset \cO$ for all $t\ge t_*$.
\par

The main result in \cite{Chu11} reads as follows:
\begin{theorem}\label{th1:attractor}
Assume in addition  that
 $f'(s)\ge -c$ for all $s\in\R$ in
the non-supercritical case (when  (\ref{f-supercrit}) does not
hold).
Then the semigroup $S_t$ given by (\ref{main_eq}) possesses
a  global partially strong attractor $\mathfrak{A}$ in the space
$\cH$. Moreover,  $\mathfrak{A}\subset \cH_1=
[H^2(\Om)\cap H^1_0(\Om)]\times H^{1}_0(\Om)$ and
\begin{equation*}
\sup_{t\in\R}\left( \|\D u(t)\|^2
+\|\g u_t(t)\|^2+ \|u_{tt}(t)\|_{-1,\Om}^2+
\int_t^{t+1}\|u_{tt}(\t)\|^2d \t \right)\le C_{\Ac}
\end{equation*}
for
any  full trajectory $\gamma=\{ (u(t); u_t(t)): t\in \R\}$
from the attractor
$\Ac$.
We also have that
\begin{equation*}
\Ac=\mathbb{M}_+(\cN),~~~ where ~~
 \cN=\{ (u;0)\in\cH : \phi(\|\cA^{1/2}u\|^2)\cA u +f(u)=h\}.
\end{equation*}
The   attractor $\mathfrak{A}$
 has a finite fractal
dimension as a compact set in the space
$[H^{1+r}(\Om)\cap H^1_0(\Om)]\times H^{r}(\Om)$
for every $r<1$.
\end{theorem}
To prove this result we first establish   quasi-stability properties
of $S_t$ on  different topological  scales  (see \cite{Chu11} for details).
\par
 In similar way we can also study
a model with structural damping with a non-supercritical force
of the form
\begin{equation}\label{main_eq2}
    \left\{ \begin{array}{l}
        u_{tt} +\phi(\|\cA^{1/2} u\|^2)\cA u +
\si (\|\cA^{1/2} u\|^2)\cA^\theta u_t+ F(u)=0,~~  t>0,\\ [2mm]
 u(0)=u_0, \quad u_t(0)=u_1,  \\
    \end{array} \right.
\end{equation}
with $\th\in [1/2,1)$, where
  $\cA$ is a  linear
 positive self-adjoint operator with   domain $\sD(\cA)$ and with a compact resolvent
 in
 a separable infinite dimensional Hilbert space $H$.
In this case we assume that the damping  $\si$ and  the stiffness $\phi$
factors  are positive $C^1$  functions and
\[
\Phi(s)\equiv \int_0^s \phi(\xi)d\xi \to+\infty~~
\mbox{as}~~
s\to+\infty.
\]
 The nonlinear operator $F$  is
locally Lipschitz in the following  sense
\begin{equation*}
\|\cA^{-\th}[F(u_1) -F(u_2)]\|~\leq~L(\varrho) \| \cA^{1/2}(u_1-u_2)\|,\quad
\forall \|\cA^{1/2} u_i \|  \leq \varrho\,,
\end{equation*}
and potential, i.e., $ F(u)= \Pi^\prime(u)$, where $\Pi(u)$ is a
$C^1$-functional on $\sD (\cA^{1/2})$, and
${}^\prime$ stands for the Fr\'echet derivative.
We assume that $\Pi(u)$
 is  locally bounded on $\sD (\cA^{1/2})$ and there exist $\eta<1/2$
 and $C\ge 0$ such that
\begin{equation*}
 \eta \Phi(\|\cA^{1/2}u\|^2) +\Pi(u)+C \ge 0\;,\quad u\in H_1=\sD ( \cA^{1/2})\;.
\end{equation*}
This hypotheses cover critical case of the source $F$, but not supercritical.
For details concerning the model in \eqref{main_eq2}
we refer to \cite{jadea10}.

\subsection{Plate models with structural damping}

We consider a class of plate models with the strong nonlinear damping
the abstract form of which is the following Cauchy problem
in a separable Hilbert space $H$:
\begin{equation}\label{abs-1}
u_{tt} + D(u,u_t)+\cA u+F(u)=0,~~t>0;\quad u|_{t=0}=u_0,  ~~ u_t|_{t=0}=u_1.
\end{equation}
In the case of plate models with hinged boundary conditions,
$\cA=(-\Delta_D)^2$,  where
$\Delta_D$ is the Laplace operator in a bounded smooth domain
$\Om$ in $\R^2$ with the Dirichlet boundary conditions.
We have then that  $H=L_2(\Om)$ and
$$ \sD ( \cA)=\left\{ u\in H^4(\Om) \, : \;
 u=\Delta u =0 ~~{\rm on}~~\d\Om\right\}.
$$
 The damping operator $D(u,u_t)$ may have the form
\begin{equation}\label{damping}
D(u,u_t)=\Delta\left[\si_0(u)\Delta u_t\right]-
{\rm div}\,\left[\si_1(u,\nabla u)\nabla u_t\right]+g(u,u_t),
\end{equation}
where $\si_0(s_1)$, $\si_1(s_1,s_2,s_3)$ and $g(s_1,s_2)$
are locally Lipschitz functions of $s_i\in\R$, $i=1,2,3$, such that $\si_0(s_1)>0$, $\si_1(s_1,s_2,s_3)\ge 0$ and $g(s_1,s_2)s_2\ge 0$.
Also the functions  $\si_1$ and $g$
satisfy  some growth conditions.
\index{structural damping}
We note that every term in (\ref{damping})
represents a different type of damping mechanisms. The first one is the so-called viscoelastic Kelvin--Voight damping,
\index{Kelvin--Voight damping}
the second one represents
the structural damping and the term  $g(u,u_t)$ is the dynamical friction
(or viscous damping). We refer to \cite[Chapter 3]{l-tbook}
and to the references therein for a discussion of stability
properties caused by each type  of the damping terms   in the case of
linear systems.
 We should stress  that
the conditions concerning $\si_i$ above
 allow to have a pure viscoelastic damping (i.e.,
 we can take $\si_1\equiv 0$ and $g\equiv 0$).
 However a similar results remains valid
 for the case when $\si_0\equiv 0$
and $\si_1$ is independent of $\nabla u$. In this case
the presence (or absence) of the  friction $g(u,u_t)$
has also  no importance for long-time dynamics.

\par
The nonlinear feedback (elastic) force $F(u)$ may have one
of the following forms (which represent different plate models):
\begin{itemize}
  \item[(a)] {\sl Kirchhoff model}:
  \index{Kirchhoff plate model}
  $F(u)$ is the Nemytskii operator
\begin{equation*}
u\mapsto - \kappa_1\cdot {\rm div}\left\{|\nabla u|^q\nabla u
-\mu_1 |\nabla u|^r\nabla u
\right\}+ \kappa_2\left\{| u|^l u
-\mu_2 | u|^m u\right\}-p(x),
\end{equation*}
 where  $\kappa_i\ge 0$, $q>r\ge 0$, $l>m\ge 0$, $\mu_i\in \R$ are parameters,
 $p\in L_2(\Om)$.
  \item[(b)] {\sl Von Karman model:} $F(u)=-[u, \cF(u)+F_0]-p(x)$, where
 $F_0\in H^4(\Om) $ and $p\in L_2(\Om)$ are given functions,
the von Karman bracket $[u,v]$  is given by
\eqref{bracket}   and
the Airy stress function $\cF(u) $ solves  \eqref{airy}.
  \item[(c)] {\sl Berger Model:}
    \index{Berger plate model}
   In this case the feedback force has the form
  \begin{equation}\label{Berger-force}
F(u)=- \left[ \kappa \int_\Om |\nabla u|^2 dx-\Gamma\right] \De u -p(x),
 \end{equation}
where $\kappa>0$ and $\Gamma\in\R$ are parameters,  $p\in L_2(\Om)$;
 for some details and  references see, e.g.,
\cite[Chapter 4]{Chu99} and \cite[Chapter 7]{cl-mem}.
\end{itemize}
One can show that the system generated by \eqref{abs-1} is quasi-stable.
Thus methods  presented in these notes  apply.
They  allow to show the existence of  global attractors
of finite fractal dimension which in addition posses
some smoothness properties.
See \cite{kolbasin,kolbasin2} for more details.

\subsection{Mindlin-Timoshenko plates and beams}
\index{Mindlin-Timoshenko model}
Let $\Omega \subset \R^2 $ be a bounded
domain with a  sufficiently smooth  boundary $\Gamma$.
Lets
$v(x,t)=(v_1(x,t),v_2(x,t))$ be a vector function  and $w(x,t)$ be a scalar
function
on $\Om\times \R_+$. The system of Mindlin-Timoshenko equations
describes dynamics of a plate taking into account transverse shear effects
(see, e.g., \cite[Chap.1]{lagli} and the references therein). This system
has the form
\begin{equation}\label{7md1}
\alpha v_{tt} +k\cdot g(v_t)-A v +\kappa\cdot(v+\nabla w)=
-f_0(v) +\nabla_x\left[f_1(w)\right],
\end{equation}
\begin{equation}\label{7md2}
 w_{tt} +k\cdot g_0(w_t)- \kappa\cdot {\rm div} (v+\nabla w)= -f_2(w).
\end{equation}
We supplement this problem  with the Dirichlet boundary conditions
\begin{equation}\label{7md2a}
v_1(x,t)=v_2(x,t)=0,\;  w(x,t)=0\quad\mbox{on}\quad \Gamma\times \R_+.
\end{equation}
Here the functions $v_1(x,t)$ and $v_2(x,t)$ are the angles of deflection
of a filament (they are measures of transverse shear effects)
and $w(x,t)$ is the bending component (transverse displacement).
The vector
$f_0(v)=(f_{01}(v_1,v_2);f_{02}(v_1,v_2))$
and scalar $f_1$ and $f_2$   functions represents (nonlinear) feedback forces, whereas
$g(v_1, v_2)=(g_{1}(v_1),g_{2}(v_2))$ and $g_0$ are monotone damping functions
describing resistance forces (with the intensity $k>0$).
The parameter $\alpha>0$ describes rotational inertia of filaments.
The factor $\kappa>0$ is the so-called shear modulus (from mechanical point of view
the limiting situation $\kappa\to+0$ corresponds to plane strain and the case
$\kappa\to+\infty$  corresponds to absence of transverse shear).
The operator $A$ has the form
\[
A=\left[
\begin{array}{cc}
  \partial_{x_1}^2+\frac{1-\nu}2 \partial_{x_2}^2  & \frac{1+\nu}2
\partial_{x_1x_2}^2 \\
  \\
 \frac{1+\nu}2 \partial_{x_1x_2}^2& \frac{1-\nu}2 \partial_{x_1}^2+ \partial_{x_2}^2
\\
\end{array} \right],
\]
where $0<\nu<1$ is the Poisson ratio.
In 1D case (${\rm dim}\, \Om$=1) the corresponding problem models dynamics of {\it beams} under
the Mindlin-Timoshenko hypotheses. For details concerning
the Mindlin-Timoshenko hypotheses and governing equations
see, e.g. \cite{lagnese} and \cite{lagli}. We also note that in the limit $\kappa\to+\infty$ the system in \eqref{7md1} and \eqref{7md2}
becomes Kirchhoff-Boussinesq type equation, see \cite{lagnese,cl-mt}.
\par
We refer to \cite{cl-mem} and \cite{cl-mt} for an analysis
(based on quasi-stability technology) of long time
behaviour  of the Mindlin-Timoshenko plate under different  sets of
assumptions concerning nonlinear feedback forces, damping functions and
parameters. By the same method long-time dynamics in
{\em thermo-elastic}  Mindlin-Timoshenko model was studied in
\cite{fastovska07,fastovska09}.

\subsection{Thermal-structure interactions}
\index{thermal-structure interactions}
This problem has  the form
\begin{equation}\label{im-3g}
\left\{\begin{array}{c}
 u_{tt}-\alpha A\theta  +A^2u=B(u), \quad   u |_{t=0}=u_0,\; u_t |_{t=0}=u_1,
\\  \\
\theta_t + \eta A\theta
+\alpha Au_t = 0, \quad \quad  \theta |_{t=0}=\theta_0,
\end{array}\right.
\end{equation}
where  $\cA$ is a  linear positive self-adjoint operator
with a compact inverse  in
 a separable infinite dimensional Hilbert space $H$.
The nonlinear term $B(u)$ can model von Karman (as in \eqref{1.2})
and Berger (see \eqref{Berger-force}) nonlinearities.
In this case the variable $\theta$ represents  the temperature
of the plate.
\par
For application the methods presented we refer
to \cite{cl-amo08} and \cite[Chapter 11]{cl-book} in the von Karman case,
see also \cite{bu-ch0} for the Berger nonlinearity.

 \subsection{ Structural acoustic  interactions}

\index{structural acoustic  interactions}

 The mathematical model under consideration consists of a semilinear wave
equation defined on a bounded domain $\cO$, which is strongly coupled with
the Berger or von Karman plate equation acting only on a part of the boundary of
the domain $\cO$.
This kind of models, known as structural acoustic
interactions, arise in the context of modeling gas pressure in an acoustic
chamber which is surrounded by a combination of hard (rigid) and flexible
walls.
\par
More precisely, let
  $\cO\subset \R^3$
 be a bounded domain  with a sufficiently smooth
 boundary $\partial\cO$. We assume that
$\partial\cO=\overline{\Omega}\cup \overline{S}$,
 where $\Om\cap S=\emptyset$,
$$
\Om\subset\{ x=(x_1;x_2;0)\, :\,x'\equiv
(x_1;x_2)\in\R^2\}
$$ with the smooth contour $\Gamma=\partial\Om$
and $S$
 is a  surface which lies in the subspace $\R^3_- =\{ x_3\le 0\}$.
 The exterior normal on $\partial\cO$ is denoted
 by $n$. The set
$\Om$  is referred to as the elastic wall,
whose dynamics is described some plate equation.
The acoustic medium in the chamber $\cO$ is described  by a semilinear
wave equation. Thus, we consider the following (coupled) PDE system
\begin{equation}\label{pde-sys1}
\left\{
\begin{array}{l}
z_{tt}+  g(z_t)  -  \Delta z + f(z) = 0
~~ {\rm in} ~~~  \cO\times (0,T),
\\[2mm]
\displaystyle \frac{\partial z}{\partial n}  = 0
~~ {\rm on} ~~  S\times (0,T), ~~~
\frac{\partial z}{\partial n} = \alpha  \,v_t
 ~~{\rm on}  ~~\Om\times(0,T),
\end{array}
\right.
\end{equation}
and
\begin{equation}\label{pde-sys2}
\left\{
\begin{array}{l}
v_{tt}+  b(v_t)   +\Delta^2 v
+ B(v) + \beta  z_t|_{\Om} = 0
 ~~{\rm in} ~~  \Om \times(0,T),
\\[2mm]
v=\Delta v = 0 ~~ {\rm on} ~~ \partial\Om\times(0,T),
\end{array}
\right.
\end{equation}
endowed  with initial data
\[
z(0,\cdot) = z^0\,, ~~ z_t(0,\cdot) = z^1 ~~{\rm in} ~~ \Omega,~~~
v(0,\cdot) = v^0\,, ~~ v_t(0,\cdot) = v^1 ~~ {\rm in} ~~ \Om\,.
\]
Here above, $g(s)$ and $b(s)$ are non-decreasing functions describing
the dissipation effects in the model, while the term $f(z)$ represents a
nonlinear force acting on the wave component and
$B(v)$ is (nonlinear) von Karman or Berger force;
 $\alpha$ and $\beta$ are positive constants;
The part $S$ of the boundary describes a rigid (hard) wall, while
$\Om$ is a flexible wall where the coupling with the plate equation
takes place.
The boundary term $\beta  z_t|_{\Om} $ describes back pressure
exercised by the acoustic medium  on the wall.
\par
Well-posedness issues  of the  abstract second order system,
which is a particular case of the one studied in
\cite[Sect.~2.6]{cmbs} (see also \cite[Chapter 6]{cl-book}).
\par
Long-time dynamics from point of view of quasi-stable systems
were  investigated in  \cite[Chapter 12]{cl-book} in the von Karman
case and in \cite{bu-ch} for the Berger nonlinearity.
We also  note that flow structure
model \eqref{pde-sys1} and \eqref{pde-sys2}
in combination with  thermoelastic  system \eqref{im-3g}
was studied in \cite{bu-ch0} and  \cite[Chapter 12]{cl-book}.

 \subsection{Fluid-structure interactions}
\index{fluid-structure interactions}
Our mathematical model is formulated as  follows (for details, see \cite{cryzh-12}).
\par
 Let $\cO\subset \R^3$, $\Om\subset \R^2$ and the surface $S$
be the same as in the case of the model in \eqref{pde-sys1} and \eqref{pde-sys2}.
 We consider the following {\em linear} Navier--Stokes equations in $\cO$
for the fluid velocity field $v=v(x,t)=(v^1(x,t);v^2(x,t);v^3(x,t))$
and for the pressure $p(x,t)$:
\begin{equation}\label{fl.1}
   v_t-\nu\Delta v+\nabla p=G_f\quad\mbox{and}\quad  \di v=0~~ {\rm in}~~ \cO
   \times(0,+\infty),
\end{equation}
where $\nu>0$ is the dynamical viscosity and $G_f$ is a volume force.
   We supplement (\ref{fl.1})  with  the (non-slip)  boundary
   conditions imposed  on the velocity field $v=v(x,t)$:
\begin{equation}\label{fl.4}
v=0 ~~ {\rm on}~S;
\quad
v\equiv(v^1;v^2;v^3)=(0;0;u_t) ~~{\rm on} ~ \Om.
\end{equation}
Here $u=u(x,t)$ is the transversal displacement
of the plate occupying $\Om$ and satisfying
the following  equation:
\begin{equation}\label{pl_eq}
u_{tt} + \De^2 u + F(u)= p|_\Om
~~{\rm in}~~ \Omega \times (0, \infty).
\end{equation}
The nonlinear feedback (elastic) force $F(u)$ as above may have one
of the  forms (Kirchhoff, Karman or Berger) which are present
for plates models with structural damping (see \eqref{abs-1}).
We also impose clamped boundary conditions on the plate
\begin{equation}
u|_{\pd\Om}=\left.\frac{\pd u}{\pd n} \right|_{\pd\Om}=0 \label{plBC}
\end{equation}
and supply \eqref{fl.1}--\eqref{plBC} with initial data of the form
\begin{equation}
 v(0)=v_0,\quad u(0)=u_0, \quad u_t(0)=u_1, \label{IC}
\end{equation}
We  note that  \eqref{fl.1} and \eqref{fl.4} imply the following
compatibility condition
\begin{equation}\label{Com-con}
\int_\Om u(x',t) dx'=const \quad \mbox{for all}~~ t\ge 0,
\end{equation}
which can be interpreted as preservation of the volume of the fluid
(we can choose this constant to be zero).
\par
It was shown in \cite{cryzh-12}  that equations \eqref{fl.1}--\eqref{Com-con}
generates an evolution operator $S_t$ in the space
\begin{equation*}
\cH=\left\{ (v_0;u_0;u_1)\in X\times \hat{H}^2_0(\Om)\times\hat{L}_2(\Om) :\; (v_0,n)\equiv
v_0^3 =u_1
~\mbox{on}~ \Om\right\},
\end{equation*}
where
\[
X=\left\{ v=(v^1;v^2;v^3)\in [L_2(\cO)]^3\, :\; {\rm div}\, v=0;\;
 (v,n)=0~\mbox{on}~ S\right\},
 \]
and $\hat{L}_2(\Om)$ (resp. $\!\hat{H}^2_0(\Om)$) denotes the subspace in $L_2(\Om)$
(resp.\ in  $\!\hat{H}^2_0(\Om)$)
consisting of functions with zero averages.
\par
This evolution operator $S_t$ is quasi-stable and
thus possesses a compact global attractor, see \cite{cryzh-12}.
We emphasize that we do not assume any kind of mechanical damping in the plate component.
Thus this results means that dissipation of the energy in
the fluid due to viscosity is sufficient to stabilize the system.

\par
In a similar way (see \cite{Chu11-mmas}) the model
which  deals {\em only}  with
 longitudinal deformations of the plate
neglecting transversal deformations can be also considered
(in contrast with the model  \eqref{fl.1}--\eqref{IC} which takes
into account the transversal deformations only).
This means that instead of (\ref{fl.4}) the following
boundary conditions are imposed on the velocity fluid field:
\begin{equation*}
v=0 ~~ {\rm on}~S;
\quad
v\equiv(v^1;v^2;v^3)=(u^1_t; u_t^2;0) ~~{\rm on} ~ \Om,
\end{equation*}
where $u=(u^1(x,t); u^2(x,t))$ is the
in-plane displacement vector of
the plate which solves the wave equation of the form
\begin{equation*}
                        u_{tt} -\Delta u- \g
\left[ {\rm div}\, u\right]
+\nu    (v^1_{x_3};v^2_{x_3})|_{x_3=0}  +
                        f( u)=0~~{\rm in}~~\Om;~~~u^i =0~~{\rm on}~~
  \Gamma.
\end{equation*}
 This kind of models arises in
the study of blood flows in large arteries
(see the references in \cite{Ggob-jmfm08}).
\par
One can also  analyze the corresponding model based on
the full Karman shell model with rotational inertia
(see \cite{cryzh-12f}). In this case to obtain well-posedness we need to apply Sedenko's
method. The study of long-time dynamics is based on J.Ball's method
(see Theorem~\ref{th:ball}).

 \subsection{Quantum Zakharov system}
\index{Quantum Zakharov system}
In a bounded domain $\Om \subset \R^d$, $d\le 3$,
we consider the following  system
\begin{equation}\label{QZ-1}
\left\{\begin{array}{l}
n_{tt}-\Delta\left(n+|E|^2\right)+h^2\Delta^2 n+\alpha n_t  =f(x),
\quad x\in\Om,\; t>0,
\\ \\
i E_t+\Delta E- h^2\Delta^2E  +i\gamma E - n E=g(x),\quad x\in\Om,\; t>0.
\end{array}\right.
\end{equation}
Here $E(x,t)$ is a complex function
and $n(x,t)$ is a real one, $h>0$,  $\alpha\ge 0$
 and $\gamma\ge 0$ are  parameters and $f(x)$, $g(x)$ are given
 (real and complex) functions.
\par
This system in dimension  $d=1$ was
derived  in \cite{GHGO-2005}, by use of a quantum fluid approach,
to model the nonlinear interaction between quantum
Langmuir waves and quantum ion-acoustic waves in an
electron-ion dense quantum plasma.
Later a vector 3D version of equations (\ref{QZ-1}) was suggested  in \cite{HaShu-2009}.
In dimension $d=2,3$ the system in (\ref{QZ-1}) is
also known (see, e.g., \cite{SSS-2009} and the references therein)
as a simplified "scalar model"
which is in a good agreement with the vector model derived in
\cite{HaShu-2009} (see a discussion in \cite{SSS-2009}).
\par
The Dirichlet initial boundary value problem for \eqref{QZ-1} is well-posed
and generates a dynamical system in an appropriate phase
space (see \cite{Chu12}).
Using  quasi-stability methods
one can prove the existence of  a finite-dimensional
global attractor, for details we refer to \cite{Chu12}.
\par
We also note that in the case $h=0$ we arrive to the classical
Zakharov system (see \cite{Zakharov}). Global attractors in this case
were studied in \cite{Flahaut,Goubet} in the 1D case
and in \cite{ChuShch} in the 2D case. In the latter
case to obtain uniqueness the Sedenko method was applied
and Ball's method was used to prove the existence of
a global attractor.

\subsection{Schr\"{o}dinger--Boussinesq   equations}
\index{Schr\"{o}dinger--Boussinesq  models}
The methods similar to described above can be also applied in
study of qualitative behavior
of the system consisting  of  Boussinesq  and Schr\"{o}dinger
equations  coupled in a smooth (2D) bounded domain $\Omega\subset\R^2$. The resulting system takes the form:

\begin{subequations}\label{Schr-Bouss}
\begin{eqnarray}
& & w_{tt}+\gamma_1 w_t +\Delta^2w-\Delta\left(f(w)+|E|^2\right)=g_1(x), ~~
\label{Z-1ab}\\[2mm]
& &i E_t + \Delta E - w E + i\gamma_2 E  = g_2(x), ~~~ x\in \Omega, \,  t>0,
\label{Z-1as}
\end{eqnarray}
\end{subequations}
where $E(x,t)$ and  $w(x,t)$ are unknown functions,
$E(x,t)$ is  complex and  $w(x,t)$ is  real.
Here above
 $\ga_1$ and $\gamma_2$ are nonnegative parameters and $g_1(x)$ and $g_2(x)$
  are given
 (real and complex) $L_2$-functions. We equip equations  \eqref{Schr-Bouss}
with the boundary conditions
\begin{equation}
\label{bound_cond}
w|_{\partial\Omega}=\Delta w|_{\partial\Omega}=0,~~E|_{\partial\Omega}=0,
\end{equation}
and with the initial data
\begin{equation}
\label{initial-data-s}
w_t(x,0)=w_1(x),\,w(x,0)=w_0(x),~~E(x,0)=E_0(x).
\end{equation}
Long-time dynamics in this system was studied in \cite{CS-11} by the methods described
above under
 the following hypotheses concerning the (nonlinear) function
  $f$:    $f\in C^1\left(\R\right)$,
$ f(0)=0$ and
\begin{equation}
\label{cond_f_10}
\exists  c_1, c_2 \ge 0\,:~~ F(r)=\int_0^r f(\xi) d\xi\ge -c_1 r^2-c_2, \; \forall |r|\ge r_0, \end{equation}
\begin{equation}
\label{cond_f_30}
 \exists M\ge 0, p\ge 1\,:~~ |f'(s)|\le M(1+|s|^{p-1}),~~s\in\R.
\end{equation}
We also note that
the ideas which were used in the study of the Kirchfoff-Boussinesq system
(see \eqref{1.3} with $\al=0$ can be also applied
to the  following {\em Schr\"{o}dinger-Boussinesq-Kirchfoff}   model:
\begin{align*}
 & w_{tt}-\De w_t +\Delta^2w-
 {\rm div}\left\{|\nabla u|^3\nabla u
\right\}-
\Delta\left( w^2+|E|^2\right)=g_1, ~~
\\[2mm]
 &i E_t + \Delta E - w E + i\gamma_2 E  = g_2, ~~~ x\in \Omega, \,  t>0,
\end{align*}
 equipped with the boundary conditions \eqref{bound_cond} and initial data
\eqref{initial-data-s}.
For details  we refer to \cite{CS-12}.


\flushbottom
\printindex
\end{document}